\documentclass[reqno,final]{amsart}

\usepackage[T1]{fontenc}
\usepackage[utf8]{inputenc}
\usepackage{textcomp}
\usepackage{ebgaramond}

\usepackage{amsmath,amssymb,amsthm}
\usepackage{titlesec}
\usepackage{multirow}
\usepackage{caption}
\usepackage{subcaption}
\usepackage{longtable,booktabs,array,tabularx}
\usepackage{calc}
\usepackage{graphicx}
\usepackage[a4paper,top=1.75cm,bottom=1cm,left=1.3cm,right=1.3cm,head=.5cm,foot=.5cm,includeheadfoot]{geometry}
\usepackage{xcolor}
\usepackage{enumitem}
\usepackage{tikz}
\usetikzlibrary{arrows.meta,calc,positioning,fit,backgrounds,decorations.pathreplacing}

\IfFileExists{microtype.sty}{\usepackage{microtype}}{}
\IfFileExists{xurl.sty}{\usepackage{xurl}}{}
\usepackage{hyperref}
\IfFileExists{bookmark.sty}{\usepackage{bookmark}}{}

\hypersetup{
  hidelinks,
  pdftitle={The Awareness Logic of Ambiguity (ALA): Triadic Interpretive States and Their Structure-Preserving Operational Core},
  pdfauthor={Seyyed Ahmad Edalatpanah},
  pdfsubject={Awareness Logic of Ambiguity; triadic interpretive states; structure-preserving calculus},
  pdfkeywords={Awareness Logic of Ambiguity, ALA numbers, fuzzy sets, triadic interpretive states, contextual alignment, epistemic restraint, reasoning before projection},
  pdfcreator={LaTeX}
}
\titleformat{\section}
{\normalfont\fontsize{16}{24}\selectfont\bfseries}
{\thesection\textbar}{0.25em}{}
\titleformat{\subsection}
{\normalfont\fontsize{14}{24}\selectfont\bfseries}
{\thesubsection\textbar}{0.25em}{}
\titleformat{\subsubsection}
{\normalfont\fontsize{12}{18}\selectfont\bfseries}
{\thesubsubsection\textbar}{0.25em}{}
\titlespacing*{\section}{0pt}{2.5ex plus 1ex minus .2ex}{1.5ex plus .2ex}
\titlespacing*{\subsection}{0pt}{2.2ex plus 1ex minus .2ex}{1.1ex plus .2ex}
\titlespacing*{\subsubsection}{0pt}{1.8ex plus .8ex minus .2ex}{0.8ex plus .2ex}

\theoremstyle{definition}
\newtheorem{definition}{Definition}
\newtheorem{principle}{Principle}
\newtheorem{example}{Example}

\theoremstyle{plain}
\newtheorem{proposition}{Proposition}
\newtheorem{theorem}{Theorem}
\newtheorem{lemma}{Lemma}

\theoremstyle{remark}
\newtheorem{remark}{Remark}

\providecommand{\cA}{\mathcal{A}}
\providecommand{\cC}{\mathcal{C}}
\providecommand{\cP}{\mathcal{P}}
\providecommand{\cPhi}{\Phi}
\providecommand{\preceqA}{\preceq_{\mathcal A}}
\providecommand{\precA}{\prec_{\mathcal A}}
\providecommand{\parallelA}{\parallel_{\mathcal A}}

\providecommand{\meetA}{\wedge_{\mathcal A}}
\providecommand{\joinA}{\vee_{\mathcal A}}

\definecolor{Ctheta}{RGB}{31,119,180}
\definecolor{Csigma}{RGB}{44,160,44}
\definecolor{Cgamma}{RGB}{214,39,40}
\definecolor{Cbad}{RGB}{200,30,30}
\definecolor{Cok}{RGB}{30,130,70}
\definecolor{Ccube}{RGB}{55,55,55}
\definecolor{Cproj}{RGB}{0,110,170}
\definecolor{Cagg}{RGB}{0,120,120}
\definecolor{Cmeet}{RGB}{190,95,0}
\definecolor{Cjoin}{RGB}{110,40,160}
\definecolor{jfeagold}{RGB}{190,128,0}
\definecolor{jfeaabstract}{RGB}{242,242,242}

\providecommand{\Pp}[3]{({#1+0.45*#2},{#3+0.28*#2})}

\providecommand{\ALAUnitCubeFrame}{%
  \draw[Ccube,dashed,thin] \Pp{0}{1}{0} -- \Pp{0}{0}{0};
  \draw[Ccube,dashed,thin] \Pp{0}{1}{0} -- \Pp{1}{1}{0};
  \draw[Ccube,dashed,thin] \Pp{0}{1}{0} -- \Pp{0}{1}{1};
  \draw[Ccube,thin] \Pp{0}{0}{0} -- \Pp{1}{0}{0};
  \draw[Ccube,thin] \Pp{0}{0}{0} -- \Pp{0}{0}{1};
  \draw[Ccube,thin] \Pp{1}{0}{0} -- \Pp{1}{1}{0};
  \draw[Ccube,thin] \Pp{1}{0}{0} -- \Pp{1}{0}{1};
  \draw[Ccube,thin] \Pp{1}{1}{0} -- \Pp{1}{1}{1};
  \draw[Ccube,thin] \Pp{0}{0}{1} -- \Pp{1}{0}{1};
  \draw[Ccube,thin] \Pp{0}{0}{1} -- \Pp{0}{1}{1};
  \draw[Ccube,thin] \Pp{1}{0}{1} -- \Pp{1}{1}{1};
  \draw[Ccube,thin] \Pp{0}{1}{1} -- \Pp{1}{1}{1};
}

\providecommand{\ALAUnitCubeAxes}{%
  \draw[Ctheta,-{Latex[length=2mm]},thick]
       \Pp{0}{0}{0} -- \Pp{1.05}{0}{0}
       node[below right,black,font=\footnotesize]{$\theta$ valuation};
  \draw[Csigma,-{Latex[length=2mm]},thick]
       \Pp{1}{0}{0} -- \Pp{1}{1.07}{0}
       node[right,black,font=\footnotesize]{$\sigma$ contextual alignment};
  \draw[Cgamma,-{Latex[length=2mm]},thick]
       \Pp{0}{0}{0} -- \Pp{0}{0}{1.07}
       node[left,black,font=\footnotesize]{$\gamma$ epistemic restraint};
}

\begin{document}
\raggedbottom

\title[The Awareness Logic of Ambiguity (ALA)]{The Awareness Logic of Ambiguity (ALA): Triadic Interpretive States and Their Structure-Preserving Operational Core}

\author{Seyyed Ahmad Edalatpanah}
\address{Department of Applied Mathematics, Artificial Intelligence Research Center, Ayandegan University, Tonekabon, Iran}
\email{s.a.edalatpanah@aihe.ac.ir}
\email{saedalatpanah@gmail.com}
\thanks{ORCID: \url{https://orcid.org/0000-0001-9349-5695}.}
\thanks{Published in \emph{Journal of Fuzzy Extension and Applications}, 7(2) (2026), 661--702. DOI: \href{https://doi.org/10.22105/jfea.2026.582827.2294}{10.22105/jfea.2026.582827.2294}. The published article is distributed under the Creative Commons Attribution 4.0 International (CC BY 4.0) license.}
\date{}

\keywords{Awareness Logic of Ambiguity; ALA numbers; Fuzzy sets; Triadic interpretive states; Contextual alignment; Epistemic restraint; Reasoning before projection}

\begin{abstract}
Ambiguity is not a defect to be eliminated, but a profound human feature in the act of interpretation. At the boundary between interpretation and action, this feature takes a sharper form: knowing is not permission. A judgment may be well supported, contextually aligned, and held with high confidence while still being withheld from immediate action. Starting from this separation between support and permission, the paper introduces the Awareness Logic of Ambiguity (ALA) as a formal framework for preserving it before scalar projection.

An evaluative act is formalized as a triadic interpretive state preserving valuation, contextual alignment, and epistemic restraint as distinct roles of judgment. The framework does not reject scalar membership, but treats it as a delayed interface rather than the primitive form of judgment. To prevent interpretive compression, a cognitive order with reversed polarity for epistemic restraint is introduced and shown to induce a bounded distributive lattice. Weakly and strictly admissible scalar projections are developed as declared interfaces; the strict class satisfies boundary grounding and strict isotonicity and supports the non-collapse analysis. The non-collapse theorem then shows that projection-induced scalar equality may fracture under an ALA lattice operation. Thus, numerical equality does not imply operational equivalence, and the support--permission boundary is not, in general, natively preserved by scalar membership or by the surveyed membership-status and truth-status-centered representations.

The paper further establishes the conditional semantic minimality of the triad, develops a role-preserving arithmetic core, and derives canonical permissive and conservative aggregation operators that form a bounded aggregation envelope. The unit cube is identified as the normalized real-valued realization developed here, while the underlying role architecture admits, in principle, independently structured and even heterogeneous carriers for valuation, contextual alignment, and restraint. A high-stakes clinical scenario serves as an operational stress test: the same evidential support, contextual alignment, and auxiliary confidence may be associated with lower restraint toward follow-up and higher restraint toward invasive intervention. The resulting framework provides a self-contained calculus for preserving the semantic source of a judgment before it is deliberately reported through a scalar interface.
\end{abstract}

\maketitle

\section{Introduction}
\label{sec:introduction}
\phantomsection\label{subsec:knowing-is-not-permission}

A membership grade often appears as a simple number. An evaluator reports that an object belongs to a concept with grade \(0.70\), or a diagnostic system assigns a lesion a suspiciousness score of \(0.94\). Such numbers are useful because they make vague judgments computable. At the same time, they may conceal something essential: the conditions under which the judgment was formed, the context in which the value is meaningful, and the degree of restraint required before the value is allowed to guide action. Before a membership-like judgment is compressed into a scalar grade, one may ask a more fundamental question: what structure of interpretation produced this number?

This difficulty can be seen in a simple situation. Setareh tests bath water with her hand and forms the tactile judgment that the water is \emph{Comfortably Warm}. The predicate being evaluated is the perceived warmth of the water under Setareh's assessment, not the separate action-predicate that the water should be used by a particular person. For her own use, this judgment may be sufficient for immediate action. When the same water is considered for her young child, however, she pauses. The sensory evidence has not changed. The tactile assessment has not changed. The confidence she places in the judgment has not changed. What changes is the permission to operationalize the judgment.

The same structure appears in high-stakes clinical reasoning. A lesion may be strongly suspicious on imaging, the interpretive conditions may be adequate, and the clinician may be highly confident. Yet the same diagnostic judgment may warrant close monitoring, urgent follow-up, further testing, or multidisciplinary evaluation while remaining insufficient, by itself, to license immediate invasive intervention. The preserved predicate is diagnostic suspiciousness; it is neither the separate action-predicate that invasive intervention is justified nor the separate action-predicate that routine delay is safe. In such cases, the issue is not that the judgment is false, vague, unreliable, or unsupported. The issue is that evidential support and operational permission are different roles. In short, knowing is not permission to intervene, and uncertainty is not permission to delay. This is not an exceptional anomaly, but a normal condition of high-stakes reasoning under epistemic responsibility.

For more than half a century, beginning with Zadeh's seminal work on fuzzy sets, graded membership has provided a powerful mathematical language for representing vagueness and approximate reasoning \cite{Zadeh1965,Zadeh1975}. Subsequent developments have enriched this language in several directions: type-2 and interval-valued models address uncertainty in membership grades \cite{MendelJohn2002,Bustince2016}; intuitionistic, orthopair, hesitant, neutrosophic, and related extensions introduce additional membership-centered or truth-status components \cite{Atanassov1986,Yager2014,Yager2017,Cuong2014, KutluKahraman2019, Torra2010,Smarandache1998}; Z-number frameworks represent reliability-oriented qualifications \cite{Zadeh2011,Yager2012}; and linguistic, rough, granular, and three-way decision frameworks address perception-based reasoning, approximation, granulation, and decision-region trisection \cite{Zadeh1996,Zadeh1997,Pawlak1982,Yao2010,Pedrycz2013}. Yet, as a scalar interface, graded membership answers only half the question. It tells us how strongly a concept is supported, but not what interpretive conditions produced that support, nor how cautiously the resulting judgment should be allowed to guide action.

This paper proceeds in continuity with fuzzy set theory and 
its extensions, not in rejection of them. Scalar membership 
is a powerful and often indispensable interface for comparison, 
classification, ranking, and decision interfacing. The difficulty 
appears when the scalar interface is treated as the primitive 
form of judgment itself, since a value that is operationally 
convenient is not, by that fact alone, a complete representation 
of the judgment that produced it. The value is visible; the 
interpretive source of the value is not. This loss is called 
interpretive compression: the reduction of a typed evaluative 
act to a single numerical surface before the roles that gave 
the act its meaning have participated in reasoning. It may also 
obscure the traceability of judgment, since a scalar grade may 
conceal the evaluator's perspective, the adequacy of the setting, 
and the restraint under which the judgment should be used.

What is needed is not merely another extension of fuzzy membership calculus, but a fundamental rethinking of what belonging means in a world saturated with ambiguity. Ambiguity, on this view, is not a defect to be eliminated, but a profound human feature of the act of interpretation. The Awareness Logic of Ambiguity (ALA) is introduced at this point: not as another numerical refinement of membership, but as a structure-preserving framework for reasoning about valuation, contextual alignment, and epistemic restraint before scalar projection. ALA preserves not only the reported value, but also the conditions under which the value was produced.

Here, ``awareness'' does not refer to awareness operators in epistemic logic \cite{FaginHalpern1988,FaginHalpernMosesVardi1995}. It refers to preserving awareness of the interpretive structure that scalar membership may otherwise compress. The term ``logic'' is used in the sense of a formal interpretive system equipped with structured states, cognitive ordering, semantic operations, and admissible scalar interfaces. Ambiguity is understood as a structured interpretive condition, not merely as numerical imprecision.

Before projection, an ALA judgment is represented as
\[
A=(\theta,\sigma,\gamma)\in[0,1]^3,
\]
where \(\theta\) is interpretive valuation, \(\sigma\) is contextual alignment, and \(\gamma\) is epistemic restraint. The first coordinate records the content-directed strength of the judgment. The second records the degree to which the judgment is aligned with the context in which it is made. The third records the degree to which the evaluator resists premature operationalization of the judgment. The polarity of \(\gamma\) is intentionally reversed: larger restraint lowers cognitive permissiveness.

Returning to Setareh, for analytic clarity, two minimal contrasts may be separated. In the first contrast, the same tactile judgment is used under two operational interfaces:
\[
A_{\mathrm{adult}}=(0.70,0.90,0.15),
\qquad
A_{\mathrm{child}}=(0.70,0.90,0.80).
\]
The valuation and contextual alignment are held fixed; only epistemic restraint differs. In the second contrast, the same tactile judgment is formed under two contextual settings, first in Setareh's familiar home bathroom and then in an unfamiliar hotel during family travel:
\[
A_{\mathrm{home}}=(0.70,0.90,0.15),
\qquad
A_{\mathrm{hotel}}=(0.70,0.35,0.15).
\]
Here valuation and restraint are held fixed, while contextual alignment differs. In real situations, contextual weakening and increased restraint may occur together; the two contrasts are separated here only to expose the distinct roles of \(\sigma\) and \(\gamma\). These minimal contrasts show how the same reported scalar membership grade may hide distinct interpretive sources: a difference in operational permission, represented by \(\gamma\), or a difference in contextual alignment, represented by \(\sigma\). 

\begin{figure}[!htbp]
\centering
\resizebox{0.94\textwidth}{!}{%
\begin{tikzpicture}[font=\scriptsize,>=Stealth]

\begin{scope}[shift={(0,0)}]
\draw[gray!35, fill=gray!4, rounded corners] (0,0) rectangle (6.8,5.15);

\node[font=\footnotesize\bfseries, align=center] at (3.4,4.82)
{Panel A};

\node[font=\scriptsize, align=center] at (3.4,4.52)
{\(\theta,\sigma\) fixed; \(\gamma\) varies};

\node[font=\footnotesize\bfseries] at (2.0,4.02) {Own use};

\node[font=\tiny] at (1.45,3.68) {$\theta$};
\draw[gray!55] (1.25,2.05) rectangle (1.65,3.55);
\fill[blue!38] (1.25,2.05) rectangle (1.65,{2.05+1.50*0.70});
\node[font=\tiny] at (1.45,1.78) {0.70};

\node[font=\tiny] at (2.00,3.68) {$\sigma$};
\draw[gray!55] (1.80,2.05) rectangle (2.20,3.55);
\fill[green!38] (1.80,2.05) rectangle (2.20,{2.05+1.50*0.90});
\node[font=\tiny] at (2.00,1.78) {0.90};

\node[font=\tiny] at (2.55,3.68) {$\gamma$};
\draw[gray!55] (2.35,2.05) rectangle (2.75,3.55);
\fill[red!35] (2.35,2.05) rectangle (2.75,{2.05+1.50*0.15});
\node[font=\tiny] at (2.55,1.78) {0.15};

\node[font=\footnotesize\bfseries] at (4.85,4.02) {Child care};

\node[font=\tiny] at (4.30,3.68) {$\theta$};
\draw[gray!55] (4.10,2.05) rectangle (4.50,3.55);
\fill[blue!38] (4.10,2.05) rectangle (4.50,{2.05+1.50*0.70});
\node[font=\tiny] at (4.30,1.78) {0.70};

\node[font=\tiny] at (4.85,3.68) {$\sigma$};
\draw[gray!55] (4.65,2.05) rectangle (5.05,3.55);
\fill[green!38] (4.65,2.05) rectangle (5.05,{2.05+1.50*0.90});
\node[font=\tiny] at (4.85,1.78) {0.90};

\node[font=\tiny] at (5.40,3.68) {$\gamma$};
\draw[gray!55] (5.20,2.05) rectangle (5.60,3.55);
\fill[red!35] (5.20,2.05) rectangle (5.60,{2.05+1.50*0.80});
\node[font=\tiny] at (5.40,1.78) {0.80};

\draw[->,thick] (1.25,1.02) -- (5.60,1.02);
\foreach \x/\lab in {1.25/0,2.55/0.3,4.20/0.7,5.60/1} {
  \draw (\x,0.92)--(\x,1.12);
  \node[font=\tiny] at (\x,0.65) {\lab};
}
\fill[orange!85] (4.20,1.02) circle (2.2pt);
\node[font=\tiny] at (4.20,0.35) {$\mu=0.70$};

\end{scope}

\begin{scope}[shift={(7.35,0)}]
\draw[gray!35, fill=gray!4, rounded corners] (0,0) rectangle (6.8,5.15);

\node[font=\footnotesize\bfseries, align=center] at (3.4,4.82)
{Panel B};

\node[font=\scriptsize, align=center] at (3.4,4.52)
{\(\theta,\gamma\) fixed; \(\sigma\) varies};

\node[font=\footnotesize\bfseries] at (2.0,4.02) {Home};

\node[font=\tiny] at (1.45,3.68) {$\theta$};
\draw[gray!55] (1.25,2.05) rectangle (1.65,3.55);
\fill[blue!38] (1.25,2.05) rectangle (1.65,{2.05+1.50*0.70});
\node[font=\tiny] at (1.45,1.78) {0.70};

\node[font=\tiny] at (2.00,3.68) {$\sigma$};
\draw[gray!55] (1.80,2.05) rectangle (2.20,3.55);
\fill[green!38] (1.80,2.05) rectangle (2.20,{2.05+1.50*0.90});
\node[font=\tiny] at (2.00,1.78) {0.90};

\node[font=\tiny] at (2.55,3.68) {$\gamma$};
\draw[gray!55] (2.35,2.05) rectangle (2.75,3.55);
\fill[red!35] (2.35,2.05) rectangle (2.75,{2.05+1.50*0.15});
\node[font=\tiny] at (2.55,1.78) {0.15};

\node[font=\footnotesize\bfseries] at (4.85,4.02) {Hotel};

\node[font=\tiny] at (4.30,3.68) {$\theta$};
\draw[gray!55] (4.10,2.05) rectangle (4.50,3.55);
\fill[blue!38] (4.10,2.05) rectangle (4.50,{2.05+1.50*0.70});
\node[font=\tiny] at (4.30,1.78) {0.70};

\node[font=\tiny] at (4.85,3.68) {$\sigma$};
\draw[gray!55] (4.65,2.05) rectangle (5.05,3.55);
\fill[green!38] (4.65,2.05) rectangle (5.05,{2.05+1.50*0.35});
\node[font=\tiny] at (4.85,1.78) {0.35};

\node[font=\tiny] at (5.40,3.68) {$\gamma$};
\draw[gray!55] (5.20,2.05) rectangle (5.60,3.55);
\fill[red!35] (5.20,2.05) rectangle (5.60,{2.05+1.50*0.15});
\node[font=\tiny] at (5.40,1.78) {0.15};

\draw[->,thick] (1.25,1.02) -- (5.60,1.02);
\foreach \x/\lab in {1.25/0,2.55/0.3,4.20/0.7,5.60/1} {
  \draw (\x,0.92)--(\x,1.12);
  \node[font=\tiny] at (\x,0.65) {\lab};
}
\fill[orange!85] (4.20,1.02) circle (2.2pt);
\node[font=\tiny] at (4.20,0.35) {$\mu=0.70$};

\end{scope}

\end{tikzpicture}%
}
\caption{\textbf{Two forms of scalar compression under a scalar readout.}
Panel A isolates variation in epistemic restraint, while Panel B isolates
variation in contextual alignment. In both cases, the same scalar value
hides a role-level distinction preserved by the ALA state. The displayed
scalar is the naive valuation readout \(\mu=\theta\), which Section~5 shows
is not an admissible projection of the full state.}
\label{fig:scalar-compression}
\end{figure}
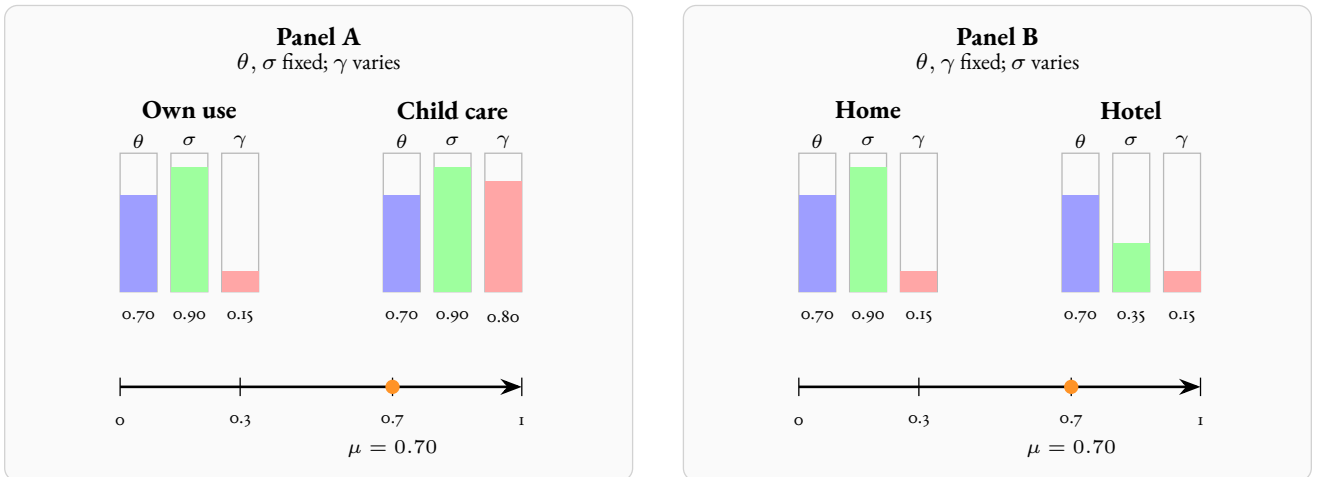

Fig.~\ref{fig:scalar-compression} separates two elementary forms of interpretive compression. In Panel A, the scalar readout remains fixed while epistemic restraint changes. In Panel B, the same scalar readout is preserved while contextual alignment changes. The two panels are separated only for analytic clarity.


The technical contribution of this paper is built around one aim: to 
keep the semantic roles of judgment available until the point at 
which a scalar interface is deliberately imposed. At the 
representational level, the paper defines ALA-valued interpretive judgment assessment, introduces a cognitive order with reversed restraint 
polarity, and proves that the resulting structure is a bounded 
distributive lattice. At the operational level, it develops 
admissible scalar projections as delayed interfaces, establishes 
the non-collapse theorem, justifies the semantic necessity of the 
triad, constructs a role-preserving arithmetic core, and derives 
an aggregation envelope from permissive and conservative aggregation 
operators.
The thesis is therefore not that a better scalar should be reported, but that the typed roles producing the judgment, namely valuation, contextual alignment, and epistemic restraint, must remain visible before projection compresses them into a scalar interface, and in particular, before support is mistaken for permission.

The remainder of the paper is organized as follows. Section~2 positions ALA within the related literature and clarifies its modeling distinction from representative fuzzy, uncertainty-based, linguistic, and decision-theoretic frameworks. Section~3 defines the normalized unit-cube realization, the operational semantics of its coordinates, the carrier-independent role architecture that it instantiates, and the admissible elicitation discipline that protects the separation between semantic roles and measurement routes. Section~4 develops the cognitive order, lattice structure, and role-preserving complement. Section~5 introduces rolewise standardization and scalar projection as distinct delayed interfaces. Section~6 establishes projection equivalence and the non-collapse result. Section~7 justifies the structural necessity of the triad while separating semantic triadicity from the internal complexity of role carriers. Sections~8 and~9 develop the role-preserving arithmetic and aggregation layers of the standard realization and state the role-local principles governing heterogeneous extensions. Section~10 presents a high-stakes clinical stress test showing how the same evidential support, contextual alignment, and auxiliary confidence may lead to different operational commitment interfaces. Sections~11 and~12 discuss the broader implications of the framework and conclude the paper. Additional diagnostic details on role-separated elicitation, projection behavior, and coordinate deletion are collected in Appendix~A.


\section{Related Work and Motivation}
\label{sec:related-work}
\label{sec:context-not-enough}
\label{subsec:common-core}
\label{subsec:carrier-honesty}
\label{subsec:per-framework}

Section~\ref{sec:introduction} positioned ALA within the broader development of fuzzy set theory and its extensions. This section states what ALA contributes and where each neighboring framework absorbs a role that ALA preserves. The organizing principle is \emph{non-absorptive role preservation}: an evaluative judgment is represented so that valuation, contextual alignment, and epistemic restraint remain separately visible, none of them silently absorbed into another and none collapsed into a single grade at the moment the assessment is formed. Scalar, linguistic, or decision outputs remain useful, but as interfaces applied to a preserved state rather than as substitutes for it. Every comparison below is therefore read through one diagnostic question: which role is absorbed, and where?

Existing frameworks have enriched membership grades, qualified information by reliability, represented indeterminacy, approximated concepts, parameterized descriptions, granulated information, and partitioned decision spaces into action regions. ALA addresses a different representational target: the internal structure of the evaluative act itself, with the third role carrying reversed polarity, so that restraint enters the order and the admissible scalar interfaces through \(1-\gamma\) rather than as an ordinary positive contribution.

A single state makes the distinction visible:
\[
A=(0.94,\,0.90,\,0.95).
\]
In ALA this state is not contradictory. The judgment has strong predicate-directed valuation and strong contextual alignment, while also carrying strong restraint against operational commitment. The proposition is not less supported, more false, or more indeterminate; it is more restrained. Encoding this situation by lowering membership, increasing indeterminacy, increasing falsity, or adding a generic reliability qualifier would change the meaning of the third role. That is precisely the absorption that ALA refuses. The same diagnostic point appears in the three states
\[
A_1=(0.35,\,0.90,\,0.20),\qquad
A_2=(0.90,\,0.35,\,0.20),\qquad
A_3=(0.90,\,0.90,\,0.95).
\]
All three may be compressed into a cautious or borderline scalar report, yet their sources differ: weak valuation in \(A_1\), weak contextual alignment in \(A_2\), and strong epistemic restraint despite strong support and alignment in \(A_3\). ALA keeps these diagnoses distinct before any scalar interface is applied.

The claim does not rest on the mere use of a three-dimensional carrier. As an ordered set, the state space is the cube \([0,1]^3\) with a role-dependent orientation, and many operations act coordinatewise. What is contributed is semantic and operational: non-interchangeable roles, reversed restraint polarity, role-preserving operations, and delayed projection. Coordinatewise operations are the operational form of role preservation, not a device for hiding semantic interaction.

The most immediate comparison is with context-dependent and context-adapted fuzzy membership, in which a contextual condition \(c\) calibrates, transforms, or selects a scalar grade through nonlinear transformations \cite{PedryczGudwinGomide1997}, evolutionary optimization \cite{GudwinGomidePedrycz1998,Botta2009}, dynamic parameter adjustment \cite{Magdalena2002,HoGaribaldi2014}, or information granulation \cite{Zadeh1997,Yao2001}. Here \(\sigma\) is neither the context variable \(c\) nor a situational label: a context descriptor records where a judgment was formed, whereas contextual alignment evaluates whether those conditions are adequate for the valuation being formed. Storing \(c\) as metadata beside the grade leaves it external to the algebra of judgment; in ALA, \(\sigma\) is an endogenous coordinate that participates in the cognitive order, the lattice, projection, arithmetic, and aggregation. A framework that genuinely reproduces this behavior has already adopted the central representational commitment of ALA rather than replaced it.

From an order-theoretic viewpoint, ALA lies near \(L\)-fuzzy and lattice-valued membership \cite{Goguen1967}. This proximity concerns representability, not semantic reduction. If the three role carriers are lattices, the typed product \(L_\theta\times L_\sigma\times L_\gamma^{\mathrm{op}}\) is itself a lattice and may therefore be encoded as an \(L\)-valued carrier. An arbitrary \(L\)-valued representation, however, does not by itself preserve the constitutive factorization into valuation, contextual alignment, and epistemic restraint, the reversed orientation of restraint, the role-local calculi, or the discipline of delayed projection. ALA may admit an \(L\)-valued realization, but it is not semantically reducible to generic \(L\)-fuzzy membership: the non-erasable contribution is the role typing of the carrier and its operations, not carrier richness alone.

Multi-component fuzzy extensions raise the broadest reduction risk. Intuitionistic \cite{Atanassov1986,DuboisEtAl2005}, orthopair \cite{Yager2014,Yager2017}, picture \cite{Cuong2014}, spherical \cite{KutluKahraman2019}, hesitant \cite{Torra2010}, neutrosophic \cite{Smarandache1998}, and related frameworks add membership-centered, hesitant, neutral, refusal, or truth-status components, all of which remain tied to the status of membership or truth. The decisive distinction is the polarity signature of \(\gamma\): increasing it does not increase falsity, indeterminacy, hesitation, refusal, or non-membership; it decreases operational permissiveness, which is why high support and high restraint coexist without contradiction.

Fuzzy multisets \cite{Yager1986Bags}, shadowed sets \cite{Pedrycz1998Shadowed}, and D-numbers \cite{Deng2012DNumbers} expand multiplicity structures, approximation regions, and evidence representations along axes different from the target of ALA. Plithogenic sets \cite{Smarandache2017Plithogenic} come closer by combining attribute-based appurtenance with contradiction-guided aggregation, but the contradiction degree steers the aggregation of attribute values; it does not preserve the typed roles of a single evaluative act before projection.

A more delicate comparison arises with Z-numbers. A Z-number is commonly understood as an ordered pair \(Z=(A,B)\), where \(A\) is a fuzzy restriction and \(B\) expresses the reliability, certainty, or sureness attached to that restriction \cite{Zadeh2011,Yager2012}. Reliability qualifies asserted information, whereas contextual alignment qualifies the fit between a judgment and the interpretive setting in which it is formed and assessed. A reliable source may judge under poorly aligned conditions, and a reliable restriction issued out of domain is not the same as a well-aligned judgment. If contextual adequacy is encoded by lowering the restriction, it is absorbed into valuation; if it is encoded by lowering reliability, it is reassigned to source or evidential trustworthiness. In neither case is \(\sigma\) natively preserved as a coordinate of the state. The reversed-polarity coordinate \(\gamma\) is equally distinct from \(B\), and along a further axis. Reliability is backward-looking, grading the evidence a judgment already rests on, while epistemic restraint is forward-looking, grading the caution required before the judgment is operationalized under a given commitment interface. This is why a Z-number cannot represent an ALA state by relabeling. A judgment may carry high reliability and high restraint at once, resting on excellent evidence yet retaining strong reason to withhold action, a combination that is contradictory only if support and permission are conflated.

Three-way decision trisects a decision space into acceptance, rejection, and deferment regions through thresholds, losses, or sequential information acquisition \cite{Yao2010,YaoTWD2012,YaoIJAR2024}. The ALA triad is different in kind: its coordinates are co-present roles inside one evaluative state, not mutually exclusive outcomes, and any assert, reject, or withhold interface is downstream of the preserved state. Three-way decision identifies noncommitment as a region; ALA diagnoses its semantic source, which may be weak valuation, weak contextual alignment, strong epistemic restraint, or projection instability.

Computing with Words treats words, linguistic variables, and perceptions as computational objects \cite{Zadeh1975,Zadeh1996}. ALA shares the human-centered motivation but operates one layer below: a linguistic judgment such as ``highly suspicious'' may be meaningful and computationally well formed while remaining insufficient for operational commitment. Hedges modify meaning; restraint governs permission. A word is one reporting interface for a preserved evaluative state.

Finally, since ALA states have three coordinates and projections may use weights, the triad may be misread as three ordinary criteria in a weighted scoring model, as in multiple-attribute decision making (MADM) \cite{HwangYoon1981,Triantaphyllou2000}. MADM criteria are analyst-selected, problem-dependent, and fungible; the ALA coordinates are constitutive roles whose meanings and polarities are fixed by the framework. A MADM-like scalar score can therefore serve only as a downstream interface, after the role-typed state has been preserved.

Table~\ref{tab:ala_positioning} summarizes these comparisons as a map of representational targets, not a hierarchy of theories: the listed frameworks remain indispensable in their own domains, and the axis at issue is the non-absorptive preservation of judgment roles before projection.

\begin{table}[!htbp]
\centering
\caption{Non-absorptive role positioning relative to adjacent frameworks.}
\label{tab:ala_positioning}
\fontsize{10}{12}\selectfont
\setlength{\tabcolsep}{2.4pt}
\renewcommand{\arraystretch}{1.16}
\begin{tabularx}{\textwidth}{@{}>{\raggedright\arraybackslash}p{0.17\textwidth}
>{\raggedright\arraybackslash}p{0.22\textwidth}
>{\raggedright\arraybackslash}p{0.24\textwidth}
>{\raggedright\arraybackslash}X@{}}
\toprule
\textbf{Framework} &
\textbf{Native target} &
\textbf{Absorption risk} &
\textbf{Role-preservation distinction} \\
\midrule
Context-adapted fuzzy membership & Context-sensitive grade \(\mu_A(x;c)\) & Contextual weakness may be absorbed into the reported scalar & \(\sigma\) is an endogenous, order-bearing alignment coordinate, not the context variable, metadata, or a context-adjusted grade. \\
\(L\)-fuzzy and lattice-valued membership & Ordered carrier for membership values & Rich carriers need not impose typed judgment roles & Semantic typing is fixed by valuation, contextual alignment, and reversed restraint polarity, not by carrier richness alone. \\
Multi-component fuzzy extensions & Membership, non-membership, hesitation, neutrality, refusal, indeterminacy, or truth-status components & Judgment roles may be recoded as membership-status or truth-status components & The coordinates are heterogeneous judgment roles; \(\gamma\) has inverse operational polarity rather than hesitation, indeterminacy, falsity, refusal, or non-membership. \\
Carrier- and aggregation-enriching extensions & Multiplicity, shadow regions, incomplete evidence, or attribute-driven aggregation & Role structure may be replaced by a richer carrier or aggregation mechanism & Valuation, contextual alignment, and restraint are preserved before projection or aggregation. \\
Z-numbers & Fuzzy restriction plus reliability qualification & Contextual alignment may be absorbed into restriction or reliability & Reliability is separated from contextual alignment, and restraint is separated from reliability. \\
Three-way decision & Acceptance, rejection, and deferment regions & The source of noncommitment may be absorbed into a decision region & The evaluative state is decomposed before any decision-region interface is imposed. \\
Computing with Words & Words, linguistic variables, and perceptions as computational objects & Judgment structure may be absorbed into a linguistic label & Words are reporting interfaces; the preserved state remains behind the linguistic label. \\
Decision-scoring interfaces & Criteria or components aggregated for ranking, choice, or reporting & Semantic roles may be treated as fungible weighted criteria & Roles are constitutive and polarity-fixed; projection is delayed and declared. \\
\bottomrule
\end{tabularx}
\end{table}

Three adjacent traditions deserve brief mention, since their vocabulary overlaps with that of ALA while their targets differ. Possibility theory pairs an event with possibility and necessity degrees \cite{DuboisPrade1988}; both degrees qualify the epistemic standing of a proposition, whereas epistemic restraint qualifies the permission to operationalize a judgment whose epistemic standing may already be excellent. Bipolar fuzzy sets attach a negative pole to membership \cite{ZhangBipolar1994}; the reversed polarity of \(\gamma\) is not a negative valuation but a brake on commitment, and it coexists with high positive valuation. Finally, because ALA speaks throughout of permission, deontic logic should be named explicitly: deontic operators formalize normative permission and obligation over propositions \cite{vonWright1951}, whereas \(\gamma\) is a graded coordinate internal to a single evaluative state. An ALA state may inform a deontic conclusion, but it is not itself a deontic formula.

The distinction has a diagnostic reading. A noisy medical image is a problem of the interpretive setting; a lesion weakly compatible with the diagnostic predicate is a problem of valuation; a strongly suspicious lesion under adequate imaging, for which immediate invasive action is premature, is a matter of epistemic restraint. One scalar grade erases this diagnostic direction, and with it the audit trail of judgment formation: the same value may hide different reasons for caution, different sources of weakness, and different routes for revision. What ALA provides is a structure in which the evaluator exposes the alignment and restraint under which a reported degree is used. Scalar projection is accordingly treated as a delayed interface, developed in Section~5 and stress-tested by the non-collapse result of Section~6.

ALA does not reject the traditions compared above; it claims that a membership-like judgment should preserve the semantic source of its value before that value is reduced to a scalar, linguistic, or decision interface. Its novelty lies in non-absorptive role preservation: the semantic typing of the coordinates, the internalization of contextual alignment, the reversed operational polarity of epistemic restraint, and the delayed-projection principle. This is the structural difference between a context-adjusted scalar grade and a triadic interpretive state.


\section{ALA State Space and Operational Semantics}
\label{sec:ala-state-space}

This section gives the formal state space and operational semantics of ALA. Section~\ref{sec:context-not-enough} showed why context dependence, by itself, does not solve interpretive compression. The present section specifies what is preserved when a membership-like judgment is represented as an ALA state. The central point is that the three coordinates of an ALA state are typed rather than interchangeable. It is a typed interpretive state whose coordinates play different roles before any scalar projection is applied.

\subsection{ALA-valued Interpretive Judgment Assessment}
Let \(X\) be a universe of discourse, let \(\cC\) denote a set of admissible contexts, and let \(\cP\) denote a set of admissible perspectives. A context \(c\in\cC\) records the situational, evidential, linguistic, or institutional conditions in which an evaluation is formed. A perspective \(p\in\cP\) records the declared interpretive purpose, assessment protocol, disciplinary lens, or intended use under which the evaluation is issued. These inputs are not themselves ALA coordinates. Their influence is role typed: contextual adequacy belongs to \(\sigma\), whereas stakes, irreversibility, and the operational commitment interface belong to \(\gamma\). A change in \(p\) that changes only the intended commitment interface should therefore alter \(\gamma\), not be silently absorbed into \(\theta\) or \(\sigma\).

\begin{definition}[ALA-valued interpretive judgment assessment]
\label{def:ala-valued}
Let \(A\) be a concept over \(X\). An ALA-valued interpretive judgment assessment associated with \(A\) is a mapping
\[
\mu_A^{\mathrm{ALA}}:X\times\cC\times\cP\to\cA,
\qquad
\cA=[0,1]^3,
\]
of the form
\begin{equation}
\mu_A^{\mathrm{ALA}}(x,c,p)
=
\bigl(\theta_A(x,c,p),\sigma_A(x,c,p),\gamma_A(x,c,p)\bigr),
\label{eq:ala-valued-output}
\end{equation}
where
\[
\theta_A,\sigma_A,\gamma_A:X\times\cC\times\cP\to[0,1].
\]
The three component maps are \emph{role-transparent}. The coordinate \(\theta_A\) returns the interpretive valuation of \(x\) with respect to the concept \(A\); the coordinate \(\sigma_A\) returns the contextual alignment under which the valuation is formed and assessed; and the coordinate \(\gamma_A\) returns the epistemic restraint associated with carrying the judgment toward operational use. The general dependence of the three coordinates on \((x,c,p)\) does not make them interchangeable numerical attributes. Their identity is fixed by their role in the judgment: valuation, contextual alignment, and epistemic restraint. The output \(\mu_A^{\mathrm{ALA}}(x,c,p)\) is called an \emph{ALA state} associated with \(x\) under context \(c\) and perspective \(p\).
\end{definition}

\begin{remark}[Membership readouts and interpretive judgment states]
\label{rem:membership-readout-judgment-state}
An ALA state should not be read as a membership triple in the sense of a
multi-component membership-status representation. The coordinate
\(\theta_A\) is the valuation coordinate closest to classical scalar
membership, but the full ALA state is not exhausted by \(\theta_A\). Scalar
membership grades are therefore treated as possible readouts or interfaces
of a preserved interpretive judgment state, not as the primitive judgment
itself. This is consistent with the name ALA, Awareness Logic of Ambiguity:
the primitive object is an ambiguity-aware interpretive judgment state, while
membership-like scalar values are recovered only at the interface level.
\end{remark}

\begin{remark}[The epistemic-to-operational role of restraint]
\label{rem:epistemic-operational-restraint}
The term \emph{epistemic restraint} names the role that limits premature
cognitive permissiveness when a judgment is carried toward an operational
commitment interface. Thus the role is epistemic in its status within the
judgment state, while its visible effect is operational: it regulates the
permission to act on, report, escalate, or otherwise commit to the judgment.
This distinction prevents \(\gamma_A\) from being read as confidence,
reliability, hesitation, deferment, or non-membership. It is a
permission-limiting role internal to the interpretive judgment state.
\end{remark}

For fixed \((x,c,p)\), an ALA state is denoted by
\begin{equation}
\label{eq:ala-state}
A=(\theta_A,\sigma_A,\gamma_A)\in\cA.
\end{equation}
When no confusion can arise, the simpler coordinate notation \((\theta,\sigma,\gamma)\) is used. We also use the term \emph{ALA number} when the state is treated as an object of computation. This terminology does not remove the semantic typing of the coordinates. It only indicates that operations, orders, projections, and aggregations act on triadic interpretive states.

\begin{definition}[Role-factorized ALA assessment]
\label{def:role-factorized}
A role-factorized ALA assessment is an ALA-valued interpretive judgment assessment for which there exist component maps
\[
\theta_A^{\mathrm{fac}}:X\times\cP\to[0,1],\qquad
\sigma_A^{\mathrm{fac}}:X\times\cC\to[0,1],\qquad
\gamma_A^{\mathrm{fac}}:\cC\times\cP\to[0,1],
\]
such that
\[
\theta_A(x,c,p)=\theta_A^{\mathrm{fac}}(x,p),
\qquad
\sigma_A(x,c,p)=\sigma_A^{\mathrm{fac}}(x,c),
\qquad
\gamma_A(x,c,p)=\gamma_A^{\mathrm{fac}}(c,p).
\]
This factorized form is not a universal restriction on ALA. It is a canonical transparent subclass in which each coordinate is elicited from its most direct semantic arguments: valuation from the object and perspective, contextual alignment from the object and contextual condition, and epistemic restraint from the contextual condition and operational interface. When case-specific stakes make restraint object-dependent, as discussed in Remark~\ref{rem:general-factorized}, the general form rather than this subclass applies.
\end{definition}

\begin{principle}[Semantic auditability of judgment]
An ALA state should preserve not only a reported degree, but also a trace of the conditions under which that degree was produced. The purpose is not to guarantee that every evaluator is truthful, but to make the semantic source of a judgment auditable: what is supported, how it is contextually aligned, and how cautiously it should be used.
\end{principle}

\begin{remark}[General and factorized forms]
\label{rem:general-factorized}
The general definition allows all three coordinates to depend on \((x,c,p)\), which avoids treating the factorized dependencies as ontological restrictions. The factorized form of Definition~\ref{def:role-factorized} is used when it improves interpretability and makes the role structure explicit. General dependence, however, does not license semantic absorption: contextual inadequacy should not be silently absorbed into valuation, and epistemic restraint should not be confused with low support or low contextual alignment. The governing principle is that, even under general dependence, each coordinate may depend on any of \((x,c,p)\) only insofar as that input bears on its own role, and it must not take over the content for which another coordinate is responsible. Two instances make this concrete. First, valuation may depend on context to determine what the concept means in that setting, since concept meaning can be context relative, but it must not absorb the adequacy of the context, which remains the separate responsibility of contextual alignment. Second, the role of each coordinate is fixed: \(\gamma_A\) records restraint toward operational use, not a content-quality assessment of \(\theta_A\), so allowing \(\gamma_A\) to depend on \(x\) accommodates case-specific stakes, such as the irreversibility associated with a particular object, without converting restraint into a second-order judgment about valuation. This keeps the ontological distinction between epistemic restraint and a reliability-type qualifier, drawn in Section~\ref{sec:related-work}, intact under general dependence. The algebraic operations, the cognitive order, the admissible scalar projections, and the non-collapse results developed below apply to the resulting ALA states \(A\in\cA\) and do not depend on a particular factorization of the component maps.
\end{remark}

\begin{table}[!htbp]
\centering
\caption{Role boundary among the inputs and coordinates of an ALA assessment.}
\label{tab:input-role-boundary}
\fontsize{9.5}{11.5}\selectfont
\setlength{\tabcolsep}{3pt}
\renewcommand{\arraystretch}{1.12}
\begin{tabularx}{\textwidth}{>{\raggedright\arraybackslash}p{0.12\textwidth}>{\raggedright\arraybackslash}p{0.24\textwidth}>{\raggedright\arraybackslash}X}
\toprule
Symbol & Status & Operative meaning \\
\midrule
\(c\) & Input condition & Situational, evidential, linguistic, and institutional conditions under which the judgment is formed. \\
\(p\) & Declared perspective or protocol & Interpretive purpose, assessment protocol, disciplinary lens, or intended use. Its influence must remain role typed. \\
\(\sigma\) & ALA coordinate & Adequacy of the contextual conditions for forming and assessing the valuation under the declared purpose or protocol. \\
\(\gamma\) & ALA coordinate & Pressure for restraint toward operational commitment, including stakes, irreversibility, loss severity, and ethical burden. \\
\bottomrule
\end{tabularx}
\end{table}

The table is a boundary discipline, not a factorization theorem. In the general form, all three coordinates may depend on \((x,c,p)\); the requirement is that each dependence be interpreted through the coordinate's own semantic role.

\subsection{The Three Interpretive Coordinates}
The first coordinate records how strongly the object supports the concept under the adopted perspective. It is the interpretive valuation.

\begin{definition}[Interpretive valuation]
For an ALA state \(A\in\cA\), the coordinate \(\theta_A\in[0,1]\) is called the \emph{interpretive valuation}. It records the degree to which the object is interpreted as satisfying the concept under the relevant perspective.
\end{definition}

The valuation coordinate is the closest coordinate to classical membership. If only \(\theta\) is reported, then the ALA state is reduced to a familiar scalar assessment. Nevertheless, \(\theta\) is not the whole state. It records the intensity of the interpreted predicate, but it does not record whether the context supports that interpretation or whether the judgment is being held under strong epistemic restraint.

A value such as \(\theta=0.80\) is a valuation statement, not a complete ALA state. The value may be associated with a highly aligned and weakly restrained assessment, such as \((0.80,0.90,0.10)\), or with a weakly aligned and strongly restrained assessment, such as \((0.80,0.30,0.85)\). These states share the same valuation but do not carry the same interpretive information.

The second coordinate does not grade the object at all, but the conditions under which the valuation is formed. This is contextual alignment.

\begin{definition}[Contextual alignment]
For an ALA state \(A\in\cA\), the coordinate \(\sigma_A\in[0,1]\) is called \emph{contextual alignment}. It records the degree to which the judgment is aligned with, supported by, or appropriately grounded in the context in which the assessment is made.
\end{definition}

Contextual alignment is not the context itself. It is not a label for the situation, and it is not a context-dependent membership grade in disguise. It is a preserved coordinate of the output state. A high value of \(\sigma\) indicates that the interpretive environment coherently supports the valuation. A low value indicates contextual fragility, weak contextual support, mismatch between the evidence and the decision environment, or instability of the valuation under the relevant contextual profile.

It is also important to separate contextual alignment from reliability, probability, and indeterminacy. Reliability concerns the trustworthiness of a source or measuring process. Probability concerns likelihood. Indeterminacy concerns unresolved truth-status or incomplete determination. Contextual alignment concerns the fit between the judgment and the interpretive situation in which that judgment is produced. A reliable expert can issue a low-alignment judgment if the available context is poorly matched to the predicate being evaluated. Conversely, a moderately reliable source can operate in a highly aligned context when the situational cues are coherent and task-appropriate.

In applied settings, \(\sigma\) may be elicited through contextual support, contextual coherence, deviation from an ideal contextual profile, stability across admissible contextual perturbations, or agreement between the judgment and domain-specific background conditions. These are elicitation routes, not universal definitions. ALA requires the preservation of contextual alignment as a role; it does not require every discipline to measure that role by the same formula. A compact admissible elicitation schema for \(\sigma\), together with the discipline that keeps such elicitation compatible with the primitive status of the coordinates, is given in Section~\ref{subsec:elicitation}.

The third coordinate turns from evidence to action, and enters the state with reversed polarity. It is epistemic restraint.

\begin{definition}[Epistemic restraint]
For an ALA state \(A\in\cA\), the coordinate \(\gamma_A\in[0,1]\) is called \emph{epistemic restraint}. It records the degree to which the assessment is held back from immediate operational permissiveness because of stakes, irreversibility, cost of false commitment, ethical burden, or the need to avoid premature operational commitment.
\end{definition}

The polarity of \(\gamma\) is reversed relative to \(\theta\) and \(\sigma\). Larger \(\theta\) expresses stronger valuation, and larger \(\sigma\) expresses stronger contextual alignment. Larger \(\gamma\), however, expresses stronger restraint. A high-restraint state is not cognitively more permissive. It is more cautious. This polarity is not a cosmetic convention. It is the reason why the cognitive order and the lattice structure developed in Section~4 treat the third coordinate in the opposite direction.

Epistemic restraint is not passivity, and should not be interpreted as a preference for inaction.
It applies to any premature operational commitment, including both
intervention and non-intervention. A high value of \(\gamma\) may therefore
block an unsafe action, but it may also block an unsafe delay.

Epistemic restraint is also not a substitute for non-membership, hesitation, or indeterminacy. Non-membership concerns opposition to belonging. Hesitation concerns unallocated or unresolved membership mass in some fuzzy extensions. Indeterminacy concerns an unresolved truth-status. Epistemic restraint is different. It concerns the inhibition of operational assertion, even when valuation and contextual alignment may be substantial.

The representation \(\cA=[0,1]^3\) does not assert that valuation, contextual alignment, and epistemic restraint are raw commensurable quantities. It asserts that each role is represented by a normalized bounded score. The common carrier \([0,1]\) supplies a shared mathematical codomain for order, boundary elements, closure of operations, and scalar interfaces. It does not erase semantic typing.

\begin{principle}[Normalized bounded codomain]
In ALA, the interval \([0,1]\) is a normalized bounded codomain for role scores. It is not a common physical unit, not a claim that the three coordinates have the same empirical meaning, and not a requirement that the same elicitation instrument be used for all coordinates.
\end{principle}

\begin{proposition}[Bounded normalization]
Suppose that each interpretive role is elicited on a bounded ordered scale \(I_j=[a_j,b_j]\), with \(a_j<b_j\), for \(j\in\{\theta,\sigma,\gamma\}\). Then each role score can be normalized to \([0,1]\) by an order-preserving affine map
\begin{equation}
r_j(t)=\frac{t-a_j}{b_j-a_j}.
\label{eq:role-normalization}
\end{equation}
If a role has reversed cognitive polarity, as epistemic restraint does, the polarity is handled by the cognitive order, not by denying the normalized codomain.
\end{proposition}

\begin{proof}
For each bounded interval \([a_j,b_j]\), the map in Eq.~\eqref{eq:role-normalization} is increasing, sends \(a_j\) to \(0\), sends \(b_j\) to \(1\), and preserves the order internal to the elicited role. Hence each role can be expressed on a unit interval without making the roles semantically identical. For \(\gamma\), the normalized value still measures restraint. Its reversed effect on cognitive permissiveness is encoded later by the order relation, where larger \(\gamma\) means stronger restraint.
\end{proof}

\begin{remark}[Carrier-independent role architecture]
\label{rem:carrier-independent-architecture}
The unit cube \(\cA=[0,1]^3\) is the normalized real-valued realization developed completely in this paper; it is not an ontological restriction on the ALA role architecture. At a more general level, the typed state space may be written as
\begin{equation}
\mathcal A^{\star}
=
L_\theta\times L_\sigma\times L_\gamma^{\mathrm{op}},
\label{eq:general-ala-carrier}
\end{equation}
provided that each role carrier is equipped with a declared compatible order, boundary elements, and the role-preserving operations required by the intended calculus. The three carriers need not be identical. For example, a valuation role may be interval-valued while contextual alignment remains real-valued and restraint is represented by another bounded ordered carrier:
\[
A^{\star}
=
\bigl(a_\theta,a_\sigma,a_\gamma\bigr)
\in
\mathbb I_\theta\times[0,1]_\sigma\times\mathbb T_\gamma^{\mathrm{op}}.
\]
Its ALA identity is fixed by the three semantic roles, not by a requirement that the carriers share one numerical form.

This observation is architectural rather than a claim that every structured uncertainty formalism automatically supplies an admissible ALA carrier. A carrier-specific extension is available only after its native order, bounds, rolewise operations, and any standardization map have been declared and shown compatible with the ALA role directions. Internal components of a structured role value qualify that role; they do not thereby become additional primitive ALA roles. The unit cube is therefore a common normalized interface and a complete first realization of ALA, rather than the outer boundary of the theory.
\end{remark}

\begin{proposition}[Typed-product lattice inheritance]
\label{prop:typed-product-lattice-inheritance}
Let \(L_\theta\), \(L_\sigma\), and \(L_\gamma\) be bounded distributive lattices. Define the role-typed order on \(\mathcal A^{\star}=L_\theta\times L_\sigma\times L_\gamma^{\mathrm{op}}\) by
\[
(a_\theta,a_\sigma,a_\gamma)
\preceq_{\mathcal A^{\star}}
(b_\theta,b_\sigma,b_\gamma)
\]
if and only if
\[
a_\theta\leq_\theta b_\theta,
\qquad
a_\sigma\leq_\sigma b_\sigma,
\qquad b_\gamma\leq_\gamma a_\gamma.
\]
Then \(\mathcal A^{\star}\) is a bounded distributive lattice.
\end{proposition}

\begin{proof}
The order dual of a bounded distributive lattice is again a bounded distributive lattice, and finite direct products preserve boundedness and distributivity. Hence \(L_\theta\times L_\sigma\times L_\gamma^{\mathrm{op}}\) inherits a bounded distributive lattice structure under the stated role-typed order.
\end{proof}

The proposition transfers the order-theoretic skeleton, not automatically every operation developed below. A generalized complement, arithmetic, aggregation family, or projection requires the corresponding native structure on each selected carrier. The present paper supplies those constructions for \(L_\theta=L_\sigma=L_\gamma=[0,1]\), while carrier-specific and heterogeneous realizations are left as a declared research program.

This point is crucial for avoiding a category error. ALA does not claim that \(0.70\) units of valuation, \(0.70\) units of contextual alignment, and \(0.70\) units of restraint are physically the same kind of quantity. It claims that all three can be represented as bounded role intensities inside a common ordered state space.

\subsection{Non-absorptive Role Separation}
\label{subsec:role_separated_elicitation}

The role boundary between \(\theta\) and \(\sigma\) is fixed by what \(\theta\) is. If \(\theta\) were a final context-adjusted scalar judgment, a separate contextual coordinate would be redundant, and contextual weakness could simply be folded into a lower reported valuation. In ALA, \(\theta\) is not a final scalar judgment. It is the primitive content-directed valuation of the object with respect to a concept under an adopted perspective, and \(\sigma\) records the degree to which the interpretive conditions are aligned with the formation and assessment of that valuation.

The two coordinates therefore answer different role-specific questions. The valuation coordinate answers: to what degree does the object, signal, pattern, or predicate-relevant content support the concept? The contextual-alignment coordinate answers: to what degree are the interpretive conditions appropriate, aligned, and contextually supportive for forming and assessing that valuation? The first question concerns the content of the judgment. The second concerns the contextual alignment of the conditions under which the valuation is formed and assessed.

\begin{principle}[Valuation-context separation]
The valuation coordinate records the content-directed strength of the predicate assessment. The contextual-alignment coordinate records the degree to which the interpretive conditions are aligned with the formation and assessment of that valuation. Contextual weakness may affect a derived scalar readout, but it should not be silently absorbed into the primitive valuation coordinate when the purpose is to preserve the semantic source of the judgment.
\end{principle}

The principle does not claim that context has no effect on reported judgments. On the contrary, ALA explicitly allows context to influence subsequent reporting, thresholding, ranking, or decision interfacing. The point is that this influence should occur after the primitive state has been preserved. If contextual weakness is absorbed into \(\theta\) at the primitive level, the resulting number no longer reveals whether a low scalar value is caused by weak predicate support, weak contextual alignment, or both.

\begin{definition}[Role-separated elicitation]
An elicitation protocol for ALA states is called role-separated if the valuation coordinate and the contextual-alignment coordinate are obtained from distinct operational questions:
\[
\theta_A:\quad \text{How strongly does the object support the concept under the adopted perspective?}
\]
and
\[
\sigma_A:\quad \text{How well are the interpretive conditions aligned with the formation and assessment of that valuation?}
\]
A protocol violates role separation when evidence of contextual inadequacy is silently absorbed into \(\theta_A\) without preserving a separate record of contextual alignment.
\end{definition}

Role-separated elicitation can be guided by three counterfactual tests. First, if the object, predicate, and content-directed evidence are held fixed while measurement quality, calibration, domain fit, or interpretive environment changes, the change should primarily affect \(\sigma\), not \(\theta\). Second, if a weakness can be repaired by improving contextual information, calibration, background conditions, or domain alignment without changing the object-predicate evidence itself, the weakness belongs primarily to \(\sigma\); if the remedy is hypothesis revision or a change in predicate-fit, it belongs primarily to \(\theta\). Third, if the evidential basis and contextual alignment remain fixed while the intended use changes from a lower-commitment, comparatively reversible action to a high-stakes irreversible commitment, the change should be represented in \(\gamma\), not in \(\theta\) or \(\sigma\). These tests do not eliminate expert judgment, but they make role assignment auditable and prevent contextual weakness or operational risk from being silently absorbed into the wrong coordinate.

\begin{remark}[Noisy-scan test]
Consider a radiologist evaluating a lesion from a degraded or noisy scan. If the lesion-related morphology, boundary pattern, and predicate-directed evidence are held fixed while only the imaging quality or calibration condition deteriorates, the appropriate ALA response is not to lower \(\theta\) merely to compensate for the poorer scan. Lowering \(\theta\) would state that the lesion is less suspicious as a predicate-directed object, although the change concerns the interpretive setting rather than the lesion itself. The degradation should instead be recorded primarily in \(\sigma\), while any change in \(\theta\) should be declared only when the predicate-directed evidence itself has changed. This test makes the valuation--context boundary operational: contextual degradation may later reduce a scalar readout, but it should not be silently absorbed into primitive valuation.
\end{remark}

Role separation should not be interpreted as the claim that observation is never context-mediated. In many applications, the observed signal itself may be affected by measurement conditions, environmental factors, or the available information channel. The requirement is therefore not that \(\theta\) must remain numerically invariant under every contextual change. The requirement is that evidence concerning predicate-content support and evidence concerning contextual alignment should not be conflated without declaration. When the weakness concerns the interpretive setting, calibration condition, evidential environment, domain match, or contextual support of the valuation, it belongs primarily to \(\sigma\). When the weakness concerns the object-concept match itself, it belongs primarily to \(\theta\).

This separation has direct diagnostic consequences. Consider two ALA states with the same restraint coordinate:
\[
A_1=(0.80,0.30,\gamma),\qquad A_2=(0.30,0.80,\gamma).
\]
Under an absorptive scalar readout such as \(M(\theta,\sigma)=\theta\sigma\), both states produce the same value:
\[
M(0.80,0.30)=M(0.30,0.80)=0.24.
\]
However, the two states do not have the same operational meaning. In \(A_1\), the content-directed valuation is strong but the contextual alignment is weak. The appropriate response may be to improve the context, collect better data, recalibrate the measurement condition, or request a second assessment. In \(A_2\), the contextual alignment is strong but the valuation itself is weak. The appropriate response may instead be to revise the hypothesis, reconsider the object-concept match, or reduce the role of that judgment. A scalar value alone cannot distinguish these two diagnostic directions.

\begin{proposition}[Diagnostic indistinguishability under absorptive scalar projection]
Let \(A_1=(\theta_1,\sigma_1,\gamma)\) and \(A_2=(\theta_2,\sigma_2,\gamma)\) be two ALA states such that
\[
\theta_1>\theta_2,\qquad \sigma_1<\sigma_2.
\]
Suppose a scalar readout \(M:[0,1]^2\to[0,1]\) satisfies
\[
M(\theta_1,\sigma_1)=M(\theta_2,\sigma_2).
\]
Then any representation retaining only the scalar value \(M(\theta,\sigma)\) cannot distinguish whether the reduced scalar output is associated with stronger valuation under weaker contextual alignment or weaker valuation under stronger contextual alignment.
\end{proposition}

\begin{proof}
The reduced representation maps \(A_1\) and \(A_2\) to the same scalar value by assumption. Hence the retained scalar information is identical for the two states. However, the primitive ALA coordinates distinguish them: \(A_1\) has larger valuation and smaller contextual alignment, whereas \(A_2\) has smaller valuation and larger contextual alignment. Therefore the scalar readout loses the diagnostic direction of the difference between valuation weakness and contextual weakness.
\end{proof}

The proposition is elementary, but it shows why scalar projection before role preservation is operationally lossy. Its purpose is not to prohibit scalar readouts. ALA does not deny that context may modulate a final reported value. It denies that such modulation should erase the primitive roles before reasoning begins.

To make this distinction explicit, one may define a context-modulated valuation as a derived readout:
\[
\widetilde{\theta}_M(A)=M(\theta_A,\sigma_A),
\]
where \(M\) is a declared context-modulation operator. A conservative context-modulation operator may be required to satisfy
\[
M(0,\sigma)=0,\qquad M(\theta,1)=\theta,\qquad M(\theta,\sigma)\leq \theta,
\]
and to be nondecreasing in both arguments. Typical conservative choices include
\[
M_{\times}(\theta,\sigma)=\theta\sigma,\qquad
M_{\min}(\theta,\sigma)=\min\{\theta,\sigma\},
\]
or, more generally,
\[
M_{\alpha}(\theta,\sigma)=\theta\sigma^\alpha,\qquad \alpha>0.
\]
Such a readout may be useful when an external system requires a single context-sensitive valuation. Nevertheless, \(\widetilde{\theta}_M\) is not the primitive ALA valuation and is not a replacement for the ALA state. Unlike a full admissible scalar projection, \(\widetilde{\theta}_M\) does not necessarily incorporate epistemic restraint and therefore should be understood only as a partial context-sensitive readout. The correct direction is
\[
A\in\cA\longrightarrow \widetilde{\theta}_M(A),
\]
not
\[
A\in\cA\longrightarrow \theta^\ast
\quad\text{and then discard }\sigma.
\]

This distinction also clarifies the status of \(\theta\) alone. The coordinate \(\theta\) is meaningful without \(\sigma\), because it records the content-directed valuation. However, it is not complete without \(\sigma\), because it does not record the contextual alignment of the conditions under which the valuation is made. In the same way, a reported effect size in an empirical study is meaningful, but it is not sufficient for a full judgment unless the study design, sample adequacy, and contextual conditions are also considered.

Contextual alignment may modulate a derived scalar readout, but it must not be absorbed into primitive valuation if the aim is to preserve the interpretive source of the judgment. This is the methodological form of reasoning before projection.

The distinction between epistemic restraint and confidence should not be reduced to a statement of semantic independence. Independence is necessary, but it is not sufficient for establishing a primitive role in a formal calculus. The more fundamental distinction is that confidence and epistemic restraint answer different operational questions.

A confidence-like quantity usually answers the question: how assured is the evaluator about the judgment? By contrast, epistemic restraint answers the question: to what degree should the judgment resist premature operationalization? Thus, confidence concerns the strength, stability, or assurance of belief, whereas epistemic restraint concerns the permission boundary between a preserved judgment and its immediate use in assertion, action, thresholding, ranking, or scalar projection.

This gives the following principle.

\begin{principle}[Support-permission separation]
The valuation and contextual-alignment coordinates describe the support carried by a judgment. The restraint coordinate describes the degree to which this support is withheld from immediate operational commitment. Hence epistemic restraint is not a support-strengthening coordinate. It is a permission-limiting coordinate.
\end{principle}

Accordingly, an ALA state may be strongly supported and strongly restrained at the same time. A state such as
\[
A=(0.90,0.85,0.90)
\]
is not contradictory. It represents a judgment with strong content-directed valuation and strong contextual alignment, but with a high degree of restraint against immediate operationalization. In a high-stakes setting, this means that the judgment may be well supported while still being withheld from direct action because the cost of a false commitment, irreversibility, ethical burden, or operational sensitivity is high.

This distinction is central to ALA. The pair \((\theta,\sigma)\) records the support side of the judgment, whereas \(\gamma\) records the permission-limiting side. The scalar interface may later combine these components, but the primitive state preserves the distinction before projection.

\subsection{Admissible Elicitation}
\label{subsec:elicitation}

The three coordinates have now been introduced as primitive interpretive roles. In a given application, their numerical values are obtained through domain-specific elicitation. In a clinical setting, for instance, valuation may be informed by lesion shape, boundary, and texture; contextual alignment by image quality, calibration, and domain fit; and epistemic restraint by the severity and irreversibility of the contemplated action. Such factors may supply numerical access to the coordinates, but they do not define them. This subsection records the minimal discipline that keeps elicitation compatible with the primitive status of the triad.

The guiding distinction is between a semantic role and a route of access to that role.

\begin{principle}[Role and measurement separation]
\label{prin:role-measurement}
Each ALA coordinate is a primitive semantic role. An elicitation procedure is a domain-specific route for assigning a numerical value to that role. Eliciting a coordinate from role-specific factors does not convert the coordinate into a derived quantity, just as measuring a temperature with a particular instrument does not make temperature a property of that instrument.
\end{principle}

Equivalently, elicitation supplies access to a primitive role; it does not define the role away.

This separation is the upstream counterpart of scalar projection. An admissible projection is a declared downstream interface from the state to a scalar readout, and no one concludes from its existence that the state is merely a derived scalar. By the same reasoning, elicitation is a declared upstream interface into the state, and the state is not a derived quantity of its elicitation factors. 

\begin{figure}[!htbp]
\centering
\resizebox{0.95\textwidth}{!}{%
\begin{tikzpicture}[
    font=\small,
    >=Latex,
    box/.style={
    draw=Ccube,
    rounded corners,
    align=center,
    inner sep=5pt,
    minimum height=0.95cm,
    text width=3.7cm
    },
     corebox/.style={
    draw=Cproj,
    very thick,
    rounded corners,
    align=center,
    inner sep=5pt,
    minimum height=1.00cm,
    text width=3.9cm
    },
    arr/.style={-{Latex[length=2.5mm]}, thick, Ccube},
    lab/.style={
        font=\scriptsize,
        align=center,
        fill=white,
        inner sep=2pt
    }
]

\node[box] (factors) at (0,2.5)
{role-typed factors\\[-1pt]
\emph{boundary scope}};

\node[corebox] (state) at (7.0,2.5)
{\(A\in\cA\)\\[-1pt]
\emph{protected ALA core}};

\node[box] (ops) at (7.0,-1.1)
{state-level operations\\[-1pt]
\emph{order, lattice, complement}\\
\emph{arithmetic, aggregation}};

\node[box] (scalar) at (14.0,-1.1)
{\(\phi(A)\)\\[-1pt]
\emph{scalar readout}};

\draw[arr] (factors.east) -- node[lab,above=4pt]
{declared admissible\\elicitation} (state.west);

\draw[arr] (state.south) -- node[lab,right=6pt]
{operate} (ops.north);

\draw[arr] (ops.east) -- node[lab,above=4pt]
{declared admissible\\projection} (scalar.west);

\end{tikzpicture}%
}

\caption{Protected core between elicitation and projection.}
\label{fig:elicitation-state-projection}
\end{figure}
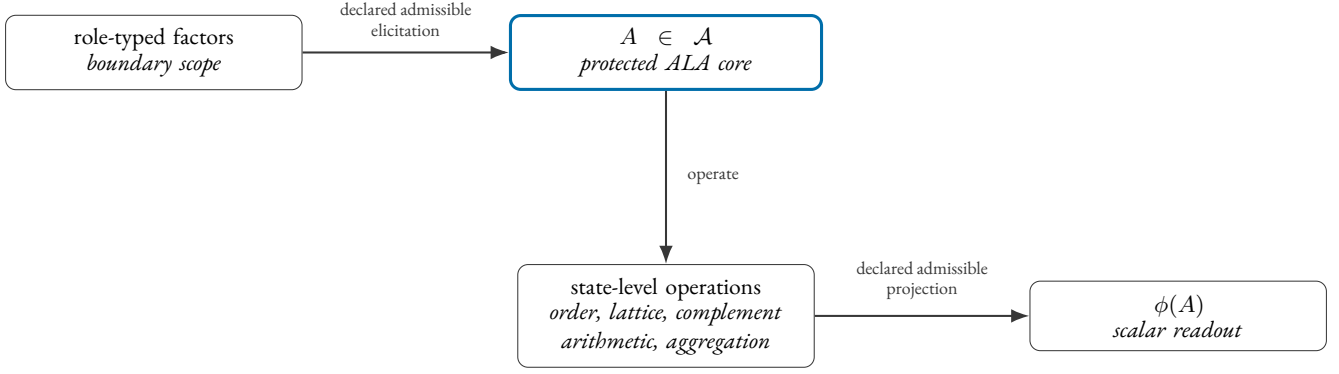

Fig.~\ref{fig:elicitation-state-projection} separates three layers that must not be conflated. Role-typed factors have boundary scope: they provide admissible access to the ALA state but do not become coordinates of the calculus. Once formed, the protected state \(A\in\cA\) is the object on which the cognitive order, lattice operations, complement, arithmetic, aggregation, projection, and non-collapse results operate. Scalar projection appears only as a declared downstream interface from the preserved ALA state.

The reason the elicitation factors are not themselves coordinates is a difference of operative scope.

A primitive ALA coordinate has global operative scope: it remains active throughout the ALA calculus, participating in the cognitive order, lattice operations, complement, projection, arithmetic, and aggregation. An elicitation factor has boundary scope: it is active only in forming the state and does not participate in any ALA operation once the state has been formed. After elicitation, the factor is discharged; the value it helped produce participates as a primitive coordinate.

A minimal admissibility standard follows. A route \(E_r\) that elicits a coordinate \(r\in\{\theta,\sigma,\gamma\}\) from a role profile is admissible when it is bounded in \([0,1]\), boundary-grounded, monotone with respect to the role it elicits, and role-typed: its input factors must belong to the semantic type of the target role. Content-directed factors elicit \(\theta\), context-directed factors elicit \(\sigma\), and commitment-directed factors elicit \(\gamma\). Cross-role leakage, such as lowering \(\theta\) to express the irreversibility of an action, is an elicitation type error, since it reintroduces at the input boundary precisely the absorption that the valuation--context separation and the support--permission separation forbid. ALA therefore requires admissible role-preserving routes to its coordinates, not a single universal elicitation formula; its universality lies in the admissibility conditions rather than in any particular function. When several admissible routes are available, the selected route is declared, exactly as a scalar projection is declared.

For example, a contextual-alignment value may be elicited from a context-directed profile \(z_A^\sigma(x,c,p)\), such as data completeness, calibration quality, and domain match. If a normalized discrepancy
\[
d_{A,p}^{\sigma}
\bigl(z_A^{\sigma}(x,c,p),z_{A,p}^{\sigma,\star}\bigr)\in[0,1]
\]
measures the deviation from an ideal contextual profile \(z_{A,p}^{\sigma,\star}\), one admissible route is
\[
\sigma_A(x,c,p)
=
1-d_{A,p}^{\sigma}
\bigl(z_A^{\sigma}(x,c,p),z_{A,p}^{\sigma,\star}\bigr).
\]
This route measures contextual adequacy, not predicate support and not permission for operational use. Once \(\sigma_A(x,c,p)\) is elicited, the profile variables are not processed by the ALA calculus; only the resulting state coordinate is.

The valuation coordinate is the anchor case of this scheme. When \(\theta\) is elicited directly from a normalized predicate-support score, the route reduces to classical fuzzy membership. ALA thus does not displace graded membership; it situates it as the simplest admissible elicitation of a single coordinate, while the triadic state keeps contextual alignment and epistemic restraint explicit rather than absorbing them into valuation. In this sense, ALA preserves continuity with the fuzzy-set tradition while making explicit the additional roles that classical membership leaves untyped.

For \(\gamma\), admissible elicitation must instead be commitment-directed. To keep \(\gamma\) from collapsing into confidence or reliability, a factor used to elicit \(\gamma\) should not measure the strength or quality of evidence for the predicate; it should measure pressure for restraint toward commitment, such as loss severity or irreversibility. In the present foundational treatment, an uncertainty term that reports weak predicate support belongs to \(\theta\), and one that reports inadequate interpretive conditions belongs to \(\sigma\), not to the restraint profile.

One consequence of this discipline concerns aggregation. Because factors have boundary scope, the aggregation operators of Section~\ref{sec:aggregation} act on states, not on factors. State-level aggregation of already formed ALA states,
\[
\operatorname{Agg}_{A}(A_1,\ldots,A_n),
\]
is in general not equal to factor-level fusion followed by re-elicitation. Here \(z_i^r\) denotes the role profile used in a particular elicitation route. Thus, in general,
\[
\operatorname{Agg}_{A}(A_1,\ldots,A_n)
\neq
E_r\bigl(\operatorname{Agg}_{F}(z_1^r,\ldots,z_n^r)\bigr).
\]
The latter may be a legitimate pre-state modeling choice, but it is not the state-level aggregation defined by the ALA calculus.

\begin{remark}[Scope of this subsection]
\label{rem:elicitation-scope}
The present subsection states only the discipline that protects the primitive roles. A fuller theory of role profiles, admissible elicitation families, traceable elicitation, profile-level fragility, and diagnostic auditing is a natural extension and is left to future work. Nothing in the formal development that follows depends on a particular elicitation route; the calculus operates on the state \(A\in\cA\) regardless of how its coordinates were obtained.
\end{remark}

\section{The Cognitive Order}
\label{sec:cognitive-order}

Section~\ref{sec:ala-state-space} introduced the ALA state as a semantic object. The present section turns that object into an ordered structure. The central point is the reversed polarity of epistemic restraint. Larger valuation and stronger contextual alignment make a state cognitively stronger, whereas larger restraint makes it less permissive for assertion or action. The order must therefore increase in \(\theta\) and \(\sigma\), but decrease in \(\gamma\).

\subsection{Order Definition and Alignment}
Let
\[
A=(\theta_A,\sigma_A,\gamma_A),\qquad
B=(\theta_B,\sigma_B,\gamma_B)
\]
be two ALA states. The cognitive order is defined as follows.

\begin{definition}[Cognitive order]
For \(A,B\in\cA\), write \(A\preceqA B\) if and only if
\begin{equation}
\theta_A\leq \theta_B,\qquad
\sigma_A\leq \sigma_B,\qquad
\gamma_A\geq \gamma_B.
\label{eq:cognitive-order}
\end{equation}
When \(A\preceqA B\), the state \(B\) is said to be cognitively at least as permissive as \(A\).
\end{definition}

When neither \(A\preceqA B\) nor \(B\preceqA A\) holds, the states \(A\) and \(B\) are called incomparable. In that case we write \(A\parallelA B\). We write \(A\precA B\), and say that \(B\) is strictly more permissive than \(A\), when \(A\preceqA B\) and \(A\neq B\).

The word permissive is used in an operational sense. A state is more permissive when it has larger interpretive valuation, stronger contextual alignment, and weaker epistemic restraint. This does not mean that a more permissive state is always preferable. Preference and action policy are task-dependent matters introduced later through admissible projections and decision interfaces. The cognitive order records structural direction before such a policy is imposed.

It is often useful to replace restraint by the order-aligned coordinate
\[
\rho_A=1-\gamma_A.
\]
The coordinate \(\rho_A\) measures released permissiveness rather than restraint. Define
\begin{equation}
\Psi(A)=(\theta_A,\sigma_A,1-\gamma_A)=(\theta_A,\sigma_A,\rho_A).
\label{eq:psi-map-section4}
\end{equation}
Then Eq.~\eqref{eq:cognitive-order} is equivalent to the ordinary coordinatewise order of \(\Psi(A)\) and \(\Psi(B)\) in \([0,1]^3\).

\begin{proposition}[Order alignment]
\label{prop:order-alignment}
For all \(A,B\in\cA\),
\[
A\preceqA B
\quad\Longleftrightarrow\quad
\Psi(A)\leq_{\mathrm{cw}}\Psi(B),
\]
where \(\leq_{\mathrm{cw}}\) denotes the usual coordinatewise order on \([0,1]^3\).
\end{proposition}

\begin{proof}
By definition, \(\Psi(A)\leq_{\mathrm{cw}}\Psi(B)\) means
\(\theta_A\leq\theta_B\), \(\sigma_A\leq\sigma_B\), and \(1-\gamma_A\leq 1-\gamma_B\). The last inequality is equivalent to \(\gamma_A\geq\gamma_B\). These are exactly the three inequalities in Eq.~\eqref{eq:cognitive-order}.
\end{proof}


\begin{figure}[!htbp]
\centering
\resizebox{0.72\textwidth}{!}{%
\begin{tikzpicture}[x=4.2cm,y=4.2cm,>=Latex]

  \ALAUnitCubeFrame

  \draw[Ctheta,-{Latex[length=2mm]},thick]
    \Pp{0}{0}{0} -- \Pp{1.08}{0}{0}
    node[below right,black,font=\footnotesize]{$\theta$ valuation};

  \draw[Csigma,-{Latex[length=2mm]},thick]
    \Pp{1}{0}{0} -- \Pp{1}{1.08}{0}
    node[right,black,font=\footnotesize]{$\sigma$ contextual alignment};

  \draw[Cgamma,-{Latex[length=2mm]},thick]
    \Pp{0}{0}{0} -- \Pp{0}{0}{1.08}
    node[left,black,font=\footnotesize]{$\gamma$ epistemic restraint};

  \node[
    draw,
    rounded corners,
    align=center,
    font=\scriptsize,
    inner sep=4pt,
    fill=white
  ] at \Pp{0.58}{0.90}{1.10}
  {$A\preceqA B$ iff\\
  \(\theta_A\le\theta_B,\ \sigma_A\le\sigma_B,\ \gamma_A\ge\gamma_B\)};

  \draw[Cgamma!80!black,thick,-{Latex[length=2mm]}]
    \Pp{-0.10}{0}{0.82} -- \Pp{-0.10}{0}{0.48};

  \node[Cgamma!80!black,font=\scriptsize,align=center]
    at \Pp{-0.25}{0}{0.65}
    {order\\decreases};

\end{tikzpicture}%
}

\caption{Cognitive order in the state space.}
\label{fig:ala-state-space-order}
\end{figure}
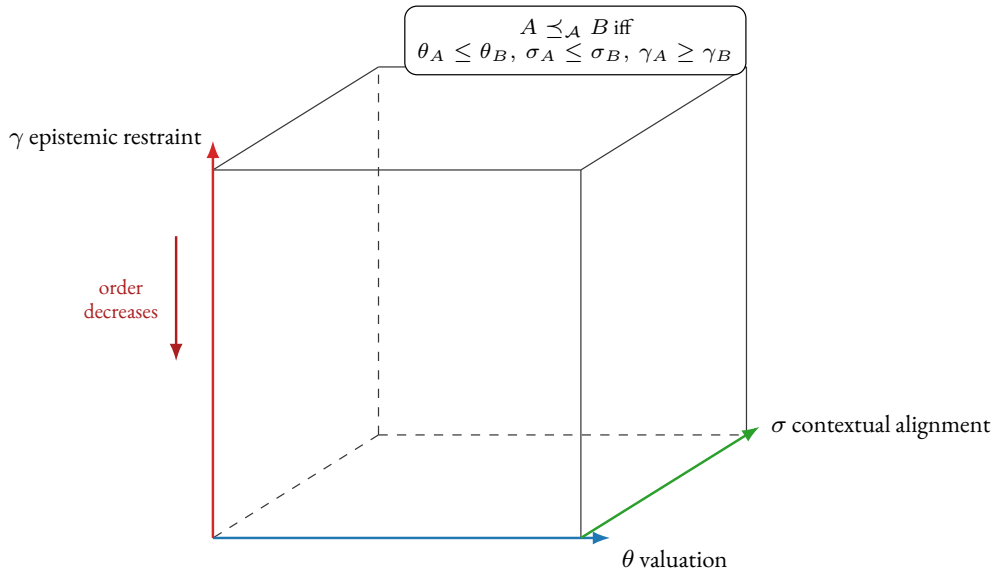

Fig.~\ref{fig:ala-state-space-order} shows the ALA state space in the original role coordinates. The first two coordinates are order-increasing, while epistemic restraint has reversed polarity: a larger \(\gamma\) means stronger restraint and therefore lower cognitive permissiveness. The order-alignment map \(\Psi(A)=(\theta_A,\sigma_A,1-\gamma_A)\) remains the proof device that converts this role-typed order into the ordinary coordinatewise order.

The least and greatest elements of \((\cA,\preceqA)\) are determined by the polarity of the three roles.

\begin{definition}[Bottom and top states]
The bottom and top states of \(\cA\) are
\begin{equation}
\bot_{\cA}=(0,0,1),
\qquad
\top_{\cA}=(1,1,0).
\label{eq:bottom-top-ala}
\end{equation}
\end{definition}

The bottom state has no valuation, no contextual alignment, and maximal restraint. The top state has maximal valuation, maximal contextual alignment, and no restraint. These states should be read structurally. They define the extreme boundaries of the ALA order; they are not empirical claims that real judgments must frequently attain these limits.

\begin{proposition}[Bounds]
For every \(A\in\cA\),
\[
\bot_{\cA}\preceqA A\preceqA \top_{\cA}.
\]
\end{proposition}

\begin{proof}
For \(\bot_{\cA}\preceqA A\), we need \(0\leq\theta_A\), \(0\leq\sigma_A\), and \(1\geq\gamma_A\), all of which hold in \([0,1]^3\). For \(A\preceqA\top_{\cA}\), we need \(\theta_A\leq1\), \(\sigma_A\leq1\), and \(\gamma_A\geq0\), which also hold in \([0,1]^3\).
\end{proof}

\subsection{Lattice Structure}
The cognitive order induces two canonical binary operations. The join selects the least state that is at least as permissive as both inputs. The meet selects the greatest state that is no more permissive than both inputs.

\begin{definition}[Join and meet]
For \(A,B\in\cA\), define
\begin{equation}
A\vee_{\cA}B=
\bigl(\max\{\theta_A,\theta_B\},\max\{\sigma_A,\sigma_B\},\min\{\gamma_A,\gamma_B\}\bigr),
\label{eq:ala-join}
\end{equation}
\begin{equation}
A\wedge_{\cA}B=
\bigl(\min\{\theta_A,\theta_B\},\min\{\sigma_A,\sigma_B\},\max\{\gamma_A,\gamma_B\}\bigr).
\label{eq:ala-meet}
\end{equation}
\end{definition}

The operation \(\vee_{\cA}\) is permissive because it keeps the larger valuation, the stronger contextual alignment, and the weaker restraint. The operation \(\wedge_{\cA}\) is conservative because it keeps the smaller valuation, the weaker contextual alignment, and the stronger restraint. These descriptions are order-theoretic, not psychological. They specify how upper and lower bounds are constructed inside the ALA state space.

The geometry of these two operations is shown in Fig.~\ref{fig:meet-join-incomparability}. The example uses two incomparable states in which valuation and contextual alignment point in opposite directions; their meet and join are then obtained by applying the reversed polarity of \(\gamma\).


\begin{figure}[!htbp]
\centering
\begin{tikzpicture}[x=4.2cm,y=4.2cm,>=Latex]
  \ALAUnitCubeFrame
  \ALAUnitCubeAxes

  \def\ax{0.30}\def\ay{0.65}\def\az{0.70}
  \def\bx{0.70}\def\by{0.25}\def\bz{0.35}
  \def\mx{0.30}\def\my{0.25}\def\mz{0.70}
  \def\jx{0.70}\def\jy{0.65}\def\jz{0.35}

  \draw[Cproj!60,thin]
    \Pp{\mx}{\my}{\jz} -- \Pp{\jx}{\my}{\jz} -- \Pp{\jx}{\jy}{\jz} -- \Pp{\mx}{\jy}{\jz} -- cycle;
  \draw[Cproj!60,thin]
    \Pp{\mx}{\my}{\mz} -- \Pp{\jx}{\my}{\mz} -- \Pp{\jx}{\jy}{\mz} -- \Pp{\mx}{\jy}{\mz} -- cycle;
  \draw[Cproj!60,thin] \Pp{\mx}{\my}{\jz} -- \Pp{\mx}{\my}{\mz};
  \draw[Cproj!60,thin] \Pp{\jx}{\my}{\jz} -- \Pp{\jx}{\my}{\mz};
  \draw[Cproj!60,thin] \Pp{\jx}{\jy}{\jz} -- \Pp{\jx}{\jy}{\mz};
  \draw[Cproj!60,thin] \Pp{\mx}{\jy}{\jz} -- \Pp{\mx}{\jy}{\mz};

  \fill[Ctheta] \Pp{\ax}{\ay}{\az} circle (1.6pt)
    node[above left,black,font=\scriptsize]
    {$a=(0.30,0.65,0.70)$};
  \fill[Csigma] \Pp{\bx}{\by}{\bz} circle (1.6pt)
    node[below right,black,font=\scriptsize]
    {$b=(0.70,0.25,0.35)$};

  \fill[Cmeet] \Pp{\mx}{\my}{\mz} circle (1.8pt)
    node[above left,Cmeet,font=\scriptsize]
    {$a\meetA b=(0.30,0.25,0.70)$};
  \fill[Cjoin] \Pp{\jx}{\jy}{\jz} circle (1.8pt)
    node[below right,Cjoin,font=\scriptsize]
    {$a\joinA b=(0.70,0.65,0.35)$};

  \node[draw,rounded corners,align=center,font=\scriptsize,inner sep=4pt]
       at \Pp{0.50}{0.95}{1.08}
       {$a\parallel b$ because \(\theta\) and \(\sigma\) point in opposite directions\\
        \(\meetA=(\min\theta,\min\sigma,\max\gamma)\),\quad
        \(\joinA=(\max\theta,\max\sigma,\min\gamma)\)};
\end{tikzpicture}
\caption{Meet--join geometry under reversed restraint.}
\label{fig:meet-join-incomparability}
\end{figure}
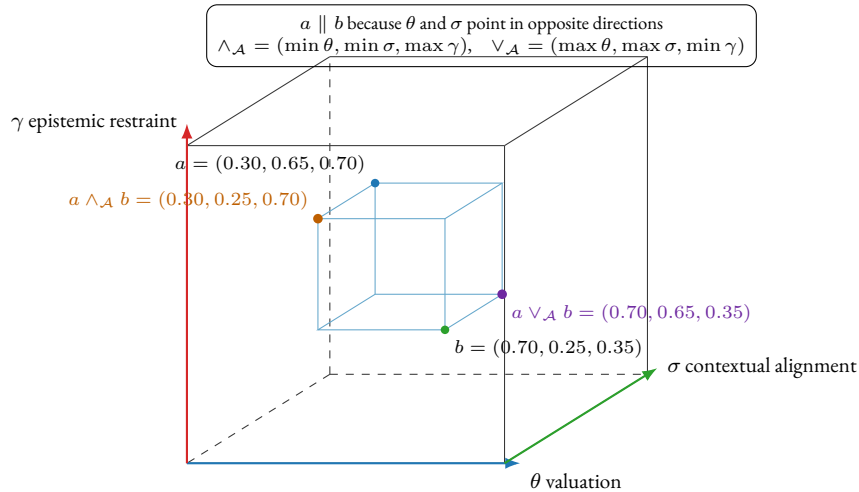
In this example, \(a\parallelA b\) because \(a\) has smaller valuation but larger contextual alignment than \(b\). The meet keeps the lower valuation, lower contextual alignment, and stronger restraint, whereas the join keeps the higher valuation, higher contextual alignment, and weaker restraint.

\begin{proposition}[Least upper bound and greatest lower bound]
For all \(A,B\in\cA\), \(A\vee_{\cA}B\) is the least upper bound of \(\{A,B\}\) and \(A\wedge_{\cA}B\) is the greatest lower bound of \(\{A,B\}\) with respect to \(\preceqA\).
\end{proposition}

\begin{proof}
Apply the order-alignment map \(\Psi\). In \([0,1]^3\) with the ordinary coordinatewise order, the least upper bound is obtained by componentwise maximum and the greatest lower bound is obtained by componentwise minimum. Since \(\Psi(A)=(\theta_A,\sigma_A,1-\gamma_A)\), componentwise maximum in the third order-aligned coordinate corresponds to the minimum of \(\gamma\), and componentwise minimum corresponds to the maximum of \(\gamma\). This gives Eqs.~\eqref{eq:ala-join} and \eqref{eq:ala-meet}.
\end{proof}

The construction above yields more than binary bounds: the induced structure is a bounded distributive lattice. 

\begin{theorem}[Bounded distributive lattice]
The structure
\[
(\cA,\vee_{\cA},\wedge_{\cA},\bot_{\cA},\top_{\cA})
\]
is a bounded distributive lattice.
\end{theorem}

\begin{proof}
The map \(\Psi:\cA\to[0,1]^3\), defined by Eq.~\eqref{eq:psi-map-section4}, is a bijection. By Proposition~\ref{prop:order-alignment}, it is an order-isomorphism from \((\cA,\preceqA)\) to \(([0,1]^3,\leq_{\mathrm{cw}})\). The coordinatewise ordered cube \([0,1]^3\), equipped with componentwise maximum and minimum, is a bounded distributive lattice with bottom \((0,0,0)\) and top \((1,1,1)\). Pulling this lattice structure back through \(\Psi\) gives Eqs.~\eqref{eq:bottom-top-ala}, \eqref{eq:ala-join}, and \eqref{eq:ala-meet}. Since distributivity is preserved by order-isomorphism, \(\cA\) is a bounded distributive lattice.
\end{proof}

Equivalently, distributivity can be checked coordinate by coordinate from the distributive identities of \(\max\) and \(\min\) on \([0,1]\). The order-isomorphism proof is preferable because it makes the role of reversed restraint polarity explicit.

Fig.~\ref{fig:meet-join-incomparability} illustrates the same formulas geometrically. The point is not merely that meet and join are coordinatewise; it is that their third-coordinate behavior is coordinatewise only after respecting the reversed polarity of epistemic restraint.

\subsection{The Interpretive Complement}
ALA also needs a complement operation, but the complement must respect the semantic typing of the coordinates. The complement of an evaluative claim reverses the valuation of the claim. It does not reverse the contextual alignment of the assessment, and it does not reverse the restraint under which the assessment is held.

\begin{definition}[Interpretive complement]
For an ALA state \(A=(\theta_A,\sigma_A,\gamma_A)\), the interpretive complement, taken at fixed context and fixed commitment interface, is
\begin{equation}
A^{c}=(1-\theta_A,\sigma_A,\gamma_A).
\label{eq:ala-complement}
\end{equation}
If negating the predicate changes the intended context or operational interface, the resulting object is a new ALA assessment rather than the complement of the original state.
\end{definition}

Thus, if \(A\) represents the assessment that an object belongs to a concept, \(A^c\) represents the complementary valuation of that concept under the same contextual alignment and the same epistemic restraint. For instance, if an image is assessed as \textsc{Tumor-Like} with a certain contextual alignment and a certain restraint, the complementary assessment \textsc{Not Tumor-Like} reverses the valuation coordinate. The imaging context and the epistemic restraint associated with the medical stakes do not automatically become their opposites.

\begin{proposition}[Basic properties of the interpretive complement]
For every \(A\in\cA\),
\[
(A^c)^c=A.
\]
Moreover, \(A^c\in\cA\). The operation preserves contextual alignment and epistemic restraint.
\end{proposition}

\begin{proof}
Since \(\theta_A\in[0,1]\), also \(1-\theta_A\in[0,1]\). The other two coordinates are unchanged and remain in \([0,1]\), so \(A^c\in\cA\). Applying the operation twice gives
\((A^c)^c=(1-(1-\theta_A),\sigma_A,\gamma_A)=A\). Preservation of \(\sigma\) and \(\gamma\) follows directly from Eq.~\eqref{eq:ala-complement}.
\end{proof}

The following remark gives a concrete order-theoretic check of this point. It shows that the complement reverses the valuation direction while leaving the contextual-alignment and restraint roles in place; consequently, a comparable pair may become incomparable after complementation.

\begin{remark}[The interpretive complement may break comparability]
The interpretive complement \(A^c=(1-\theta,\sigma,\gamma)\) reverses only the valuation coordinate and preserves contextual alignment and epistemic restraint. Hence it is neither a monotone nor an order-reversing map with respect to the ALA cognitive order. For example, let
\[
a=(0.25,0.40,0.70),\qquad b=(0.55,0.60,0.50).
\]
Then \(a\preceqA b\), since
\[
0.25\le 0.55,\qquad 0.40\le 0.60,\qquad 0.70\ge 0.50.
\]
However,
\[
a^c=(0.75,0.40,0.70),\qquad b^c=(0.45,0.60,0.50),
\]
and the complemented pair has mixed coordinate directions. The valuation coordinate favors \(a^c\), while contextual alignment and the reversed restraint order favor \(b^c\). Therefore,
\[
a^c\parallelA b^c.
\]
Thus the interpretive complement is not a De Morgan negation, not a lattice complement, and not an order-reversing operation.
\end{remark}

The interpretive complement is not a Boolean lattice complement. In general,
\[
A\vee_{\cA}A^c\neq \top_{\cA},
\qquad
A\wedge_{\cA}A^c\neq \bot_{\cA}.
\]
This is intended. ALA separates the truth-direction of the evaluated concept from the contextual and epistemic conditions under which the evaluation is made. Negating the concept changes valuation; it does not negate the context and it does not negate restraint.

A full Boolean complement would require contextual alignment to become \(1-\sigma\) and restraint to become \(1-\gamma\), or some analogous reversal. That would be semantically wrong in ALA. Poor contextual alignment is not the complement of good contextual alignment with respect to the negated concept. High restraint is not the complement of low restraint with respect to the negated concept. Both are features of the assessment situation, not membership and non-membership degrees.

The invariance of restraint under complementation rests on a stake-symmetry. Operational risk belongs to the commitment interface, that is, to which action is contemplated, and not to the truth-direction of the predicate. Acting on a proposition and acting on its negation can both be high-stakes: judging a lesion ``tumor-like'' may license a risky intervention, while judging it ``not tumor-like'' may license a risky discharge. Permission-to-act therefore does not track the valuation coordinate, and negating the predicate, which holds the contemplated action fixed, does not by itself alter the restraint that should govern commitment. Restraint may of course differ across different actions; but that variation comes from a change of commitment interface, not from the logical complement. A coordinate that stays fixed under predicate-negation, while valuation flips, is by that behavior not a truth-status component, which is the formal counterpart of the semantic claim that restraint is neither non-membership nor indeterminacy.

\begin{remark}[Distinction from non-membership and hesitation]
The complement in Eq.~\eqref{eq:ala-complement} also fixes the status of the third coordinate: \(\gamma\) is not non-membership and not hesitation. If it were non-membership, complement would naturally exchange it with membership-like valuation. If it were hesitation, it would belong to a different mass-balance logic. In ALA, \(\gamma\) measures epistemic restraint. It is therefore preserved under interpretive complementation.
\end{remark}

\section{Scalar Projection as a Delayed Interface}
\label{sec:scalar-projection}

The preceding sections define ALA states as triadic objects and equip them with a cognitive order. Nevertheless, many external procedures still require a scalar output. A journal reviewer may ask for a single membership value, a decision model may require a one-dimensional score, a ranking procedure may need a numerical index, or an existing fuzzy system may accept only a scalar membership grade. ALA does not deny the practical usefulness of such scalar outputs. It rejects only the assumption that the scalar output is the primitive representation of the judgment.

This section formalizes scalar projection as a delayed interface. Projection is applied after the triadic state has been formed, interpreted, ordered, and, when needed, operated upon. The principle is therefore not ``avoid projection'', but rather ``do not project before the semantic roles of the judgment have been preserved''.

\begin{remark}[Rolewise standardization before cross-role projection]
\label{rem:rolewise-standardization}
For a generalized state in Eq.~\eqref{eq:general-ala-carrier}, two interfaces must be distinguished. A \emph{rolewise standardization interface} has the product form
\begin{equation}
\mathcal R
=
\rho_\theta\times\rho_\sigma\times\rho_\gamma:
\mathcal A^{\star}\longrightarrow\cA=[0,1]^3,
\label{eq:rolewise-standardization}
\end{equation}
where each \(\rho_r\) reads only the native carrier of its own role and is order-preserving in that carrier. It does not mix valuation, contextual alignment, and restraint. By contrast, a scalar projection \(\phi:\cA\to[0,1]\) is a cross-role decision or reporting interface that deliberately combines the standardized roles.

For a structured carrier, \(\rho_r\) may itself factor as
\[
\rho_r=q_r\circ N_r,
\]
where \(N_r\) is a structure-preserving normalization within the native family and \(q_r\) is a declared scalar readout of that normalized role value. An interval may first be normalized as an interval, and another structured role value may first be normalized within its own declared carrier; only the readout \(q_r\) reduces that internal role structure to one value in \([0,1]\). Such reduction is therefore not hidden inside the word ``normalization''. Any reliability, hesitation, or other internal qualification used within a carrier must remain a qualification of that role and must not silently take over the function of contextual alignment or restraint.

The generalized processing order is
\[
(\mathcal A^{\star})^n
\xrightarrow{\text{native role-local reasoning}}
\mathcal A^{\star}
\xrightarrow{\mathcal R}
\cA
\xrightarrow{\phi}
[0,1].
\]
Native reasoning should normally precede \(\mathcal R\), because standardizing or reducing each input first may discard interval, fuzzy, hesitant, reliability, or other internal information before it participates in the calculus. The complete theory of admissible \(\mathcal R\), including bound-valued standardizations and commutation conditions with native operations, is outside the scope of this paper. Equation~\eqref{eq:rolewise-standardization} locates the unit-cube calculus developed here as the standard interface between heterogeneous role carriers and the delayed scalar projections defined below.
\end{remark}

\subsection{Admissible Scalar Projection}
A scalar projection is not an ALA state, not a universal truth value, and not a replacement for triadic judgment state. It is a task-dependent scalar readout of a state whose semantic structure remains available before and after projection. Two consequences follow immediately.

First, different projections may be appropriate for different tasks. A neutral descriptive report may use a balanced projection, whereas a high-stakes intervention may assign greater importance to contextual alignment and epistemic restraint. Second, equality after projection does not imply equality of ALA states. The latter issue is central enough to be treated separately in Section~\ref{sec:projection-equivalence}.

\begin{principle}[Delayed scalar interface]
A scalar projection in ALA is an interface applied to a preserved triadic state. At minimum, it must respect the boundary states and the cognitive order. Strictly admissible projections additionally make every strict cognitive improvement visible at the scalar interface. Projection is not the primitive representation of membership-based judgment.
\end{principle}

The admissibility of a projection is not a matter of arbitrary preference. A scalar interface must respect the semantic geometry established by the cognitive order. Two levels are useful. Weak admissibility covers order-compatible practical interfaces, including thresholded or plateaued readouts. Strict admissibility is the stronger class required by the separation and non-collapse results developed below.

\begin{definition}[Boundary-grounded scalar projection]
A scalar projection \(\phi:\cA\to[0,1]\) is boundary-grounded if
\[
\phi(\bot_{\cA})=0,
\qquad
\phi(\top_{\cA})=1,
\]
where \(\bot_{\cA}=(0,0,1)\) and \(\top_{\cA}=(1,1,0)\).
\end{definition}

\begin{definition}[Weakly admissible scalar projection]
A mapping \(\phi:\cA\to[0,1]\) is weakly admissible if it is continuous, boundary-grounded, and isotone with respect to the ALA cognitive order:
\[
A\preceqA B
\quad\Longrightarrow\quad
\phi(A)\leq\phi(B).
\]
The set of all weakly admissible projections is denoted by
\(\cPhi_{\mathrm{wadm}}\).
\end{definition}

\begin{definition}[Strict isotonicity]
A scalar projection \(\phi:\cA\to[0,1]\) is strictly isotone with respect to the ALA cognitive order if, for every \(A,B\in\cA\),
\[
A\precA B
\quad\Longrightarrow\quad
\phi(A)<\phi(B).
\]
\end{definition}

\begin{definition}[Strictly admissible scalar projection]
A weakly admissible projection is strictly admissible if it is strictly isotone. The set of all strictly admissible scalar projections is denoted by \(\cPhi_{\mathrm{adm}}\). Unless explicitly qualified otherwise, ``admissible projection'' in the remainder of this paper means strictly admissible projection.
\end{definition}

Continuity prevents artificial discontinuities in the scalar interface. Boundary grounding fixes the interpretation of complete absence and complete presence in the order-aligned sense. Weak isotonicity prevents order reversal, while strict isotonicity ensures that every genuine cognitive improvement, with no deterioration in another role, remains visible at the scalar interface.

\subsection{Projection Families and Scalar Readouts}
The most transparent admissible projections are obtained by using the order-aligned coordinate
\[
\rho_A=1-\gamma_A.
\]
In the coordinates \((\theta,\sigma,\rho)\), all three components increase in the direction of cognitive permissiveness. Positive weighted combinations of these order-aligned coordinates therefore provide a natural family of admissible scalar interfaces.

Let
\[
\Delta_+^2=\left\{w=(w_\theta,w_\sigma,w_\rho)\in(0,1)^3:
 w_\theta+w_\sigma+w_\rho=1\right\}
\]
be the relative interior of the three-dimensional probability simplex. For \(w\in\Delta_+^2\), define
\begin{equation}
\phi_w(A)
= w_\theta\theta_A+w_\sigma\sigma_A+w_\rho(1-\gamma_A).
\label{eq:weighted-order-aligned-projection}
\end{equation}

\begin{proposition}[Admissibility of positive weighted order-aligned projections]
For every \(w\in\Delta_+^2\), the mapping \(\phi_w\) defined by Eq.~\eqref{eq:weighted-order-aligned-projection} is an admissible scalar projection.
\end{proposition}

\begin{proof}
Continuity is immediate because \(\phi_w\) is affine in the coordinates of \(A\). Boundary grounding follows from
\[
\phi_w(0,0,1)=0,
\qquad
\phi_w(1,1,0)=w_\theta+w_\sigma+w_\rho=1.
\]
It remains to verify strict isotonicity. Suppose \(A\precA B\). Then
\[
\theta_A\leq\theta_B,
\qquad
\sigma_A\leq\sigma_B,
\qquad
1-\gamma_A\leq1-\gamma_B,
\]
and at least one of these inequalities is strict. Since all weights are strictly positive,
\[
\phi_w(B)-\phi_w(A)
=w_\theta(\theta_B-\theta_A)+w_\sigma(\sigma_B-\sigma_A)+w_\rho\big((1-\gamma_B)-(1-\gamma_A)\big)>0.
\]
Thus \(\phi_w(A)<\phi_w(B)\), and \(\phi_w\) is strictly isotone.
\end{proof}

The affine family in Eq.~\eqref{eq:weighted-order-aligned-projection} is not the only possible admissible family. It is useful because it is interpretable and easy to report. A more general family can be obtained by applying strictly increasing coordinate response functions before aggregation.

\begin{definition}[Response-adjusted admissible projection]
Let \(w\in\Delta_+^2\), and let \(\eta_\theta,\eta_\sigma,\eta_\rho:[0,1]\to[0,1]\) be continuous strictly increasing functions satisfying \(\eta_i(0)=0\) and \(\eta_i(1)=1\). Define
\begin{equation}
\phi_{w,\eta}(A)
= w_\theta\eta_\theta(\theta_A)
+w_\sigma\eta_\sigma(\sigma_A)
+w_\rho\eta_\rho(1-\gamma_A).
\label{eq:response-adjusted-projection}
\end{equation}
Then \(\phi_{w,\eta}\) is called a response-adjusted order-aligned projection.
\end{definition}

\begin{proposition}[Admissibility of response-adjusted projections]
Every response-adjusted order-aligned projection \(\phi_{w,\eta}\) defined by Eq.~\eqref{eq:response-adjusted-projection} is admissible.
\end{proposition}

\begin{proof}
Continuity follows from the continuity of the coordinate response functions. Boundary grounding follows from \(\eta_i(0)=0\), \(\eta_i(1)=1\), and \(w_\theta+w_\sigma+w_\rho=1\). If \(A\precA B\), then the order-aligned coordinates of \(B\) are coordinatewise no smaller than those of \(A\), with at least one strict improvement. Since the response functions are strictly increasing and all weights are positive, the weighted sum strictly increases. Hence \(\phi_{w,\eta}(A)<\phi_{w,\eta}(B)\).
\end{proof}

A canonical neutral choice is the balanced projection
\begin{equation}
\phi_{\mathrm{bal}}(A)
=\frac{\theta_A+\sigma_A+1-\gamma_A}{3}.
\label{eq:balanced-projection}
\end{equation}
This projection assigns equal importance to interpretive valuation, contextual alignment, and the order-aligned restraint coordinate. It is useful when no application-specific weighting policy is available.

Weighted projections encode stated interface policies. For example, a context-sensitive reporting policy may choose \(w_\sigma>w_\theta\), whereas a high-stakes intervention policy may choose a larger \(w_\rho\), because high restraint reduces \(1-\gamma\) and therefore lowers the projected scalar. In this way, the projection can express a conservative interface without redefining the ALA state itself.

\begin{example}[Balanced and conservative projections]
Let
\[
A=(0.80,0.60,0.70),
\qquad
B=(0.70,0.85,0.30).
\]
The balanced projections are
\[
\phi_{\mathrm{bal}}(A)=\frac{0.80+0.60+0.30}{3}=0.5667,
\qquad
\phi_{\mathrm{bal}}(B)=\frac{0.70+0.85+0.70}{3}=0.7500.
\]
Although \(A\) has the larger valuation, its weaker contextual alignment and stronger restraint lower its scalar interface. A more conservative projection with \(w=(0.25,0.25,0.50)\) gives
\[
\phi_w(A)=0.25(0.80)+0.25(0.60)+0.50(0.30)=0.5000,
\]
whereas
\[
\phi_w(B)=0.25(0.70)+0.25(0.85)+0.50(0.70)=0.7375.
\]
The ranking under this interface is not caused by valuation alone. It reflects the stated scalar policy applied to the preserved triadic states.
\end{example}

The pure valuation readout
\[
\phi_\theta(A)=\theta_A
\]
may be useful as a diagnostic display of the first coordinate. It is weakly admissible, because it is continuous, boundary-grounded, and isotone, but it is not strictly admissible: it ignores improvements in contextual alignment and epistemic restraint and therefore fails strict isotonicity.

\begin{proposition}[Weak but not strict admissibility of pure valuation]
The mapping \(\phi_\theta(A)=\theta_A\) is weakly admissible but not strictly admissible on \((\cA,\preceqA)\).
\end{proposition}

\begin{proof}
Consider
\[
A=(0.50,0.20,0.80),
\qquad
B=(0.50,0.90,0.10).
\]
Then \(A\precA B\), because \(\theta_A=\theta_B\), \(\sigma_A<\sigma_B\), and \(\gamma_A>\gamma_B\). However,
\[
\phi_\theta(A)=0.50=\phi_\theta(B).
\]
Thus the strict improvement from \(A\) to \(B\) is invisible under \(\phi_\theta\). Hence \(\phi_\theta\) fails strict isotonicity and is not strictly admissible. Its weak admissibility follows directly from continuity, boundary grounding, and isotonicity in the valuation coordinate.
\end{proof}

This result is central to the distinction between ALA and a scalar-valued fuzzy readout. If only \(\theta\) is retained, then the preserved semantic roles of contextual alignment and epistemic restraint are no longer available before reasoning begins. ALA permits the display of \(\theta\) as one coordinate, but it does not treat \(\theta\) alone as a strictly admissible projection of the whole state.

\subsection{Projection Policy and Restraint Polarity}
Scalar projection is therefore best understood as an interface layer. The ALA state remains the computational and semantic object. The scalar projection is a controlled output used when another layer of analysis demands a number. This separation is important for the non-collapse analysis of the next section. If scalar equality is treated as equality of states, then contextual alignment and epistemic restraint disappear. If scalar equality is treated as projection-induced equivalence, then the hidden triadic structure remains available for further operations, comparisons, and interpretation.

Projection in ALA is therefore delayed, admissible, and declared: delayed because it follows state formation. It is admissible because it respects boundary grounding and strict isotonicity. It is declared because the scalar value has no complete meaning without the projection policy that generated it.

\begin{principle}[Policy-complete scalar reporting]
\label{prin:policy-complete-reporting}
In ALA, a scalar value without its projection policy is not a complete
ALA report. A reported scalar value \(\phi(A)\) is complete only when the
admissible projection that generated it, together with its declared
interface purpose and role-priority parameters, is also reported. Thus, for
a response-adjusted projection, the scalar report should be read as
\[
(T,\phi_{w,\eta},w,\eta,\phi_{w,\eta}(A)),
\]
where \(T\) denotes the declared interface purpose, \(w\) records the
role-priority vector, and \(\eta\) records the coordinate response
functions. The scalar number is therefore an auditable interface output,
not a standalone replacement for the ALA state.
\end{principle}

\begin{remark}[Projection is not aggregation]
\label{rem:projection-not-aggregation}
In ALA, projection and aggregation belong to different layers and carry
different types. A projection is a policy-indexed scalar readout of a single
preserved state,
\[
\phi_\pi:\cA\to[0,1],
\]
whereas an aggregation operator fuses several ALA states into another ALA
state,
\[
G:\cA^n\to\cA.
\]
In particular, the scalar produced by \(\phi_\pi\) is not an aggregate of
the three coordinates under a shared membership budget, and there is no
canonical scalar value intrinsically determined by the ALA state itself. A
given policy may well compute a weighted sum, as \(\phi_w\) does, but this
reflects a declared scalar interface choice, not an intrinsic scalar carried
by the state. Different admissible policies may therefore yield different
readouts of the same preserved state. Treating the scalar as a within-state
aggregate would reintroduce precisely the kind of membership-budget collapse
that the preserved-state architecture of ALA is designed to avoid.
\end{remark}

\begin{proposition}[Order-unanimity of admissible projections]
\label{prop:order-unanimity}
Let \(A,B\in\cA\). Recall that the ALA cognitive order \(\preceq_{\mathcal A}\) is the
product-type order obtained from the order-aligned coordinates
\[
\Psi(A)=(\theta_A,\sigma_A,1-\gamma_A),
\]
and that admissible scalar projections are defined to be isotone with
respect to this same order. If \(A\preceq_{\mathcal A} B\), then
\[
\phi(A)\leq \phi(B)
\]
for every admissible scalar projection \(\phi\in\cPhi_{\mathrm{adm}}\).
If \(A\prec_{\mathcal A} B\), then
\[
\phi(A)<\phi(B)
\]
for every \(\phi\in\cPhi_{\mathrm{adm}}\). Consequently, an admissible
projection cannot reverse a strict cognitive dominance relation.
\end{proposition}

\begin{proof}
If \(A=B\), then \(\phi(A)=\phi(B)\) for every projection. If
\(A\prec_{\mathcal A} B\), then the strict inequality \(\phi(A)<\phi(B)\) follows
directly from strict isotonicity with respect to the ALA cognitive order
in the definition of admissible scalar projection. Therefore
\(A\preceq_{\mathcal A} B\) implies \(\phi(A)\leq\phi(B)\) for every admissible
projection.
\end{proof}

\begin{remark}[Where ALA departs from weighted aggregation]
\label{rem:projection-vs-weighted-aggregation}
Order-unanimity in Proposition~\ref{prop:order-unanimity} is an internal
consistency requirement. By itself, it does not distinguish ALA from
strictly positive weighted aggregation, since such aggregations also
preserve dominance with respect to the relevant product-type order. The
genuine departure lies elsewhere. In ALA, the product-type order is carried
by semantically typed coordinates: valuation, contextual alignment, and
order-aligned restraint. More importantly, incomparable states are an
explicitly represented structural fact, not a defect to be removed.

In standard weighted aggregation, the weights constitute the comparison and
force a scalar ordering, thereby suppressing incomparability at the level of
the reported score. In ALA, the partial cognitive order is primary and prior
to any projection. An admissible projection is a declared and auditable
interface that may present or resolve structurally incomparable states at
the reporting boundary, without altering the underlying ALA state. The
distinction is not that ALA avoids producing a scalar; it does produce
\(\phi_{w,\eta}(A)\). The distinction is that the scalar never replaces or
erases the underlying state, which remains recoverable and continues to
carry the role-level incomparability structure. Projection policy therefore
does not create the primitive comparison relation; it only selects a scalar
presentation among states whose role-level structure has already been
preserved.
\end{remark}

\label{subsec:polarity_signature_restraint}

The structures developed so far show that epistemic restraint is not a relabeled form of confidence: its mathematical signature is reversed operational polarity.

A confidence-like coordinate, when interpreted as assertion-supportive confidence, is positively oriented: increasing confidence should not weaken the cognitive permissiveness of a state when all other coordinates are fixed. Epistemic restraint behaves in the opposite direction. Increasing \(\gamma\) decreases cognitive permissiveness, because stronger restraint limits premature operationalization.

\begin{proposition}[Polarity incompatibility of confidence and epistemic restraint]
Let \(c\in[0,1]\) be interpreted as an assertion-supportive confidence coordinate. For fixed valuation and contextual alignment, increasing such a coordinate should not decrease the cognitive permissiveness of the state. In ALA, however, epistemic restraint has the opposite order behavior:
\[
\gamma_1<\gamma_2
\quad\Longrightarrow\quad
(\theta,\sigma,\gamma_2)\preceqA(\theta,\sigma,\gamma_1).
\]
Therefore, \(\gamma\) cannot be interpreted as assertion-supportive confidence without reversing the semantic polarity of confidence.
\end{proposition}

\begin{proof}
Fix \(\theta,\sigma\in[0,1]\) and let \(\gamma_1<\gamma_2\). By the definition of the ALA cognitive order,
\[
(\theta,\sigma,\gamma_2)\preceqA(\theta,\sigma,\gamma_1),
\]
because the first two coordinates are equal and the restraint coordinate is ordered in the reverse direction. Thus, increasing \(\gamma\) lowers the state in the order of cognitive permissiveness. A coordinate interpreted as assertion-supportive confidence cannot have this behavior while preserving its usual positive orientation. Hence \(\gamma\) is polarity-incompatible with assertion-supportive confidence.
\end{proof}

This proposition does not claim that no mathematical model can define a negatively oriented auxiliary variable. Rather, it states a more precise fact: no coordinate that retains the usual meaning of assertion-supportive confidence can play the role of \(\gamma\) in the ALA cognitive order. If such a coordinate is given reversed polarity, it no longer functions as confidence in the ordinary support-amplifying sense; it becomes a restraint-like or permissiveness-complement coordinate.

The same polarity signature appears in scalar projection. In the positive linear order-aligned projection,
\[
\phi_w(A)
=
w_\theta\theta_A+w_\sigma\sigma_A+w_\rho(1-\gamma_A),
\]
the restraint coordinate contributes through \(1-\gamma_A\), not through \(\gamma_A\) itself. Thus, for fixed \(\theta_A\) and \(\sigma_A\),
\[
\frac{\partial \phi_w}{\partial \gamma_A}=-w_\rho<0,
\]
whenever \(w_\rho>0\). A larger value of \(\gamma_A\) lowers the scalar interface because it indicates stronger restraint. By contrast, an assertion-supportive confidence coordinate would normally enter a scalar support interface with positive polarity.

Hence, epistemic restraint is not defined merely by independence from confidence. It is defined by reversed operational polarity; Table~\ref{tab:restraint-distinction} in Section~7 collects the corresponding semantic contrasts. ALA separates evidential support from permission to operationalize that support: valuation and contextual alignment may be high while restraint is also high. This possibility is not a defect of the framework; it is precisely one of the structural reasons for preserving the triadic state before scalar projection.

\section{Projection Equivalence and Non-Collapse}
\label{sec:projection-equivalence}

Scalar projection has been established as a delayed and admissible interface; this section explains why it cannot replace the triadic state itself. The central point is that equality has several layers in ALA. Two states may be exactly equal as triadic objects, equivalent under a declared scalar projection, indistinguishable under a sufficiently rich family of projections, or incomparable under the cognitive order. These notions are not interchangeable.

This distinction is the formal expression of the principle of reasoning before projection. If a scalar value is treated as the primitive state, then contextual alignment and epistemic restraint may disappear before any operation is performed. If the triadic state is preserved, scalar equality can be recognized as projection-induced equivalence rather than as genuine equality of interpretive states.

\subsection{Equality Layers}
The strongest equality relation in ALA is coordinatewise equality. It is independent of any scalar interface.

\begin{definition}[Exact equality]
For two ALA states
\[
A=(\theta_A,\sigma_A,\gamma_A),
\qquad
B=(\theta_B,\sigma_B,\gamma_B),
\]
we say that \(A\) and \(B\) are exactly equal, written \(A=B\), if and only if
\[
\theta_A=\theta_B,\qquad
\sigma_A=\sigma_B,\qquad
\gamma_A=\gamma_B.
\]
\end{definition}

Exact equality means that the two states have the same interpretive valuation, the same contextual alignment, and the same epistemic restraint. It is therefore the equality relation of the ALA state space itself. Scalar equality under a projection is weaker, because it is possible for changes in one coordinate to compensate changes in another coordinate at the scalar interface.

Let \(\phi\in\cPhi_{\mathrm{adm}}\) be an admissible scalar projection. Projection induces an equivalence relation by identifying states with the same scalar interface value.

\begin{definition}[Projection-induced equivalence]
\label{def:projection-equivalence}
For \(A,B\in\cA\), define
\[
A\sim_{\phi}B
\quad\Longleftrightarrow\quad
\phi(A)=\phi(B).
\]
The equivalence class of \(A\) under \(\phi\) is
\[
[A]_{\phi}=\{B\in\cA:\phi(B)=\phi(A)\}.
\]
\end{definition}

\begin{proposition}[Projection equivalence is an equivalence relation]
For every admissible scalar projection \(\phi\), the relation \(\sim_\phi\) is reflexive, symmetric, and transitive on \(\cA\).
\end{proposition}

\begin{proof}
For every \(A\in\cA\), \(\phi(A)=\phi(A)\), so \(A\sim_\phi A\). If \(A\sim_\phi B\), then \(\phi(A)=\phi(B)\), hence \(\phi(B)=\phi(A)\), so \(B\sim_\phi A\). If \(A\sim_\phi B\) and \(B\sim_\phi C\), then \(\phi(A)=\phi(B)=\phi(C)\), hence \(A\sim_\phi C\).
\end{proof}

Projection-induced equivalence is not an error. It is the expected behavior of a scalar interface. What matters is that ALA does not confuse \(A\sim_\phi B\) with \(A=B\). The equivalence class \([A]_\phi\) is a scalar level set inside the triadic state space, not a set of identical interpretive states.

The following example shows explicitly, under the balanced projection of Eq.~\eqref{eq:balanced-projection}, that scalar equality is not triadic equality.

\begin{example}[Scalar equality without exact equality]
Let
\[
A=(0.80,0.40,0.20),
\qquad
B=(0.60,0.80,0.40).
\]
Under the balanced projection
\[
\phi_{\mathrm{bal}}(A)=\frac{\theta_A+\sigma_A+(1-\gamma_A)}{3},
\]
we obtain
\[
\phi_{\mathrm{bal}}(A)=\frac{0.80+0.40+0.80}{3}=\frac{2}{3},
\]
and
\[
\phi_{\mathrm{bal}}(B)=\frac{0.60+0.80+0.60}{3}=\frac{2}{3}.
\]
Thus \(A\sim_{\phi_{\mathrm{bal}}}B\), but \(A\neq B\). Moreover, \(A\parallelA B\): state \(A\) has higher valuation and lower restraint, whereas state \(B\) has stronger contextual alignment. The scalar interface equalizes them, while the triadic representation preserves the role-level difference.
\end{example}

The example shows that scalar equality may arise by compensation among semantic roles. ALA records the compensation instead of erasing it.

\subsection{Operational Fracturing}
Projection equivalence can be fractured by a structure-preserving ALA operation. This is the key operational reason for preserving the triadic state before projection. Two states that are equal under a scalar interface may respond differently when combined with a probe, context, evidence item, or operation.

\begin{lemma}[Compensatory projection equivalence]
\label{lem:projection-equivalence-existence}
For every admissible projection \(\phi\in\cPhi_{\mathrm{adm}}\), there exist distinct ALA states \(A_1,A_2\in\cA\) such that
\begin{equation}
\phi(A_1)=\phi(A_2).
\label{eq:existence-same-projection}
\end{equation}
\end{lemma}

\begin{proof}
Choose an interior state
\[
A_1=(\theta_1,\sigma_0,\gamma_1)\in(0,1)^3.
\]
Because \(\phi\) is strictly isotone with respect to the cognitive order, it is strictly decreasing in \(\gamma\) when \(\theta\) and \(\sigma\) are fixed. Hence
\[
\phi(\theta_1,\sigma_0,1)<\phi(\theta_1,\sigma_0,\gamma_1)=\phi(A_1).
\]
By continuity, there exists a sufficiently small \(\epsilon>0\) such that \(\theta_2=\theta_1+\epsilon<1\) and
\[
\phi(\theta_2,\sigma_0,1)<\phi(A_1).
\]
On the other hand, strict isotonicity in the valuation coordinate gives
\[
\phi(\theta_2,\sigma_0,\gamma_1)>
\phi(\theta_1,\sigma_0,\gamma_1)=\phi(A_1).
\]
Define
\[
h(\gamma)=\phi(\theta_2,\sigma_0,\gamma),
\qquad \gamma\in[\gamma_1,1].
\]
The function \(h\) is continuous, and the preceding inequalities imply
\[
h(1)<\phi(A_1)<h(\gamma_1).
\]
By the intermediate value theorem, there exists \(\gamma_2\in(\gamma_1,1)\) such that
\[
h(\gamma_2)=\phi(A_1).
\]
Setting
\[
A_2=(\theta_2,\sigma_0,\gamma_2)
\]
gives \(A_1\neq A_2\) and \(\phi(A_1)=\phi(A_2)\).
\end{proof}

The proof constructs scalar equality by compensation: \(A_2\) has a larger valuation than \(A_1\), but it also has stronger epistemic restraint. A scalar projection may balance these two changes. The next lemma shows that this balance is not operationally stable.

\begin{lemma}[Operational fracturing of projection equivalence]
\label{lem:operational-fracturing}
For every admissible projection \(\phi\in\cPhi_{\mathrm{adm}}\), there exist distinct ALA states \(A_1,A_2\in\cA\), a state \(B\in\cA\), and a lattice operation
\[
\star\in\{\vee_{\cA},\wedge_{\cA}\}
\]
such that
\begin{equation}
\phi(A_1)=\phi(A_2),
\qquad
\phi(A_1\star B)\neq \phi(A_2\star B).
\label{eq:fracturing-statement}
\end{equation}
\end{lemma}

\begin{proof}
It is enough to prove the result for one of the two structure-preserving lattice operations. We use \(\star=\vee_{\cA}\). By Lemma~\ref{lem:projection-equivalence-existence}, choose distinct states of the form
\[
A_1=(\theta_1,\sigma_0,\gamma_1),
\qquad
A_2=(\theta_2,\sigma_0,\gamma_2),
\]
with
\[
\theta_2>\theta_1,
\qquad
\gamma_2>\gamma_1,
\qquad
\phi(A_1)=\phi(A_2).
\]
Choose a probe state
\[
B=(\theta_B,\sigma_B,\gamma_B)\in\cA
\]
such that
\[
\theta_B\leq\theta_1,
\qquad
\sigma_B\leq\sigma_0,
\qquad
\gamma_B\leq\gamma_1.
\]
Using Eq.~\eqref{eq:ala-join}, we obtain
\[
A_1\vee_{\cA}B=(\theta_1,\sigma_0,\gamma_B),
\qquad
A_2\vee_{\cA}B=(\theta_2,\sigma_0,\gamma_B).
\]
Thus the join with \(B\) removes the compensating restraint difference by replacing both restraint coordinates with the common value \(\gamma_B\), while preserving the strict valuation difference \(\theta_2>\theta_1\). Therefore
\[
A_1\vee_{\cA}B\precA A_2\vee_{\cA}B.
\]
Since \(\phi\) is strictly isotone,
\[
\phi(A_1\vee_{\cA}B)<\phi(A_2\vee_{\cA}B).
\]
This proves Eq.~\eqref{eq:fracturing-statement}.
\end{proof}

Operational fracturing means that projection equivalence can be broken by an operation that is natural in the ALA state space. The scalar interface alone cannot predict this behavior unless the underlying triadic states are still available.

\subsection{Non-Collapse Theorem}
The preceding lemmas give the central non-collapse result of this work.

\begin{theorem}[ALA non-collapse theorem]
\label{thm:ala-non-collapse}
For every admissible projection \(\phi\in\cPhi_{\mathrm{adm}}\), there exist distinct ALA states \(A_1,A_2\in\cA\), a state \(B\in\cA\), and a structure-preserving lattice operation
\[
\star\in\{\vee_{\cA},\wedge_{\cA}\}
\]
such that
\[
\phi(A_1)=\phi(A_2),
\qquad
\phi(A_1\star B)\neq\phi(A_2\star B).
\]
Consequently, admissible scalar projection does not preserve, in general, all triadic distinctions that may become operationally relevant under ALA operations.
\end{theorem}

\begin{proof}
The statement follows directly from Lemma~\ref{lem:operational-fracturing}. Since \(\phi\in\cPhi_{\mathrm{adm}}\) was arbitrary, the result holds for every admissible projection.
\end{proof}

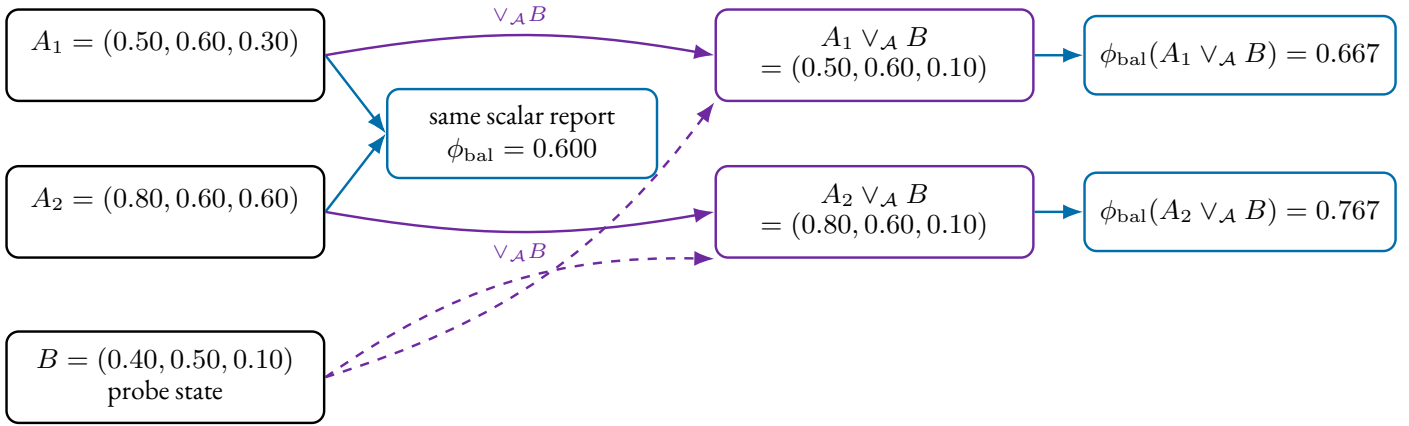
\begin{figure}[!htbp]
\centering
\resizebox{\textwidth}{!}{%
\begin{tikzpicture}[font=\small,>=Latex,node distance=1.15cm]
  \tikzstyle{statebox}=[draw,rounded corners,thick,align=center,minimum width=3.65cm,minimum height=1.05cm,inner sep=5pt]
  \tikzstyle{opbox}=[draw=Cjoin,rounded corners,thick,align=center,minimum width=3.65cm,minimum height=1.05cm,inner sep=5pt]
  \tikzstyle{scalarbox}=[draw=Cproj,rounded corners,thick,align=center,minimum width=3.10cm,minimum height=0.90cm,inner sep=5pt]
  \node[statebox] (Aone) at (0,2.4)
    {$A_1=(0.50,0.60,0.30)$\\[-1pt]
    };
  \node[statebox] (Atwo) at (0,0.6)
    {$A_2=(0.80,0.60,0.60)$\\[-1pt]
     };
  \node[scalarbox] (samephi) at (4.1,1.5)
    {same scalar report\\
\(\phi_{\mathrm{bal}}=0.600\)};
  \draw[->,Cproj,thick] (Aone.east) -- (samephi.west);
  \draw[->,Cproj,thick] (Atwo.east) -- (samephi.west);
  \node[statebox] (probe) at (0,-1.30)
    {$B=(0.40,0.50,0.10)$\\[-1pt]
     probe state};
  \node[opbox] (joinone) at (8.15,2.4)
    {$A_1\joinA B$\\[-1pt]
     $=(0.50,0.60,0.10)$};
  \node[opbox] (jointwo) at (8.15,0.6)
    {$A_2\joinA B$\\[-1pt]
     $=(0.80,0.60,0.10)$};
  \draw[->,Cjoin,thick] (Aone.east) to[bend left=10]
    node[above,font=\scriptsize] {$\joinA B$} (joinone.west);
  \draw[->,Cjoin,thick] (Atwo.east) to[bend right=10]
    node[below,font=\scriptsize] {$\joinA B$} (jointwo.west);
  \draw[->,Cjoin,thick,dashed] (probe.east) to[bend right=18] (joinone.south west);
  \draw[->,Cjoin,thick,dashed] (probe.east) to[bend left=18] (jointwo.south west);
  \node[scalarbox] (outone) at (12.35,2.4)
    {$\phi_{\mathrm{bal}}(A_1\joinA B)=0.667$};
  \node[scalarbox] (outtwo) at (12.35,0.6)
    {$\phi_{\mathrm{bal}}(A_2\joinA B)=0.767$};
  \draw[->,Cproj,thick] (joinone.east) -- (outone.west);
  \draw[->,Cproj,thick] (jointwo.east) -- (outtwo.west);
  
\end{tikzpicture}%
}
\caption{Non-collapse under a structure-preserving ALA operation.}
\label{fig:non-collapse-visual}
\label{fig:non-collapse-operation}
\end{figure}

Fig.~\ref{fig:non-collapse-visual} illustrates the non-collapse mechanism for the balanced projection~\eqref{eq:balanced-projection}. The states \(A_1=(0.50,0.60,0.30)\) and \(A_2=(0.80,0.60,0.60)\) have the same scalar report, \(\phi_{\mathrm{bal}}(A_1)=\phi_{\mathrm{bal}}(A_2)=0.600\), because the larger valuation of \(A_2\) is offset by stronger restraint. After joining both states with \(B=(0.40,0.50,0.10)\), the compensating restraint difference is removed, while the valuation difference remains. Hence the scalar reports separate after the operation.

\begin{remark}[Non-collapse as failure of congruence]
\label{rem:congruence}
Theorem~\ref{thm:ala-non-collapse} is stronger than the observation that a map from \([0,1]^3\) to \([0,1]\) may identify distinct points. The projection-induced equivalence \(\sim_\phi\) of Definition~\ref{def:projection-equivalence} is a congruence for the ALA lattice operations only if \(\phi(A_1)=\phi(A_2)\) implies \(\phi(A_1\star B)=\phi(A_2\star B)\) for every \(B\) and every operation \(\star\). The theorem says exactly that \(\sim_\phi\) is \emph{not} such a congruence: scalar equality is not preserved by ALA lattice operations, so a shared scalar value is operationally unstable rather than merely non-unique. Mechanically, the instability arises because an admissible projection mixes the three roles into a single number and is therefore non-injective, while a structure-preserving operation repartitions the coordinates and need not preserve the fibers of that mixing map; the same algebraic phenomenon would arise for any non-injective role-mixing projection on a product carrier. What gives the result its operational significance is the typing of the coordinates. When an operation exposes a difference that the scalar interface had hidden, it recovers a distinction in a role with fixed meaning, namely support held under a compensating restraint, and this distinction governs how the states behave under subsequent operations. The reversed polarity of \(\gamma\) and the role-preserving form of the operations are not the mechanism of non-collapse; they secure the semantic legitimacy of the operations, so that the fracture recovers an operationally consequential typed distinction rather than an uninterpreted numerical artifact.
\end{remark}

\section{Structural Necessity of the ALA Triad}
\label{sec:structural-necessity}

A remaining question is structural: why should the primitive state contain exactly the three coordinates
\(\theta\), \(\sigma\), and \(\gamma\)? The answer is conditional and role based. ALA adopts, as an adequacy postulate motivated by the counterfactual contrasts of Sections~1 and~3, that valuation, contextual alignment, and epistemic restraint are three distinct roles that may vary independently and should remain directly recoverable before projection. Given that postulate, the triad is not introduced to increase dimensionality for its own sake; it is the minimal role-transparent coordinate structure for preserving those three roles. This view remains compatible with the broader tradition of formal fuzzy and many-valued reasoning while changing the primitive representational unit from a scalar grade to a structured interpretive state \cite{Hajek1998}.

\subsection{Role Transparency and Independent Variability}
\label{subsec:role-preservation}

The central representational requirement of ALA is semantic transparency. A state representation is semantically transparent when each role that is essential to the judgment remains explicitly recoverable from the state. In scalar fuzzy membership, only the valuation-like output is directly visible. In ALA, the state records not only how strongly an element satisfies an interpreted property, but also how well the judgment fits its operative context and how much restraint should accompany premature commitment.

\begin{definition}[Interpretive role]
An interpretive role is a primitive semantic function that contributes to the meaning of a membership-like judgment before scalar reporting. In ALA, the primitive roles are
\[
\Theta \quad \text{valuation}, \qquad
\Sigma \quad \text{contextual alignment}, \qquad
\Gamma \quad \text{epistemic restraint}.
\]
An ALA state \(A\in\cA\) is role-transparent when \(\theta_A\), \(\sigma_A\), and \(\gamma_A\) are separately interpretable as the normalized scores of these three roles.
\end{definition}

\begin{principle}[Role preservation]
A representation of ambiguity should not merge two primitive interpretive roles before the operations that require them have been performed. Merging may be useful at a reporting interface, but it should not occur at the primitive state level.
\end{principle}

This principle is the formal content of reasoning before projection. Projection is permitted, but only after the triadic state has been used for comparison, operations, aggregation, or decision interfacing. If two roles are merged at the state level, no later operation can recover the distinction without adding information from outside the representation.

The three ALA coordinates are not treated as three labels for one latent quantity. The model is required to admit counterfactual variation in any one role while the other two are held fixed. This is an adequacy requirement of the representation, not an empirical claim that every application will realize every point of the full cube.

\begin{definition}[Independent role variability]
The roles \(\Theta\), \(\Sigma\), and \(\Gamma\) are independently variable on \(\cA=[0,1]^3\) if, for each role, there exist two ALA states that differ only in that role. Equivalently, for some fixed admissible values of the remaining coordinates, each coordinate admits nontrivial variation while the other two coordinates are held constant.
\end{definition}

\begin{proposition}[Coordinatewise variability of the normalized realization]
\label{prop:coordinatewise-independent-variability}
The ALA state space \(\cA=[0,1]^3\) satisfies the adopted adequacy requirement of independent role variability: each of valuation, contextual alignment, and epistemic restraint can vary while the other two coordinates are held fixed.
\end{proposition}

\begin{proof}
Choose any values in the open interval \((0,1)\). For valuation, let
\[
A_1=(\theta_1,\sigma,\gamma),\qquad
A_2=(\theta_2,\sigma,\gamma),
\]
with \(\theta_1\ne\theta_2\). Then \(A_1\) and \(A_2\) differ only in interpretive valuation. For contextual alignment, choose
\[
B_1=(\theta,\sigma_1,\gamma),\qquad
B_2=(\theta,\sigma_2,\gamma),
\]
with \(\sigma_1\ne\sigma_2\). For epistemic restraint, choose
\[
C_1=(\theta,\sigma,\gamma_1),\qquad
C_2=(\theta,\sigma,\gamma_2),
\]
with \(\gamma_1\ne\gamma_2\). Hence each role can be varied while the remaining two roles are fixed. This proves independent variability. 
\end{proof} 

Operationally, this independence has two complementary aspects: semantic independence, since each role can vary while the other roles are held fixed, and operational salience, since such variation can affect state-level comparison or operations rather than only a downstream scalar projection. Independent variability has an important consequence. If a representation deletes one coordinate, then it identifies states that ALA distinguishes.

\subsection{Conditional Semantic Minimality and Polarity Necessity}
\label{subsec:semantic-minimality}

The following proposition states the minimality of the triad. The argument is semantic rather than topological. ALA does not require a heavy representation theorem to justify three coordinates. It requires only the observation that three independently variable roles cannot be transparently represented by fewer than three role-specific coordinates.

\begin{proposition}[Conditional semantic minimality of the ALA triad]
\label{thm:semantic-minimality}
Let a membership-like judgment contain three primitive interpretive roles: valuation, contextual alignment, and epistemic restraint. Suppose that a coordinate representation is required to be role-transparent and to support independent variability of the three roles. Then no representation with fewer than three role-specific coordinates can satisfy these requirements. Conditional on these role and transparency assumptions, the ALA state in Eq.~\eqref{eq:ala-state} is semantically minimal among coordinate representations of the three roles.
\end{proposition}

\begin{proof}
Assume, for contradiction, that a role-transparent representation with fewer than three role-specific coordinates exists. Since there are three primitive roles and fewer than three coordinates, the pigeonhole principle implies that at least two roles must be assigned to the same coordinate, or one role must have no coordinate assigned to it. If one role has no coordinate assigned to it, that role is not recoverable from the state, so role transparency fails. If two roles share a coordinate, then a change in one of these roles while the other remains fixed cannot be represented without either changing the shared coordinate, which also changes the second role, or leaving the shared coordinate unchanged, which hides the first change. In both cases independent variability fails. Hence any role-transparent representation supporting the three independent roles requires at least three role-specific coordinates. ALA uses exactly three coordinates, so it is semantically minimal. 
\end{proof}

The result should be read in this explicitly conditional sense. It does not claim that every uncertainty model must have three coordinates, nor does Proposition~\ref{prop:coordinatewise-independent-variability} establish empirical independence by itself. It claims that once the three roles \(\Theta\), \(\Sigma\), and \(\Gamma\) are adopted as distinct, independently variable, and directly recoverable, a transparent coordinate representation requires three role-specific coordinates. A scalar representation can still be produced, but only as a projection of the triadic state.

\begin{proposition}[No monotone scalar role-transparent representation]
\label{prop:no-monotone-scalar-role-transparent}
Let \(\rho_A=1-\gamma_A\) denote the order-aligned permissiveness coordinate.
There do not exist a scalar representation
\[
s:\cA\to[0,1]
\]
and reconstruction maps
\[
F_\theta,F_\sigma,F_\rho:[0,1]\to[0,1],
\]
each monotone, either nondecreasing or nonincreasing, such that, for every
ALA state \(A\in\cA\),
\[
\theta_A=F_\theta(s(A)),\qquad
\sigma_A=F_\sigma(s(A)),\qquad
\rho_A=F_\rho(s(A)).
\]
Consequently, no monotone scalar representation can be role-transparent for
the independently variable ALA roles on the whole state space.
\end{proposition}

\begin{proof}
Suppose that such \(s\) and reconstruction maps exist. Replacing \(s\) by
\(1-s\), if necessary, we may assume that \(F_\theta\) is nondecreasing; the
other reconstruction maps remain monotone, possibly with reversed
orientation. Consider the three ALA states
\[
P=(0.20,0.50,0.50),\qquad
Q=(0.80,0.80,0.50),\qquad
R=(0.80,0.20,0.50),
\]
where the coordinates are ordered as \((\theta,\sigma,\gamma)\). These
states all belong to \(\cA=[0,1]^3\). Since
\[
\theta_P<\theta_Q
\qquad\text{and}\qquad
\theta_P<\theta_R,
\]
and \(F_\theta\) is nondecreasing, it follows that
\[
s(P)<s(Q)
\qquad\text{and}\qquad
s(P)<s(R).
\]

If \(F_\sigma\) is nondecreasing, then \(s(P)<s(R)\) implies
\[
\sigma_P=F_\sigma(s(P))\le F_\sigma(s(R))=\sigma_R,
\]
which contradicts
\[
\sigma_P=0.50>0.20=\sigma_R.
\]
If \(F_\sigma\) is nonincreasing, then \(s(P)<s(Q)\) implies
\[
\sigma_P=F_\sigma(s(P))\ge F_\sigma(s(Q))=\sigma_Q,
\]
which contradicts
\[
\sigma_P=0.50<0.80=\sigma_Q.
\]
Both possible monotone orientations of \(F_\sigma\) lead to a contradiction.
Hence no such scalar role-transparent representation exists.
\end{proof}

Two qualifications explain why monotone reconstruction is required in this statement. First, monotonicity itself is required.
Without an order-coherence condition, the set-theoretic fact that
\([0,1]\) and \([0,1]^3\) have the same cardinality would permit an
order-incoherent scalar code, such as a digit-interleaving code, from which
all three coordinates could be recovered. Such a code is not a meaningful
membership-like scalar grade. Second, both monotone orientations must be
allowed. A statement restricted only to nondecreasing reconstructions would
be too weak, since a scalar may legitimately run against an antagonistic
role. The proposition therefore rules out every monotone orientation, which
is why three witness states are used rather than only two. In fact, the
argument uses only the pair \((\theta,\sigma)\); the obstruction already
arises from any single antagonistically coupled pair of roles, and
\(F_\rho\) plays no role in the contradiction.

The result does not contradict the admissible scalar projections of
Section~\ref{sec:scalar-projection}. A weighted readout such as \(\phi_w\)
is a monotone scalar interface value, but it is not intended to recover all
three roles. It is a projection of a preserved ALA state for a declared
reporting interface, not a role-transparent representation of the state
itself.

The third coordinate has reversed polarity in the cognitive order because it is epistemic restraint. A lower value of \(\gamma\) means less restraint against commitment, whereas a higher value means stronger restraint. Thus, all else being equal, increasing \(\gamma\) does not strengthen a positive judgment. It makes the state more cautious.

\begin{proposition}[Polarity necessity]
\label{thm:polarity-necessity}
Suppose that \(\gamma\) is interpreted as epistemic restraint and that the cognitive order is intended to compare states by increasing cognitive permissiveness. Then the restraint coordinate must enter the order with reversed polarity. That is, for states differing only in restraint,
\[
(\theta,\sigma,\gamma_1) \preceqA (\theta,\sigma,\gamma_2)
\quad \Longleftrightarrow \quad
\gamma_1\ge \gamma_2.
\]
\end{proposition}

\begin{proof}
Fix \(\theta\) and \(\sigma\), and consider two states that differ only in restraint. If \(\gamma_1>\gamma_2\), then the first state carries stronger epistemic restraint than the second. Since restraint is a brake on premature commitment, the first state cannot be more cognitively permissive than the second when valuation and contextual alignment are identical. Therefore, in an order that increases with cognitive permissiveness, the state with larger restraint must be placed lower, not higher. Hence
\[
(\theta,\sigma,\gamma_1) \preceqA (\theta,\sigma,\gamma_2)
\]
whenever \(\gamma_1\ge \gamma_2\). Using the opposite polarity would rank stronger restraint as stronger support, contradicting the semantics of restraint. 
\end{proof}

This result explains why the order-aligned coordinate
\[
\rho=1-\gamma
\]
is useful. The coordinate \(\rho\) is not a new primitive role. It is a technical re-expression of restraint in an order-increasing direction. The primitive semantic coordinate remains \(\gamma\), because the framework is designed to preserve epistemic restraint explicitly.

The reversed polarity is the strongest formal anchor distinguishing \(\gamma\) from a confidence- or support-type coordinate, but its force should be stated carefully. Proposition~\ref{thm:polarity-necessity} is conditional on the semantic reading of the third coordinate as restraint: once \(\gamma\) is taken to measure resistance to commitment, its reversed orientation in a permissiveness order follows. Under the change of variable \(\rho=1-\gamma\), the third coordinate becomes positively oriented, so the orientation by itself is a coordinate convention rather than an absolute feature. What is not a convention is the semantic decision to preserve restraint, rather than permissiveness, as the primitive role, together with the consequence that increasing this primitive lowers permissiveness. The polarity argument is therefore a semantic anchor with a genuine order-theoretic shadow, not a representation-independent impossibility result, and it is presented in that spirit throughout.

\subsection{Boundaries of the Triad}

The coordinate \(\gamma\) is not confidence, reliability, non-membership, hesitation, or indeterminacy, nor a second degree of truth, a degree of falsity, or a measure of missing information. It is the degree to which the judgment should resist premature commitment under the current epistemic and operational conditions.

Confidence usually concerns the strength or stability of support behind a judgment. Reliability usually concerns the credibility of a source or measurement process. Non-membership concerns evidence against belonging. Hesitation concerns the unresolved gap between membership and non-membership in an orthopair setting. Indeterminacy concerns incomplete, inconsistent, or undecided truth status. Epistemic restraint is different from all of these. It can be high even when confidence is high, for example when a medical decision is irreversible or high-stakes. It can also be low when confidence is moderate but the action is reversible, exploratory, or low-cost.

One of the closest philosophical ancestors of this coordinate is Isaac Levi's pragmatist, decision-theoretic account of belief acceptance. In Levi's framework, evidential support alone does not determine whether a proposition should enter the corpus of full belief. Acceptance is a cognitive decision that balances the informational value of an answer against the risk of error; accordingly, even a strongly supported proposition may rationally be withheld from full belief \cite{LeviGamblingTruth1967,LeviEnterprise1980}. Suspension of judgment is therefore a rationally available doxastic stance rather than merely a failure to form a belief. Because full beliefs remain corrigible, an inquirer may also contract a previously accepted belief and return to suspension \cite{LeviMildContraction2004}.

ALA's epistemic restraint extends this line of thought without being reducible to it. Levi's framework regulates doxastic commitment: whether a proposition enters the corpus of full belief. The coordinate \(\gamma\) regulates operational commitment: whether an interpretively supported judgment is permitted to govern action, intervention, classification, or another consequential output. These two normative layers can diverge. In the clinical reading discussed above, \(\theta\) and \(\sigma\) may both be near one, and the interpretive judgment that the lesion is suspicious may even be accepted as a full belief, while \(\gamma\) remains high because the relevant operational policy requires independent confirmation before an irreversible intervention. Restraint here is not doubt about the truth of the interpretation; it is a justified limit on its operational use. Thus \(\gamma\) is neither a confidence degree nor a doxastic acceptance indicator. It belongs to the distinct normative layer captured by the support--permission separation principle: the grounds that justify an interpretation and the conditions that authorize its use need not coincide.

\begin{table}[!htbp]
\centering
\caption{Semantic distinction between epistemic restraint and neighboring notions.}
\label{tab:restraint-distinction}
\fontsize{10}{12}\selectfont
\setlength{\tabcolsep}{3pt}
\renewcommand{\arraystretch}{1.14}
\begin{tabularx}{\textwidth}{>{\raggedright\arraybackslash}p{1.18in}>{\raggedright\arraybackslash}X>{\raggedright\arraybackslash}X}
\toprule
Concept & What it measures & Why it is not \(\gamma\) in ALA \\
\midrule
Confidence & Strength, stability, or subjective certainty of support for a judgment & A highly confident judgment may still require strong restraint when action is risky, irreversible, or ethically costly. \\
Reliability & Trustworthiness of a source, instrument, or information channel & ALA restraint concerns commitment under epistemic and operational conditions, not only the source quality of evidence. \\
Non-membership & Degree to which an element does not belong to a concept & Restraint does not assert the opposite of valuation. It may increase even when valuation is high. \\
Hesitation & Unallocated or unresolved space between membership and non-membership & Restraint is not a residual gap. It is an explicit role governing cautious commitment. \\
Indeterminacy & Incomplete, inconsistent, or undecided truth status & Restraint can arise from stakes and reversibility even when the truth status is comparatively clear. \\
\bottomrule
\end{tabularx}
\end{table}

Table~\ref{tab:restraint-distinction} clarifies why replacing \(\gamma\) by a confidence score would reverse the conceptual direction of the third coordinate. If \(\gamma\) were confidence, larger values would normally support stronger commitment. In ALA, larger values impose greater restraint. This is precisely why the cognitive order uses reversed restraint polarity.

The semantic minimality result shows why fewer than three coordinates are insufficient. It remains to explain why ALA does not introduce a fourth primitive coordinate. The reason is methodological. A primitive coordinate should be added only when it corresponds to a new interpretive role that is independent, non-derivable, and required for the operations of the theory. Otherwise, the additional quantity belongs to one of three places: an elicitation model, a metadata layer, or a task-dependent projection policy.

For example, source reliability may be used to elicit \(\theta\), \(\sigma\), or \(\gamma\), but it is not identical to any one of them and need not be a primitive ALA coordinate. Probability may inform an assessment, but ALA states are not probability triples. Indeterminacy may influence restraint or contextual alignment in a particular application, but it is not the same role as epistemic restraint. Decision attitude may determine a projection operator, but projection is a delayed scalar interface, not part of the primitive state.

\begin{proposition}[No fourth coordinate without a fourth primitive role]
\label{prop:no-fourth-coordinate}
A fourth coordinate is structurally necessary for ALA only if a fourth primitive interpretive role is specified that is independent of valuation, contextual alignment, and epistemic restraint, and that cannot be represented through elicitation, metadata, or projection policy.
\end{proposition}

\begin{proof}
If a proposed fourth quantity is a function of \(\theta\), \(\sigma\), and \(\gamma\), then it is derivable and does not require a primitive coordinate. If it is used only to estimate one of the three coordinates, then it belongs to the elicitation layer. If it controls how a triadic state is reported or ranked, then it belongs to the projection policy. Therefore it becomes a primitive coordinate only when it represents a new role that remains independent of the existing three and is required at the state level. 
\end{proof}

Thus ALA is open to application-specific metadata, but it does not confuse metadata with primitive representation. The triad is minimal for the three roles that define the framework.

\begin{remark}[Semantic triadicity versus carrier complexity]
\label{rem:triadicity-versus-carrier-complexity}
The minimality of the ALA triad concerns the number of primitive judgment roles, not the number of scalar parameters or internal components used to represent a role. An ALA coordinate may itself be set-valued, interval-valued, multicomponent, or function-valued, provided that its native order and rolewise operations are declared. For example, an interval-valued coordinate contains lower and upper bounds, but those bounds remain internal components of one role. Such internal components do not constitute fourth, fifth, or further ALA coordinates unless they are assigned independent primitive roles with their own global behavior in the cognitive order, complement, projection, arithmetic, and aggregation.

Accordingly, saying that epistemic restraint is not confidence does not prohibit an application from using a declared internal qualification when eliciting or representing restraint. The qualification must, however, remain internal to that role and must not silently replace contextual alignment or operational restraint. ALA is therefore semantically triadic while remaining open to structured role carriers that satisfy the compatibility conditions stated in Remark~\ref{rem:carrier-independent-architecture}.
\end{remark}

\begin{remark}[Primitiveness and elicitation]
\label{rem:primitiveness-elicitation}
Primitiveness in ALA is a claim about an irreducible semantic role within the calculus, not about the absence of measurement or elicitation routes. A coordinate may be elicited from domain-specific factors while remaining primitive, provided that those factors do not enter the ALA state space and that the elicited coordinate participates directly in all state-level operations. As established in Section~\ref{subsec:elicitation}, elicitation factors have boundary scope and are discharged once the state is formed, whereas the three coordinates have global operative scope. The structural necessity established here therefore concerns the irreducible roles, not the elicitation routes: it is the irreducibility of valuation, contextual alignment, and epistemic restraint as distinct semantic roles that requires three coordinates, independently of how any particular coordinate value is obtained in a given domain.
\end{remark}

\begin{remark}[Why elicitation factors are not coordinates]
\label{rem:factors-not-coordinates}
The same criterion settles the mirror-image question of Proposition~\ref{prop:no-fourth-coordinate}: if restraint is elicited from factors such as loss severity and irreversibility, why are those factors not themselves the primitive coordinates, with \(\gamma\) derived? Because primitiveness in ALA is not computational independence. A quantity is a primitive coordinate of ALA when it carries a distinct structural function in the calculus, namely its own position in the cognitive order, its own operational polarity, and its own behavior under complement, projection, and aggregation, and not merely when it is conceptually distinct from other quantities. The factors that inform \(\gamma\) are heterogeneous reasons that all bear on a single role: each pushes restraint in the same direction, each enters the calculus only through its effect on \(\gamma\), and none has a separate position in the order, a separate polarity, or a separate complement behavior. They are therefore factors of one role, not several roles. Valuation, contextual alignment, and restraint, by contrast, each carry a distinct structural function; in particular, restraint alone carries the reversed polarity of Proposition~\ref{thm:polarity-necessity}, which no content-directed or context-directed factor possesses. This is why they are three coordinates rather than one elaborated coordinate or a longer list of sub-coordinates. The criterion that excludes a spurious fourth coordinate in Proposition~\ref{prop:no-fourth-coordinate} is thus the same criterion that forbids promoting an elicitation factor to a coordinate: role-distinctness within the calculus, not computational independence. Read this way, the structural-necessity argument and the elicitation discipline are not in tension; they are two applications of one criterion, and the question ``why this coordinate and not its sub-factors?'' is answered by observing that the sub-factors are not roles.
\end{remark}

\section{Role-Preserving Arithmetic Operations}
\label{sec:arithmetic}
\label{sec:arithmetic-core}

This section develops the first operational layer of the theory: arithmetic operations that preserve the typed meaning of the coordinates of an ALA state \(A\in\cA\), rather than the ordinary arithmetic of real numbers imposed on \([0,1]^3\). The valuation and contextual alignment coordinates are positively oriented. Larger values of \(\theta\) and \(\sigma\) support stronger cognitive dominance. The restraint coordinate is negatively oriented. Larger values of \(\gamma\) indicate stronger epistemic restraint and therefore weaker cognitive permissiveness. ALA arithmetic must respect this polarity. Supplementary diagnostics explaining why projection equalization and coordinate deletion weaken this operational layer are collected in Appendix~\hyperref[app:diagnostic-details]{A}.

\begin{principle}[Role-local calculus under a global cognitive policy]
\label{prin:role-local-calculus}
Each primitive ALA role may carry its own admissible order, arithmetic, and aggregation calculus. For heterogeneous carriers, an operation is performed locally within each role and assembled through the typed product:
\[
F_{\mathcal A^{\star}}
(A_1^{\star},\ldots,A_n^{\star})
=
\bigl(
F_\theta(\theta_1,\ldots,\theta_n),
F_\sigma(\sigma_1,\ldots,\sigma_n),
F_\gamma(\gamma_1,\ldots,\gamma_n)
\bigr).
\]
A common carrier or embedding between different roles is neither required nor semantically desirable before a declared cross-role interface is imposed. What must remain common is the higher-level ALA policy: a permissive construction moves valuation and contextual alignment upward while moving restraint downward, whereas a conservative construction moves valuation and contextual alignment downward while moving restraint upward. ALA therefore requires role-equivalent operations, not formula-identical operations.
\end{principle}

For example, if valuation is represented by one bounded ordered carrier, contextual alignment by intervals, and restraint by another declared ordered carrier, each coordinate is processed by the native arithmetic of its own role. The output remains heterogeneous and role typed until the rolewise standardization of Eq.~\eqref{eq:rolewise-standardization} is deliberately applied. Cross-role heterogeneity requires no common carrier. If heterogeneous forms occur within the same role across different inputs, however, a declared common super-carrier, embedding, or compatibility rule is needed so that the native operation on that role is mathematically defined.

\subsection{Operational Connectors}
\label{subsec:arithmetic-design}

An operation on ALA states is acceptable only if it respects the semantic typing of the state. The output valuation must be obtained from valuation inputs, the output contextual alignment from contextual-alignment inputs, and the output epistemic restraint from restraint inputs. This requirement prevents semantic mixing between concept valuation, contextual fit, and cautious commitment.

\begin{principle}[Role-preserving arithmetic]
\label{principle:role-preserving-arithmetic}
An arithmetic operation on ALA states is role-preserving if the output \(\theta\)-coordinate depends only on input \(\theta\)-coordinates, the output \(\sigma\)-coordinate depends only on input \(\sigma\)-coordinates, and the output \(\gamma\)-coordinate depends only on input \(\gamma\)-coordinates.
\end{principle}

\begin{principle}[Polarity consistency]
\label{principle:polarity-consistency}
An arithmetic operation on ALA states is polarity-consistent if it treats \(\theta\) and \(\sigma\) as positively oriented coordinates and \(\gamma\) as a reversed restraint coordinate. Thus a permissive accumulation should increase valuation and contextual alignment while decreasing evidence-responsive restraint, whereas a conservative fusion should restrict valuation and contextual alignment while increasing restraint.
\end{principle}

\begin{remark}[Non-removable restraint floors under a fixed commitment interface]
\label{rem:restraint-floor}
Some sources of restraint, such as irreversibility, legal prohibition, or a fixed ethical burden, do not disappear when evidence accumulates. For a fixed commitment interface \(p\), let \(\underline{\gamma}_p\in[0,1)\) denote a declared non-removable restraint floor and restrict the third coordinate to \([\underline{\gamma}_p,1]\). Define the order isomorphism
\[
r_p(x)=\frac{x-\underline{\gamma}_p}{1-\underline{\gamma}_p}
\]
and the floor-guarded connectors
\[
H_p(x,y)
=
\underline{\gamma}_p+
(1-\underline{\gamma}_p)
H\bigl(r_p(x),r_p(y)\bigr),
\]
\[
S_p(x,y)
=
\underline{\gamma}_p+
(1-\underline{\gamma}_p)
S\bigl(r_p(x),r_p(y)\bigr).
\]
Then \(H_p(x,y)\geq\underline{\gamma}_p\) and \(S_p(x,y)\geq\underline{\gamma}_p\). The first relaxes only the evidence-responsive component of restraint and can never cross the declared floor; the second accumulates restraint above that floor. All order and monoid properties below transport to the fixed-interface slice through \(r_p\). To avoid notation overload, the formulas that follow present the canonical zero-floor realization \(\underline{\gamma}_p=0\). In a high-stakes application with a nonzero floor, the corresponding floor-guarded connectors should be used, or equivalently the arithmetic should be applied to the normalized residual restraint and then mapped back.
\end{remark}

These principles lead to two complementary arithmetic attitudes. The first is permissive: it accumulates supportive valuation and contextual alignment and relaxes restraint when evidence is accumulated. The second is conservative: it restricts valuation and contextual alignment and accumulates restraint under conjunctive or safety-sensitive use. The next subsection introduces the scalar connectors used to implement these two attitudes.

Two standard scalar connectors are used on \([0,1]\). The first is the algebraic sum
\begin{equation}
S(x,y)=x+y-xy=1-(1-x)(1-y).
\label{eq:algebraic-sum}
\end{equation}
The second is the Hamacher product with parameter zero,
\begin{equation}
H(x,y)=
\begin{cases}
\dfrac{xy}{x+y-xy}, & (x,y)\ne(0,0),\\[7pt]
0, & (x,y)=(0,0).
\end{cases}
\label{eq:hamacher-product}
\end{equation}
The convention \(H(0,0)=0\) gives the natural continuous boundary value at the origin.

\begin{lemma}[Basic properties of \(S\) and \(H\)]
\label{lem:basic-S-H}
For all \(x,y,z\in[0,1]\), the functions \(S\) and \(H\) are closed on \([0,1]\), commutative, associative, and nondecreasing in each argument. Moreover,
\[
S(x,0)=x,\qquad S(x,1)=1,
\]
and
\[
H(x,1)=x,\qquad H(x,0)=0.
\]
For all \(x,y\in[0,1]\),
\[
xy\le H(x,y)\le \min\{x,y\}\le \max\{x,y\}\le S(x,y).
\]
\end{lemma}

\begin{proof}
The properties of \(S\) follow from Eq.~\eqref{eq:algebraic-sum}. For \(H\), closure follows from \(0\le xy\le x+y-xy\) whenever \((x,y)\ne(0,0)\). The boundary identities follow directly from Eq.~\eqref{eq:hamacher-product}. Commutativity is immediate. For positive \(x,y\), the reciprocal identity
\[
\frac{1}{H(x,y)}=\frac{1}{x}+\frac{1}{y}-1
\]
implies associativity; the boundary cases follow from \(H(0,x)=0\). Monotonicity is obtained by differentiating the positive-domain expression or by using the reciprocal representation. The inequality chain follows from the standard bounds for the algebraic sum and the Hamacher product on \([0,1]\).
\end{proof}

\begin{lemma}[N-ary identities]
\label{lem:nary-identities}
Let \(S_n\) and \(H_n\) denote repeated applications of \(S\) and \(H\), respectively. Then
\begin{equation}
S_n(x_1,\ldots,x_n)=1-\prod_{i=1}^{n}(1-x_i).
\label{eq:nary-S}
\end{equation}
If all \(x_i>0\), then
\begin{equation}
H_n(x_1,\ldots,x_n)^{-1}=\sum_{i=1}^{n}x_i^{-1}-(n-1).
\label{eq:nary-H}
\end{equation}
If at least one \(x_i=0\), then \(H_n(x_1,\ldots,x_n)=0\).
\end{lemma}

\begin{proof}
Eq.~\eqref{eq:nary-S} follows by induction from \(S(x,y)=1-(1-x)(1-y)\). Eq.~\eqref{eq:nary-H} follows by induction from the reciprocal identity for \(H\). The zero case follows from \(H(0,x)=0\).
\end{proof}

Three requirements select the pair \((S,H)\) from the large family of triangular norms and conorms. First, strict monotonicity on the open cube, so that accumulating further support strictly rewards it and conjunctive combination strictly reflects every component; this excludes the idempotent \(\min\) and \(\max\), which are insensitive to repeated evidence. Second, smoothness on the open cube, which excludes the nilpotent Łukasiewicz operators with their kink and their zero-divisor region. Third, the absence of zero divisors for the accumulative connector, so that accumulation does not saturate prematurely except at the boundary. The algebraic sum \(S\) satisfies these requirements and is the standard strict, smooth, Archimedean conorm; for the standard theory of triangular norms and related conorms, including Archimedean and nilpotent classes, see \cite{KlementMesiarPap2000}. For the conjunctive connector, Lemma~\ref{lem:basic-S-H} shows that the Hamacher product satisfies

\[
xy\le H(x,y)\le \min\{x,y\},
\]
so \(H\) lies between the ordinary product and the minimum. Because \(xy\leq H(x,y)\), the ordinary product is the more punitive of the two; because \(H(x,y)\leq\min\{x,y\}\), the Hamacher product is more restrictive than the idempotent minimum. It therefore supplies a smooth conjunctive compromise: stricter than minimum-based filtering, but less severe than repeated ordinary multiplication. The pair \((S,H)\) is consequently adopted as a transparent strict-and-smooth permissive--conservative pair for the normalized realization.

The pair is deliberately not De Morgan matched: the algebraic sum is the probabilistic sum, whose De Morgan dual is the ordinary product rather than the Hamacher product. In a Boolean setting this would be an inconsistency; in ALA it is a structural consequence. The interpretive complement of Section~\ref{sec:cognitive-order} acts only on valuation and is not a Boolean complement, so the De Morgan laws are not expected to hold across all coordinates in the first place. The very feature that individuates ALA, a non-Boolean complement that reverses only the truth-direction, removes the constraint that would otherwise force the conjunctive and accumulative connectors to be a dual pair. The non-duality of \((S,H)\) is thus consistent with the architecture rather than in tension with it.

Finally, the present paper fixes \((S,H)\) as a canonical concrete choice rather than as the only admissible pair. Closure and cognitive monotonicity extend to broader connector families under suitable boundary, associativity, and monotonicity assumptions. The non-collapse theorem is independent of this arithmetic choice because it is established for the ALA lattice operations. By contrast, the closed aggregation formulas and the aggregation-envelope inequality of Section~\ref{sec:aggregation} are proved here specifically for the pair \((S,H)\). Extending that envelope to other parameterized families requires separate sufficient conditions and remains outside the present scope.

\subsection{The Arithmetic Operations}
The operations below act on preserved coordinates because each coordinate carries a distinct role in later computation. The role-collapse diagnostics in Appendix~\hyperref[app:diagnostic-details]{A} give the corresponding loss-of-role interpretation when a coordinate is deleted or prematurely scalarized.

\begin{definition}[Interpretive addition]
\label{def:interpretive-addition}
For \(A,B\in\cA\), the interpretive addition of \(A\) and \(B\) is defined by
\[
A\oplus B
=
\big(S(\theta_A,\theta_B),\,S(\sigma_A,\sigma_B),\,H(\gamma_A,\gamma_B)\big).
\]
\end{definition}

The operation \(\oplus\) is permissive. It accumulates valuation and contextual alignment through \(S\), while it combines epistemic restraint through \(H\). Since \(H(x,y)\le \min\{x,y\}\), repeated compatible accumulation can reduce restraint. This is not a claim that caution disappears; it is a formal way of representing the relaxation of restraint under accumulated support.

\begin{remark}[Restraint under accumulated evidence]
\label{rem:restraint-accumulation}
The relaxation of restraint in \(\oplus\) must be read within a fixed commitment interface. Corroborating judgments may reduce the evidence-responsive reason for withholding, but they do not reduce a non-removable floor generated by irreversibility, legal constraint, loss severity, or ethical burden. Accordingly, the zero-floor formula in Definition~\ref{def:interpretive-addition} is appropriate when \(\gamma\) is already the normalized residual restraint for a common interface, or when the interface has no positive restraint floor. If a floor \(\underline{\gamma}_p>0\) is declared, the third coordinate must be combined through \(H_p\) from Remark~\ref{rem:restraint-floor}. Fusion across different commitment interfaces is not a task for \(\oplus\); it belongs to a higher-level aggregation or projection policy in which the interface difference is explicitly declared.
\end{remark}

\begin{proposition}[Elementary properties of \(\oplus\)]
\label{prop:addition-properties}
For all \(A,B,C\in\cA\), the operation \(\oplus\) is closed, commutative, and associative. Its identity element is \(\bot=(0,0,1)\), and \(\top=(1,1,0)\) is absorbing:
\[
A\oplus\bot=A,\qquad A\oplus\top=\top.
\]
\end{proposition}

\begin{proof}
The proof is coordinatewise. Closure, commutativity, and associativity follow from Lemma~\ref{lem:basic-S-H}. The identity follows from \(S(x,0)=x\) and \(H(x,1)=x\). The absorbing property follows from \(S(x,1)=1\) and \(H(x,0)=0\).
\end{proof}

\begin{proposition}[Cognitive monotonicity of \(\oplus\)]
\label{prop:addition-monotonicity}
If \(A\preceqA B\), then for every \(C\in\cA\),
\begin{equation}
A\oplus C\preceqA B\oplus C.
\label{eq:addition-monotone-one}
\end{equation}
More generally, if \(A_i\preceqA B_i\) for \(i=1,2\), then
\[
A_1\oplus A_2\preceqA B_1\oplus B_2.
\]
\end{proposition}

\begin{proof}
If \(A\preceqA B\), then \(\theta_A\le\theta_B\), \(\sigma_A\le\sigma_B\), and \(\gamma_A\ge\gamma_B\). Since \(S\) and \(H\) are nondecreasing in each argument,
\[
S(\theta_A,\theta_C)\le S(\theta_B,\theta_C),\qquad
S(\sigma_A,\sigma_C)\le S(\sigma_B,\sigma_C),
\]
and
\[
H(\gamma_A,\gamma_C)\ge H(\gamma_B,\gamma_C).
\]
The last inequality is the correct direction for the cognitive order because larger \(\gamma\) means greater restraint. Therefore Eq.~\eqref{eq:addition-monotone-one} follows. The two-input statement is identical.
\end{proof}

\begin{definition}[Interpretive multiplication]
\label{def:interpretive-multiplication}
For \(A,B\in\cA\), the interpretive multiplication of \(A\) and \(B\) is defined by
\[
A\otimes B
=
\big(H(\theta_A,\theta_B),\,H(\sigma_A,\sigma_B),\,S(\gamma_A,\gamma_B)\big).
\]
\end{definition}

The operation \(\otimes\) is conservative. It restricts valuation and contextual alignment through \(H\), while it accumulates restraint through \(S\). It is therefore suited to conjunctive fusion, safety-sensitive filtering, and situations in which failure of one component should reduce the overall interpretive support.

\begin{proposition}[Elementary properties of \(\otimes\)]
\label{prop:multiplication-properties}
For all \(A,B,C\in\cA\), the operation \(\otimes\) is closed, commutative, and associative. Its identity element is \(\top=(1,1,0)\), and \(\bot=(0,0,1)\) is absorbing:
\[
A\otimes\top=A,\qquad A\otimes\bot=\bot.
\]
\end{proposition}

\begin{proof}
Each property follows coordinatewise from Lemma~\ref{lem:basic-S-H}. The identity follows from \(H(x,1)=x\) and \(S(x,0)=x\). The absorbing property follows from \(H(x,0)=0\) and \(S(x,1)=1\).
\end{proof}

\begin{proposition}[Cognitive monotonicity of \(\otimes\)]
\label{prop:multiplication-monotonicity}
If \(A\preceqA B\), then for every \(C\in\cA\),
\[
A\otimes C\preceqA B\otimes C.
\]
More generally, if \(A_i\preceqA B_i\) for \(i=1,2\), then
\[
A_1\otimes A_2\preceqA B_1\otimes B_2.
\]
\end{proposition}

\begin{proof}
The proof is the same as the proof of Proposition~\ref{prop:addition-monotonicity}, using \(H\) in the valuation and contextual coordinates and \(S\) in the restraint coordinate.
\end{proof}

The arithmetic core continues the polarity signature of epistemic restraint. In the canonical zero-floor realization, interpretive addition accumulates valuation and contextual alignment through the algebraic sum and combines residual restraint through the Hamacher product; supportive accumulation can therefore reduce the evidence-responsive restraint barrier. Interpretive multiplication restricts valuation and contextual alignment through the Hamacher product and accumulates restraint through the algebraic sum. With a declared non-removable floor, the guarded connectors of Remark~\ref{rem:restraint-floor} preserve the same role directions without allowing total restraint to fall below the floor. This dual arithmetic behavior is consistent with \(\gamma\) as a permission-limiting coordinate rather than an assertion-supportive confidence coordinate.

Repeated interpretive addition has a closed form. This allows scalar multiplication to be defined for all nonnegative real weights, not only for positive integers.

\begin{definition}[Scalar multiplication]
\label{def:scalar-multiplication}
For \(\lambda>0\) and \(A\in\cA\), define
\begin{equation}
\lambda\odot A
=
\left(
1-(1-\theta_A)^\lambda,
\;1-(1-\sigma_A)^\lambda,
\;\frac{\gamma_A}{\lambda-(\lambda-1)\gamma_A}
\right).
\label{eq:scalar-multiplication}
\end{equation}
For \(\lambda=0\), set
\[
0\odot A=\bot.
\]
\end{definition}

For an integer \(k\ge1\), Eq.~\eqref{eq:scalar-multiplication} coincides with the repeated addition of \(A\) with itself \(k\) times. For noninteger \(\lambda\), it gives the continuous extension generated by the algebraic sum in the positive coordinates and the Hamacher product in the restraint coordinate.

\begin{proposition}[Closure and monotonicity of scalar multiplication]
\label{prop:scalar-multiplication-closure}
For every \(A\in\cA\) and \(\lambda\ge0\), \(\lambda\odot A\in\cA\). If \(A\preceqA B\), then
\[
\lambda\odot A\preceqA \lambda\odot B
\]
for every \(\lambda\ge0\).
\end{proposition}

\begin{proof}
For \(x\in[0,1]\) and \(\lambda>0\), the expression \(1-(1-x)^\lambda\) belongs to \([0,1]\). Also,
\[
\lambda-(\lambda-1)x=\lambda(1-x)+x>0
\]
for \(x\in[0,1]\) and \(\lambda>0\), and therefore \(x/[\lambda-(\lambda-1)x]\in[0,1]\). The case \(\lambda=0\) is defined as \(\bot\). Monotonicity follows because the first two coordinate functions are nondecreasing in \(\theta\) and \(\sigma\), while the restraint coordinate function is nondecreasing in \(\gamma\); the reversed order in \(\gamma\) is therefore preserved.
\end{proof}

Repeated interpretive multiplication also has a closed form. It yields a power operation that is conservative in the positive coordinates and accumulative in the restraint coordinate.

\begin{definition}[Interpretive power]
\label{def:interpretive-power}
For \(p>0\) and \(A\in\cA\), define
\begin{equation}
A^p
=
\left(
\frac{\theta_A}{p-(p-1)\theta_A},
\;\frac{\sigma_A}{p-(p-1)\sigma_A},
\;1-(1-\gamma_A)^p
\right).
\label{eq:interpretive-power}
\end{equation}
For \(p=0\), set
\[
A^0=\top.
\]
\end{definition}

For an integer \(k\ge1\), \(A^k\) is the repeated interpretive multiplication of \(A\) with itself \(k\) times. The formula extends this behavior to positive real powers.

\begin{proposition}[Closure and monotonicity of powers]
\label{prop:powers-closure}
For every \(A\in\cA\) and \(p\ge0\), \(A^p\in\cA\). If \(A\preceqA B\), then
\[
A^p\preceqA B^p
\]
for every \(p\ge0\).
\end{proposition}

\begin{proof}
The proof is analogous to Proposition~\ref{prop:scalar-multiplication-closure}. For \(x\in[0,1]\) and \(p>0\), \(x/[p-(p-1)x]\in[0,1]\), and \(1-(1-x)^p\in[0,1]\). Monotonicity follows from the monotonicity of these scalar functions and the reversed interpretation of the restraint coordinate in \(\preceqA\). The case \(p=0\) is defined as \(\top\).
\end{proof}

\subsection{Algebraic Structure}
\label{subsec:algebraic-properties}

The following theorem summarizes the algebraic structure of the total arithmetic operations.

\begin{theorem}[Commutative monoid structures]
\label{thm:commutative-monoids}
The triples \((\cA,\oplus,\bot)\) and \((\cA,\otimes,\top)\) are commutative monoids. Moreover, \(\top\) is absorbing for \(\oplus\), and \(\bot\) is absorbing for \(\otimes\).
\end{theorem}

\begin{proof}
Closure, commutativity, associativity, identity, and absorber properties follow from Propositions~\ref{prop:addition-properties} and \ref{prop:multiplication-properties}.
\end{proof}

\begin{theorem}[Compatibility with the cognitive order]
\label{thm:arithmetic-order-compatibility}
The operations \(\oplus\), \(\otimes\), scalar multiplication, and powers are compatible with the cognitive order. In particular, if \(A_i\preceqA B_i\) for all involved inputs, then applying the same operation to the \(A_i\)'s gives a result cognitively below the result obtained from the \(B_i\)'s.
\end{theorem}

\begin{proof}
The binary cases follow from Propositions~\ref{prop:addition-monotonicity} and \ref{prop:multiplication-monotonicity}. The scalar and power cases follow from Propositions~\ref{prop:scalar-multiplication-closure} and \ref{prop:powers-closure}. The general finite-input case follows by induction.
\end{proof}

\begin{table}[!htbp]
\centering
\caption{Summary of structure-preserving ALA arithmetic operations.}
\label{tab:ala-arithmetic-summary}
\fontsize{10}{12}\selectfont
\setlength{\tabcolsep}{2pt}
\renewcommand{\arraystretch}{1.15}
\begin{tabularx}{\textwidth}{>{\raggedright\arraybackslash}p{0.62in}>{\raggedright\arraybackslash}X>{\raggedright\arraybackslash}X>{\raggedright\arraybackslash}X>{\raggedright\arraybackslash}p{0.72in}}
\toprule
Operation & \(\theta\)-coordinate & \(\sigma\)-coordinate & \(\gamma\)-coordinate & Role \\
\midrule
\(A\oplus B\) & \(S(\theta_A,\theta_B)\) & \(S(\sigma_A,\sigma_B)\) & \(H(\gamma_A,\gamma_B)\) & Permissive \\
\(A\otimes B\) & \(H(\theta_A,\theta_B)\) & \(H(\sigma_A,\sigma_B)\) & \(S(\gamma_A,\gamma_B)\) & Conservative \\
\(\lambda\odot A\) & \(1-(1-\theta_A)^\lambda\) & \(1-(1-\sigma_A)^\lambda\) & \(\dfrac{\gamma_A}{\lambda-(\lambda-1)\gamma_A}\) & Repeated \(\oplus\) \\
\(A^p\) & \(\dfrac{\theta_A}{p-(p-1)\theta_A}\) & \(\dfrac{\sigma_A}{p-(p-1)\sigma_A}\) & \(1-(1-\gamma_A)^p\) & Repeated \(\otimes\) \\

\bottomrule
\end{tabularx}
\end{table}

\section{Role-Preserving Aggregation and the Aggregation Envelope}
\label{sec:aggregation}

The arithmetic core of Section~\ref{sec:arithmetic-core} is evidence-level and therefore non-idempotent. Aggregation has a different purpose. It fuses a finite family of ALA states into a single representative state while preserving the three semantic roles and while respecting normalized weights. Classical ordered weighted averaging and aggregation-function theory provide important precedents for formal aggregation design, but ALA uses aggregation to preserve role-separated ambiguity rather than to collapse it prematurely \cite{Yager1988,BeliakovPraderaCalvo2007}. Hence aggregation must be role-preserving, cognitively monotone, and idempotent. If all inputs are the same ALA state, the aggregate must return that state.

This section introduces the first aggregation layer of the normalized ALA realization. Two Hamacher-type normalized operators are defined. The first is permissive and is obtained from weighted interpretive addition; the second is conservative and is obtained from weighted interpretive multiplication. Their relationship yields an aggregation envelope for the canonical pair \((S,H)\), which is the endpoint of the foundational operational core developed in this paper.

Unlike scalar projection in Section~\ref{sec:scalar-projection}, which maps a single preserved ALA state to a scalar interface value, the aggregation operators in this section fuse multiple ALA states into a new ALA state. Thus projection has type \(\cA\to[0,1]\), whereas aggregation has type \(\cA^n\to\cA\).

\begin{remark}[Heterogeneous aggregation envelopes]
\label{rem:heterogeneous-aggregation-envelope}
Principle~\ref{prin:role-local-calculus} suggests how an aggregation envelope may be extended beyond the unit cube. For inputs \(A_i^{\star}\in\mathcal A^{\star}\), each role carrier would need a declared conservative and permissive aggregator satisfying the native order relation
\[
G_{r,\mathbf w}^{-}
\leq_r
G_{r,\mathbf w}^{+}
\]
for positively oriented roles, with the corresponding reversed direction for restraint. The role-typed endpoints would then satisfy
\[
A_{\mathbf w}^{\star,-}
\preceq_{\mathcal A^{\star}}
A_{\mathbf w}^{\star,+}.
\]
This is an architectural template, not a theorem for every structured carrier. A carrier-specific extension requires native orders, boundary conventions, rolewise aggregation formulas, and sufficient conditions ensuring the endpoint inequality. If a rolewise standardization \(\mathcal R\) is isotone, a valid native envelope is carried to an ordered standard envelope. The operators below establish these properties completely only for the normalized realization \([0,1]^3\) with the canonical pair \((S,H)\).
\end{remark}

\subsection{Weighted Aggregation Operators}
\label{subsec:arithmetic-to-aggregation}

Let
\[
A_i=(\theta_i,\sigma_i,\gamma_i)\in\cA,\qquad i=1,\ldots,n,
\]
be a finite family of ALA states, and let
\begin{equation}
\mathbf{w}=(w_1,\ldots,w_n),
\qquad
w_i>0,
\qquad
\sum_{i=1}^{n}w_i=1,
\label{eq:weight-vector}
\end{equation}
be a normalized positive weight vector. The weights may represent source importance, expert reliability at the source level, criterion importance, or observation relevance. They are not additional ALA coordinates. They specify how much each input state contributes to the fusion process. The formulas below use the canonical zero-floor restraint calculus. When a fixed commitment interface has a declared non-removable floor \(\underline{\gamma}_p>0\), the weighted construction is applied to the normalized residual restraint, or equivalently built from the floor-guarded connectors of Remark~\ref{rem:restraint-floor}.

The distinction between arithmetic and aggregation is important. Arithmetic operations such as \(\oplus\) and \(\otimes\) model accumulation and restriction at the evidence level. Aggregation normalizes these operations by weights. This normalization is what restores idempotency. Repeating the same state as evidence may change the arithmetic result, but aggregating identical weighted inputs must not change the state.

\begin{definition}[Hamacher-type ALA weighted permissive mean]
\label{def:ala-wam}
The Hamacher-type ALA weighted permissive mean, retained under the abbreviation ALA-WAM, of \(A_1,\ldots,A_n\) with respect to \(\mathbf{w}\) is defined by
\[
A_{\mathbf{w}}^{+}
=
\operatorname{WAM}_{\cA,\mathbf{w}}(A_1,\ldots,A_n)
=
\bigoplus_{i=1}^{n}(w_i\odot A_i).
\]
\end{definition}

The superscript \(+\) indicates the permissive side of the aggregation envelope. The operator first scales each input state by its weight through \(w_i\odot A_i\), then fuses the weighted states by interpretive addition. Thus valuation and contextual alignment are accumulated in a permissive manner, while epistemic restraint is combined through the Hamacher component inherited from \(\oplus\).

\begin{definition}[Hamacher-type ALA weighted conservative mean]
\label{def:ala-wgm}
The Hamacher-type ALA weighted conservative mean, retained under the abbreviation ALA-WGM, of \(A_1,\ldots,A_n\) with respect to \(\mathbf{w}\) is defined by
\[
A_{\mathbf{w}}^{-}
=
\operatorname{WGM}_{\cA,\mathbf{w}}(A_1,\ldots,A_n)
=
\bigotimes_{i=1}^{n} A_i^{w_i}.
\]
\end{definition}

The superscript \(-\) indicates the conservative side of the aggregation envelope. Each input is first transformed by its weighted interpretive power, and the resulting states are fused by interpretive multiplication. Thus valuation and contextual alignment are restricted conservatively, while epistemic restraint is accumulated through the algebraic-sum component inherited from \(\otimes\).

\subsection{Closed Forms and Properties}
\label{subsec:aggregation-closed-forms}

The following theorem gives the coordinatewise closed forms. Throughout this section, the convention
\begin{equation}
\left(\sum_{i=1}^{n}\frac{w_i}{x_i}\right)^{-1}=0
\quad\text{whenever at least one }x_i=0
\label{eq:harmonic-convention}
\end{equation}
is used for weighted harmonic-type expressions.

\begin{theorem}[Closed forms of ALA-WAM and ALA-WGM]
\label{thm:closed-forms-wam-wgm}
Let \(A_i=(\theta_i,\sigma_i,\gamma_i)\in\cA\) and let \(\mathbf{w}\) satisfy Eq.~\eqref{eq:weight-vector}. Then
\begin{equation}
A_{\mathbf{w}}^{+}
=
\left(
1-\prod_{i=1}^{n}(1-\theta_i)^{w_i},
\;1-\prod_{i=1}^{n}(1-\sigma_i)^{w_i},
\;\left(\sum_{i=1}^{n}\frac{w_i}{\gamma_i}\right)^{-1}
\right),
\label{eq:closed-wam}
\end{equation}
and
\begin{equation}
A_{\mathbf{w}}^{-}
=
\left(
\left(\sum_{i=1}^{n}\frac{w_i}{\theta_i}\right)^{-1},
\;\left(\sum_{i=1}^{n}\frac{w_i}{\sigma_i}\right)^{-1},
\;1-\prod_{i=1}^{n}(1-\gamma_i)^{w_i}
\right).
\label{eq:closed-wgm}
\end{equation}
\end{theorem}

\begin{proof}
For ALA-WAM, Definition~\ref{def:ala-wam} gives
\[
A_{\mathbf{w}}^{+}=\bigoplus_{i=1}^{n}(w_i\odot A_i).
\]
Using Eq.~\eqref{eq:scalar-multiplication}, the valuation coordinate of \(w_i\odot A_i\) is \(1-(1-\theta_i)^{w_i}\). Applying the n-ary algebraic sum identity in Eq.~\eqref{eq:nary-S} gives
\[
1-\prod_{i=1}^{n}\left(1-[1-(1-\theta_i)^{w_i}]\right)
=
1-\prod_{i=1}^{n}(1-\theta_i)^{w_i}.
\]
The same argument gives the contextual coordinate in Eq.~\eqref{eq:closed-wam}.

For the restraint coordinate, the restraint component of \(w_i\odot A_i\) is
\[
y_i=\frac{\gamma_i}{w_i-(w_i-1)\gamma_i}.
\]
If every \(\gamma_i>0\), then
\[
\frac{1}{y_i}=1+w_i\left(\frac{1}{\gamma_i}-1\right).
\]
Using the n-ary Hamacher identity in Eq.~\eqref{eq:nary-H},
\[
\frac{1}{H_n(y_1,\ldots,y_n)}
=
\sum_{i=1}^{n}\frac{1}{y_i}-(n-1)
=
\sum_{i=1}^{n}\frac{w_i}{\gamma_i},
\]
where \(\sum_i w_i=1\) has been used. This proves the restraint coordinate of Eq.~\eqref{eq:closed-wam}; the zero case follows from the convention in Eq.~\eqref{eq:harmonic-convention} and the fact that Hamacher product has zero as absorbing element.

The proof of Eq.~\eqref{eq:closed-wgm} is dual. Apply Definition~\ref{def:ala-wgm}, use the power formula in Eq.~\eqref{eq:interpretive-power}, use Eq.~\eqref{eq:nary-H} for the valuation and contextual coordinates, and use Eq.~\eqref{eq:nary-S} for the restraint coordinate.
\end{proof}

Eqs.~\eqref{eq:closed-wam} and \eqref{eq:closed-wgm} show why the aggregation operators are not coordinate-mixing. The \(\theta\)-output is generated only from \(\theta\)-inputs, the \(\sigma\)-output only from \(\sigma\)-inputs, and the \(\gamma\)-output only from \(\gamma\)-inputs. This is the aggregation-level form of role preservation.

\begin{remark}[Naming and the Hamacher family]
\label{rem:hamacher-naming}
The abbreviations ALA-WAM and ALA-WGM are retained for continuity with weighted aggregation terminology, but the operators are not the coordinatewise arithmetic and geometric means. Their native construction is Hamacher based. In particular, the \(\theta\)- and \(\sigma\)-coordinates of ALA-WGM in Eq.~\eqref{eq:closed-wgm} are weighted harmonic means rather than products \(\prod_i\theta_i^{w_i}\). The full names ``Hamacher-type permissive mean'' and ``Hamacher-type conservative mean'' are therefore the mathematically more informative descriptions; the shorter acronyms serve only as compact notation \cite{LiuHamacher2014}.
\end{remark}

\begin{theorem}[Idempotency and closure]
\label{thm:aggregation-idempotency}
For every \(A\in\cA\),
\[
\operatorname{WAM}_{\cA,\mathbf{w}}(A,\ldots,A)=A,
\qquad
\operatorname{WGM}_{\cA,\mathbf{w}}(A,\ldots,A)=A.
\]
Moreover, \(A_{\mathbf{w}}^{+}\in\cA\) and \(A_{\mathbf{w}}^{-}\in\cA\).
\end{theorem}

\begin{proof}
Closure follows from Theorem~\ref{thm:closed-forms-wam-wgm}, since all coordinates remain in \([0,1]\). If \(A_i=A=(\theta,\sigma,\gamma)\) for all \(i\), then
\[
1-\prod_{i=1}^{n}(1-\theta)^{w_i}=1-(1-\theta)^{\sum_iw_i}=\theta,
\]
and similarly for \(\sigma\). Also,
\[
\left(\sum_{i=1}^{n}\frac{w_i}{\gamma}\right)^{-1}=\gamma
\]
whenever \(\gamma>0\), with the boundary case \(\gamma=0\) following from Eq.~\eqref{eq:harmonic-convention}. Thus ALA-WAM is idempotent. The proof for ALA-WGM is analogous.
\end{proof}

\begin{remark}[Boundary convention and \(\varepsilon\)-regularization]
\label{rem:boundary-aggregation}
The convention in Eq.~\eqref{eq:harmonic-convention} gives exact zeros veto-like or absorbing semantics: one input with \(\gamma_i=0\) forces the restraint coordinate of \(A^{+}_{\mathbf w}\) to zero, while one input with \(\theta_i=0\) or \(\sigma_i=0\) annihilates the corresponding coordinate of \(A^{-}_{\mathbf w}\), regardless of a small positive weight. This behavior is mathematically coherent when an exact boundary value is intended as a categorical release or veto.

When exact boundary values are measurement artifacts rather than categorical statements, applications should adopt and report an \(\varepsilon\)-regularization convention before aggregation:
\[
x_i^{(\varepsilon)}
=
\min\{1-\varepsilon,\max\{\varepsilon,x_i\}\},
\qquad
0<\varepsilon<\tfrac12.
\]
The closed forms and envelope theorem then apply to the regularized inputs. Thus the framework does not silently choose between absorbing boundary semantics and weight-sensitive interior behavior; the application must declare which interpretation is intended.
\end{remark}

\begin{theorem}[Cognitive monotonicity of aggregation]
\label{thm:aggregation-monotonicity}
Let \(A_i,B_i\in\cA\) for \(i=1,
\ldots,n\). If \(A_i\preceqA B_i\) for every \(i\), then
\[
\operatorname{WAM}_{\cA,\mathbf{w}}(A_1,\ldots,A_n)
\preceqA
\operatorname{WAM}_{\cA,\mathbf{w}}(B_1,\ldots,B_n),
\]
and
\[
\operatorname{WGM}_{\cA,\mathbf{w}}(A_1,\ldots,A_n)
\preceqA
\operatorname{WGM}_{\cA,\mathbf{w}}(B_1,\ldots,B_n).
\]
\end{theorem}

\begin{proof}
For ALA-WAM, the result follows from the cognitive monotonicity of scalar multiplication and interpretive addition, established in Theorem~\ref{thm:arithmetic-order-compatibility}. For ALA-WGM, it follows from the cognitive monotonicity of interpretive powers and interpretive multiplication. The finite-input case is obtained by repeated application of the binary monotonicity results.
\end{proof}

\subsection{The Aggregation Envelope}
\label{subsec:aggregation-envelope-theorem}

The permissive and conservative aggregates form an order interval in the cognitive order. This interval is the first formal expression of aggregation ambiguity in ALA.

\begin{theorem}[Aggregation envelope]
\label{thm:aggregation-envelope}
For every finite family \(A_1,\ldots,A_n\in\cA\) and every positive normalized weight vector \(\mathbf{w}\),
\[
A_{\mathbf{w}}^{-}\preceqA A_{\mathbf{w}}^{+}.
\]
Equivalently,
\begin{equation}
\theta_{\mathbf{w}}^{-}\le \theta_{\mathbf{w}}^{+},
\qquad
\sigma_{\mathbf{w}}^{-}\le \sigma_{\mathbf{w}}^{+},
\qquad
\gamma_{\mathbf{w}}^{-}\ge \gamma_{\mathbf{w}}^{+}.
\label{eq:aggregation-envelope-coordinates}
\end{equation}
\end{theorem}

\begin{proof}
It is enough to prove the scalar inequality
\begin{equation}
\left(\sum_{i=1}^{n}\frac{w_i}{x_i}\right)^{-1}
\le
1-\prod_{i=1}^{n}(1-x_i)^{w_i}
\label{eq:scalar-envelope-inequality}
\end{equation}
for \(x_i\in[0,1]\), with the convention in Eq.~\eqref{eq:harmonic-convention}. If some \(x_i=0\), the left-hand side is zero and the inequality is immediate. Assume therefore that all \(x_i>0\). The weighted harmonic mean is not larger than the weighted arithmetic mean:
\[
\left(\sum_{i=1}^{n}\frac{w_i}{x_i}\right)^{-1}
\le
\sum_{i=1}^{n}w_i x_i.
\]
By the weighted arithmetic-geometric mean inequality applied to \(1-x_i\),
\[
\prod_{i=1}^{n}(1-x_i)^{w_i}
\le
\sum_{i=1}^{n}w_i(1-x_i)
=
1-\sum_{i=1}^{n}w_i x_i.
\]
Hence
\[
\sum_{i=1}^{n}w_i x_i
\le
1-\prod_{i=1}^{n}(1-x_i)^{w_i}.
\]
Combining the two inequalities gives Eq.~\eqref{eq:scalar-envelope-inequality}.

Applying Eq.~\eqref{eq:scalar-envelope-inequality} to \(x_i=\theta_i\) and \(x_i=\sigma_i\) gives \(\theta_{\mathbf{w}}^{-}\le\theta_{\mathbf{w}}^{+}\) and \(\sigma_{\mathbf{w}}^{-}\le\sigma_{\mathbf{w}}^{+}\). Applying it to \(x_i=\gamma_i\) gives \(\gamma_{\mathbf{w}}^{+}\le\gamma_{\mathbf{w}}^{-}\). This is precisely Eq.~\eqref{eq:aggregation-envelope-coordinates}, which is equivalent to \(A_{\mathbf{w}}^{-}\preceqA A_{\mathbf{w}}^{+}\) because \(\gamma\) has reversed polarity in the cognitive order.
\end{proof}

The theorem does not rank the two aggregators. It locates the conservative aggregate cognitively below the permissive aggregate in the ALA order. The coordinatewise aggregation latitude is the nonnegative vector
\[
\Delta_{\mathbf w}
=
\bigl(
\theta_{\mathbf w}^{+}-\theta_{\mathbf w}^{-},
\;
\sigma_{\mathbf w}^{+}-\sigma_{\mathbf w}^{-},
\;
\gamma_{\mathbf w}^{-}-\gamma_{\mathbf w}^{+}
\bigr).
\]
It records, role by role, the spread between the conservative and permissive endpoints. A scalar norm of \(\Delta_{\mathbf w}\) may be introduced by an application, but no such norm is primitive in the present paper.

\begin{example}[A numerical aggregation envelope]
\label{ex:aggregation-envelope}
Let \(A_1=(0.80,0.60,0.20)\) and \(A_2=(0.50,0.90,0.60)\), with \(\mathbf{w}=(0.6,0.4)\). Theorem~\ref{thm:closed-forms-wam-wgm} gives, to four decimals,
\[
A^{+}_{\mathbf{w}}=(0.7115,\;0.7703,\;0.2727),
\qquad
A^{-}_{\mathbf{w}}=(0.6452,\;0.6923,\;0.3937).
\]
The envelope inequality is visible coordinatewise: the conservative aggregate carries lower valuation, lower contextual alignment, and stronger restraint. Here
\[
\Delta_{\mathbf w}
=
(0.0663,\;0.0780,\;0.1210),
\]
so the largest rolewise latitude occurs in restraint. A downstream interface may operate at a declared point or by a declared rule inside the order interval; that selection is a policy rather than a hidden consequence of the aggregation formula.
\end{example}

\begin{remark}[Rolewise interpolation within the envelope]
\label{rem:power-mean}
For a positively oriented scalar coordinate, the conservative endpoint is the weighted harmonic mean \(M_{-1}\), while the permissive endpoint dominates the weighted arithmetic mean \(M_1\). Since the weighted power mean
\[
M_p(x_1,\ldots,x_n;\mathbf{w})
=
\Bigl(\sum_{i=1}^{n}w_i x_i^{p}\Bigr)^{1/p},
\qquad
p\in[-1,1]\setminus\{0\},
\]
is nondecreasing in \(p\), with the geometric mean recovered as \(p\to0\), every \(M_p\) on this range lies between the scalar endpoints for \(\theta\) and \(\sigma\). For restraint, the same numerical family must be interpreted through the reversed cognitive polarity. Consequently, \(p\) is a rolewise interpolation parameter within the aggregation envelope, not automatically a single global ``conservatism dial'' for the entire ALA state. A globally ordered interpolation would require coordinate-specific parameterization consistent with the reversed restraint direction.
\end{remark}

\section{Operational Use and Clinical Stress Testing}

This section moves from the structural development of ALA to its operational interpretation. Its purpose is not to introduce a domain-specific decision rule, but to show how a triadic ALA state can be constructed, processed, and finally projected when an external action or report is required. The clinical stress test then illustrates why this operational separation matters in a high-stakes setting, where evidential support alone does not automatically authorize intervention.

\subsection{Operational Protocol}
\label{subsec:compact-operational-protocol}

Before turning to the clinical stress test, it is useful to state how ALA may be used operationally in an applied setting. A compact deployment proceeds in four stages. First, \emph{role elicitation} gathers domain inputs for valuation, contextual alignment, and epistemic restraint without premature aggregation. Second, \emph{state construction} maps these inputs to a triadic ALA state \(A\in\cA\), so that no role is already absorbed into another. Third, \emph{triadic processing} applies the cognitive order, lattice operations, arithmetic operations, aggregation operators, or other role-preserving updates at the level of the full state. Fourth, \emph{interface projection} may be applied when a scalar, linguistic, ranking, thresholding, or decision-ready readout is required downstream. The fourth stage is terminal and optional: it is a declared readout for a specific external purpose, not a replacement for the underlying triadic state.

The computational core of this protocol is lightweight once the ALA states and the selected state-level operators have been fixed. The basic operations developed in this paper are closed on \([0,1]^3\), coordinate-preserving, and require only elementary arithmetic and exponentiation. This tractability claim concerns the state-level operational kernel, not ALA as a whole; elicitation, interpretation, and policy selection remain domain-bound tasks whose complexity depends on the application.

\subsection{Clinical Stress Test}
\label{subsec:clinical-stress-test}

This case study illustrates the central operational distinction preserved by ALA: strong evidential support does not by itself imply permission for immediate operational commitment. The example is not intended as clinical advice. It is a methodological scenario designed to show why a membership-like judgment may need to preserve valuation, contextual alignment, and epistemic restraint before scalar projection.

This case gives operational content to the introductory principle that knowing is not permission to intervene, and uncertainty is not permission to delay. A judgment may have strong evidential support, strong contextual alignment, and high auxiliary confidence, while still being restrained from direct operational use. Conversely, the absence of permission for immediate invasive intervention does not by itself license routine delay. The issue is not whether the judgment is supported. The issue is whether that support is permitted to drive a particular form of operational commitment.

The case is built on a single contrast: the same support, carried through different commitment interfaces.

For clarity, the assessments in this section are stated as role-factorized ALA specifications in the sense of Definition~\ref{def:role-factorized}, so that each numerical coordinate can be read directly from its most direct semantic arguments; nothing in the analysis depends on this choice rather than on a general assessment producing the same states.

Consider an AI-assisted clinical assessment of a suspicious lesion \cite{Amann2020}, where the preserved judgment concerns diagnostic suspiciousness alone. The evidential basis is strong: the lesion is highly compatible with the target predicate in terms of shape, boundary, texture, and growth pattern. The interpretive conditions are also strongly aligned: the imaging quality is high, the device is calibrated, the clinical information is adequate, and the case lies within the intended domain of the diagnostic model. Suppose also that an auxiliary confidence score associated with the diagnostic evidence is high:
\[
\kappa=0.96.
\]
The score \(\kappa\) is not an ALA coordinate. It is an auxiliary elicitation input with boundary scope in the sense of Section~3.4, recorded here only to make explicit that the contrast below is not driven by any loss of confidence.
In ALA, the support-side coordinates are therefore assigned as
\[
\theta=0.94,
\qquad
\sigma=0.90.
\]

The case is constructed so that the evidential basis, contextual alignment, and confidence level remain fixed, while the operational commitment interface changes. The same preserved clinical judgment is first used in a lower-commitment, comparatively reversible follow-up interface and then in a high-stakes, invasive intervention interface. The point is not that the evidence changes, nor that its contextual adequacy changes. The point is that the permission boundary for operationalizing the same preserved judgment changes.

The two ALA states are accordingly assigned the same valuation and contextual-alignment coordinates:
\[
\theta_1=\theta_2=0.94,
\qquad
\sigma_1=\sigma_2=0.90.
\]
They differ only in epistemic restraint:
\[
A_F=(0.94,0.90,0.15),
\qquad
A_I=(0.94,0.90,0.95).
\]
The first state represents the use of the judgment in a lower-commitment follow-up interface. The second represents the use of the same preserved judgment in a high-stakes intervention interface. The difference is therefore not evidential and not contextual in the sense of contextual alignment. It is a difference in the boundary between preserved judgment and operational commitment. Formally, the two commitment interfaces correspond to two perspectives \(p_{F},p_{I}\in\mathcal{P}\) in the sense of Definition~\ref{def:ala-valued}: a perspective includes the operational policy under which the judgment is to be used, so the change of interface is a change of \(p\) with the object \(x\) and the context \(c\) held fixed.


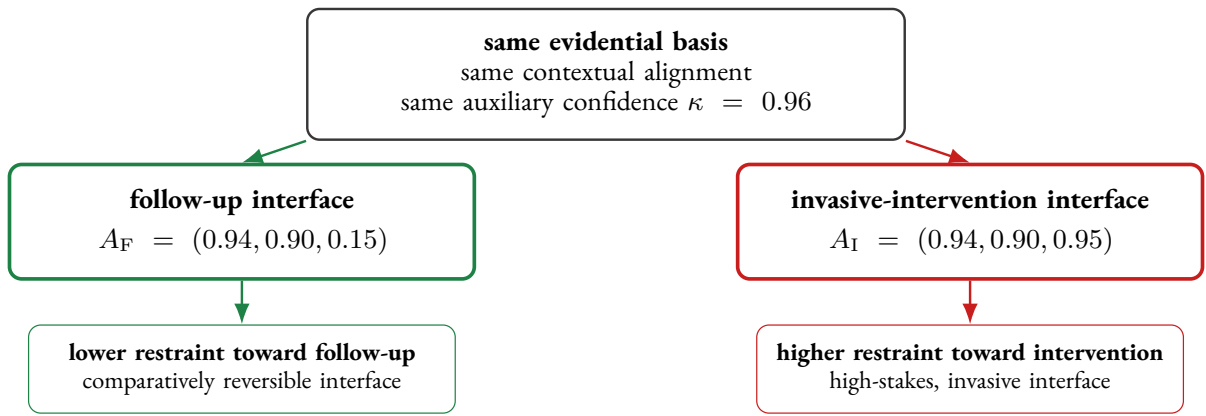
\begin{figure}[!htbp]
\centering
\resizebox{0.86\textwidth}{!}{%
\begin{tikzpicture}[
    font=\small,
    >=Latex,
    shared/.style={
        draw=Ccube,
        thick,
        rounded corners,
        align=center,
        inner sep=7pt,
        text width=6.4cm
    },
    card/.style={
        draw,
        rounded corners,
        align=center,
        inner sep=8pt,
        minimum height=1.25cm,
        text width=4.8cm
    },
    status/.style={
        rounded corners,
        align=center,
        inner sep=6pt,
        text width=4.5cm,
        font=\footnotesize
    },
    arr/.style={-{Latex[length=2.4mm]}, thick}
]

\node[shared] (shared) at (0,2.25)
{\textbf{same evidential basis}\\[-1pt]
same contextual alignment\\[-1pt]
same auxiliary confidence \(\kappa=0.96\)};

\node[card, draw=Cok, very thick] (follow) at (-4.2,0.55)
{\textbf{follow-up interface}\\[2pt]
\(A_{\mathrm F}=(0.94,0.90,0.15)\)};

\node[card, draw=Cbad, very thick] (intervention) at (4.2,0.55)
{\textbf{invasive-intervention interface}\\[2pt]
\(A_{\mathrm I}=(0.94,0.90,0.95)\)};

\node[status, draw=Cok] (followstatus) at (-4.2,-1.15)
{\textbf{lower restraint toward follow-up}\\[-1pt]
comparatively reversible interface};

\node[status, draw=Cbad] (intervstatus) at (4.2,-1.15)
{\textbf{higher restraint toward intervention}\\[-1pt]
high-stakes, invasive interface};

\draw[arr, Cok] (shared.south west) -- (follow.north);
\draw[arr, Cbad] (shared.south east) -- (intervention.north);

\draw[arr, Cok] (follow.south) -- (followstatus.north);
\draw[arr, Cbad] (intervention.south) -- (intervstatus.north);

\end{tikzpicture}%
}

\caption{Support and permission under different commitment interfaces.}
\label{fig:case-study-permission}
\end{figure}
Fig.~\ref{fig:case-study-permission} gives operational content to the case-study principle: same support, different commitment interface. The evidential basis, contextual alignment, and auxiliary confidence are held fixed, while the operational interface changes from lower-commitment follow-up to high-stakes invasive intervention. The distinction is represented through \(\gamma\), not through a change in \(\theta\) or \(\sigma\).

In this case, the evidential and interpretive context relevant to valuation and contextual alignment remains fixed, while the commitment interface changes; the change is carried entirely by epistemic restraint.

\begin{remark}[Intervention and delay]
The preceding distinction should not be read as a preference for delay. 
Withholding immediate invasive intervention does not automatically 
license routine delay, because non-intervention is also an operational 
commitment and must itself be justified. Thus, ALA does not prescribe 
inaction; it requires both intervention and non-intervention to pass 
through an appropriate commitment interface. In this sense, the case 
operationalizes the principle stated above: supported knowing is not 
permission to intervene, and unresolved uncertainty is not permission 
to delay.
\end{remark}

\subsection{Representation and Projection Lessons}
\label{subsec:case-semantic-reassignment}

A standard fuzzy representation would assign the same membership value, for example
\[
\mu=0.94,
\]
to both uses of the judgment. This scalar value preserves the strength of the lesion-related assessment, but it does not preserve the difference between a lower-commitment follow-up recommendation and a high-stakes intervention commitment.

An intuitionistic or neutrosophic representation could introduce non-membership, hesitation, or indeterminacy. However, the present example is not primarily a case of high non-membership, weak evidence, or unresolved truth-status. The lesion remains strongly supported by the evidence. If hesitation or indeterminacy were artificially increased merely to prevent the invasive intervention, the semantics of these coordinates would be reassigned from truth-status or information-status to action-permission.

A Z-number representation could attach a high reliability score to the assessment, for example
\[
Z=(0.94,0.96).
\]
If reliability is kept high, as the evidence warrants, the representation still does not record the difference between the two operational commitment interfaces. If reliability is lowered merely to block the high-stakes intervention, then an action-permission constraint is absorbed into a reliability coordinate, although the evidential reliability itself has not changed.

Such frameworks can emulate the behavior by adding an external policy layer, a utility function, or an application-specific decision rule. What they do not do is preserve the distinction between evidential support and permission natively, inside the membership-like state itself. ALA represents this distinction internally.

\begin{proposition}[Role-accounting limitation of fixed bivariate support--quality representations]
\label{prop:bivariate_insufficiency}
Consider a bivariate representation \(\mathcal{F}(s)=(u(s),v(s))\) whose native roles are fixed: \(u\) represents predicate-directed support and \(v\) represents evidential reliability, hesitation, or contextual adequacy. Suppose two scenarios \(s_1\) and \(s_2\) share the same evidential basis, predicate support, and contextual adequacy, but differ in their permission boundary for operational commitment, for example follow-up versus invasive surgery.

If the representation preserves its native semantics, then it assigns \(\mathcal{F}(s_1)=\mathcal{F}(s_2)\). Consequently, it cannot internally distinguish the two operational uses. If it distinguishes them while retaining only the two fixed native coordinates, then at least one coordinate has been used to encode epistemic restraint rather than its stated role. Thus, this limitation applies to fixed bivariate support--quality representations whose two coordinates are not themselves defined to include restraint. It does not preclude an external policy layer, nor does it address a bivariate model whose second coordinate is explicitly redefined as restraint.
\end{proposition}

\begin{proof}
By assumption, the predicate-directed support and the evidential or contextual quality are identical across \(s_1\) and \(s_2\). Role preservation therefore gives \(u(s_1)=u(s_2)\) and \(v(s_1)=v(s_2)\). Hence \(\mathcal{F}(s_1)=\mathcal{F}(s_2)\). Any distinction between the two operational uses inside the same bivariate representation requires one of the coordinates to vary despite no change in its native role. The distinction is therefore obtained either by an external policy layer or by using one native coordinate for an operational role that it does not normally carry. ALA avoids this reassignment by preserving epistemic restraint as a primitive coordinate.
\end{proof}

\begin{table}[!htbp]
\centering
\caption{Representations of the support--permission contrast across frameworks.}
\label{tab:case-framework-comparison}
\fontsize{10}{12}\selectfont
\setlength{\tabcolsep}{3pt}
\renewcommand{\arraystretch}{1.14}
\begin{tabular}{>{\raggedright\arraybackslash}p{0.18\textwidth}>{\raggedright\arraybackslash}p{0.34\textwidth}>{\raggedright\arraybackslash}p{0.38\textwidth}}
\toprule
Framework & Native representation & Role-level limitation \\
\midrule
Scalar fuzzy set & \(\mu=0.94\) for both interfaces & Preserves valuation strength but not the permission boundary. \\
Intuitionistic or neutrosophic form & May introduce non-membership, hesitation, or indeterminacy & The case is not primarily weak membership or unresolved truth-status. Their native roles do not distinguish operational permission unless an additional policy layer or role reassignment is introduced. \\
Z-number & \(Z=(0.94,0.96)\), if reliability is kept honest & Reliability remains high in both interfaces; lowering it to block intervention absorbs action-permission into reliability. \\
ALA & \(A_1=(0.94,0.90,0.15)\), \(A_2=(0.94,0.90,0.95)\) & Preserves the same support while distinguishing the operational commitment interface. \\
\bottomrule
\end{tabular}
\end{table}
To display the distinction computationally, consider the following
order-aligned, restraint-sensitive family of scalar interfaces:
\[
\phi_\alpha(A)
=
(1-\alpha)
\bigl(
\eta_\theta\theta_A+\eta_\sigma\sigma_A
\bigr)
+
\alpha(1-\gamma_A),
\]
where
\[
\alpha\in[0,1],
\qquad
\eta_\theta,\eta_\sigma\geq 0,
\qquad
\eta_\theta+\eta_\sigma=1.
\]
Here, \(\alpha\) controls the importance assigned to released
permissiveness \(1-\gamma_A\), while
\(\eta_\theta\) and \(\eta_\sigma\) distribute the remaining weight
between evidential valuation and contextual alignment.

\begin{remark}[Calibration from a commitment threshold]
The restraint weight need not be selected as an arbitrary numerical
preference. Let
\[
A_H=(\theta_H,\sigma_H,\gamma_H)
\]
be a reference state that should not pass a specified high-stakes
commitment threshold \(\tau\). Define
\[
S_H
=
\eta_\theta\theta_H+\eta_\sigma\sigma_H,
\qquad
R_H=1-\gamma_H.
\]
Then
\[
\phi_\alpha(A_H)
=
S_H-\alpha(S_H-R_H).
\]
If \(S_H>R_H\) and
\[
R_H\leq\tau\leq S_H,
\]
the requirement
\[
\phi_\alpha(A_H)\leq\tau
\]
is equivalent to
\[
\alpha
\geq
\frac{S_H-\tau}{S_H-R_H}.
\]
Thus, conditional on an explicitly declared commitment threshold,
the restraint weight can be calibrated rather than chosen
post hoc. The threshold itself remains domain- and policy-dependent.
In deployment, the reference state, commitment threshold, and relative
role weights should be prespecified or calibrated on an independent
reference set rather than selected from the same case that is being
evaluated.
\end{remark}

For the present illustration, take the high-restraint intervention
state
\[
A_I=(0.94,0.90,0.95),
\]
assign equal relative importance to valuation and contextual
alignment,
\[
\eta_\theta=\eta_\sigma=0.50,
\]
and set the high-stakes commitment threshold to
\[
\tau=0.40.
\]
Then
\[
S_H
=
0.50(0.94)+0.50(0.90)
=
0.92,
\qquad
R_H
=
1-0.95
=
0.05,
\]
and therefore
\[
\alpha
\geq
\frac{0.92-0.40}{0.92-0.05}
=
0.5977\ldots.
\]
Choosing \(\alpha=0.60\) yields
\[
(1-\alpha)\eta_\theta
=
(1-\alpha)\eta_\sigma
=
0.20,
\]
and hence the calibrated interface
\[
\phi_w(A)
=
0.20\theta_A
+
0.20\sigma_A
+
0.60(1-\gamma_A).
\]

For the lower-commitment, comparatively reversible follow-up
interface,
\[
A_F=(0.94,0.90,0.15),
\]
one obtains
\[
\phi_w(A_F)
=
0.20(0.94)
+
0.20(0.90)
+
0.60(1-0.15)
=
0.878.
\]
For the high-stakes intervention interface,
\[
A_I=(0.94,0.90,0.95),
\]
one obtains
\[
\phi_w(A_I)
=
0.20(0.94)
+
0.20(0.90)
+
0.60(1-0.95)
=
0.398.
\]

The interpretive valuation, contextual alignment, and auxiliary
confidence level are unchanged across the two interfaces. The scalar
outputs differ solely because the intended commitment interface, and
therefore the associated epistemic restraint, differs. The values
\(0.878\) and \(0.398\) are policy-indexed interface scores, not
probabilities of disease, treatment success, or action permissibility.

ALA does not encode this restriction by weakening the valuation,
degrading contextual alignment, or understating confidence. It
preserves the strong support and records the restriction on premature
operational commitment in a distinct coordinate. This is the
operational content of reasoning before projection:
\[
\text{same support, different commitment interface.}
\]

These scalar values do not themselves prescribe an action. They are
outputs of an explicitly declared interface policy. ALA preserves the
triadic state for comparison, operation, and aggregation, and permits
a scalar readout only afterward. The resulting representation retains
the role-typed source of the difference and is therefore more
diagnostically informative and auditable than a scalar report that
records support alone.


\begin{remark}[Analytical status of the clinical stress test]
\label{rem:case-study-status}
The clinical scenario is a methodological stress test, not an empirical study or a medical recommendation. The structural phenomenon it illustrates was established analytically in Section~\ref{sec:projection-equivalence}: projection-induced equivalence is not, in general, a congruence for the ALA lattice operations. The scenario shows how that proved distinction becomes visible at a high-stakes commitment interface. Empirical studies may later quantify the frequency or magnitude of projection collapse under declared sampling models, but such studies are not required for the existence result proved here.
\end{remark}

\section{Discussion}
\label{sec:discussion}

ALA is a structure-preserving calculus for ambiguous membership-like judgments, and its central distinction is not numerical. ALA separates the value assigned to a predicate, the contextual alignment under which that value is formed, and the restraint governing whether the judgment should be operationalized. This separation prevents the scalar interface from becoming the primitive representation.

Several boundaries follow from this architecture. Context-dependent membership treats context as a condition shaping a reported grade, whereas ALA preserves contextual alignment as part of the output state. Reliability concerns the trustworthiness of information; \(\sigma\) concerns the adequacy of the interpretive setting for the judgment. Confidence strengthens assertion; epistemic restraint limits premature operational commitment. The order-aligned form of restraint is therefore \(1-\gamma\), not \(\gamma\) itself.

The normalized cube \([0,1]^3\) should also not be misread as raw commensurability. The three coordinates share a bounded carrier for order-theoretic and computational reasons, but they do not share the same semantic substance. Projection weights are therefore not arbitrary decorations. They are declared interfaces for a particular task, such as reporting, ranking, screening, or safety-critical decision support.

ALA is operational as well as philosophical. It makes the origin of a scalar grade more accountable by requiring the evaluator or system to preserve the conditions under which the value was produced. A grade such as \(0.80\) is no longer an isolated number. In ALA, it may be accompanied by moderate contextual alignment and strong restraint, indicating that the value should be interpreted cautiously. This is especially important in education, medicine, peer review, social evaluation, and AI-assisted decision systems, where scalar outputs often hide the conditions of their own production.

The same point appears outside clinical reasoning. A scalar grade assigned to a student, for example \(0.80\), may appear to say only that the student belongs strongly to the class of good performers. Yet the grade may have been produced under online instruction, an examination submitted from home, uneven supervision, or a possible evaluator relationship. In ALA, a state such as \(A_{\mathrm{grade}}=(0.80,0.50,0.60)\) does not accuse the evaluator of bias; it records strong valuation, moderate contextual alignment, and non-negligible restraint, and prevents the number from appearing anonymous. The evaluator is placed in a transparent representational setting: the value is visible, but so are the conditions under which the value was produced.

Two further observations place the contribution in perspective. The first concerns the organizing principle. The deepest commitment of ALA is not the triad as such but the discipline that representation must precede projection: a scalar grade is a late, policy-laden, and lossy interface, and reasoning should operate on the role-typed state before that interface is imposed. Stated this way, the principle reaches beyond fuzzy membership, since many systems compress structured judgment into a single number, including credit and risk scores, recommender rankings, and the confidence outputs of automated assessment systems. The natural domain of application is correspondingly the class of high-stakes or irreversible decisions, in medicine, safety-critical autonomy, legal assessment, and content moderation, where the cost of collapsing support into permission is asymmetric. The reversed-polarity restraint coordinate is built for precisely such asymmetric-risk settings, and this is the application class that motivates the framework rather than a single illustrative example.

The clearest philosophical antecedent is the pragmatist analysis of acceptance developed by Isaac Levi: strong evidential support does not by itself compel acceptance into a corpus of full belief, because commitment carries informational value and risk of error \cite{LeviGamblingTruth1967,LeviEnterprise1980}. ALA addresses a related but distinct boundary. Even when a judgment is retained as well supported, its operational use may remain restrained. This lineage motivates the support--permission distinction without serving as a proof of the ALA formalism; the mathematical claims of the paper rest on the definitions, order, and theorems developed above. Rumi's elephant in the dark may be read only as a limited illustration of contextual dependence: disagreement can reflect an impoverished interpretive setting rather than a need to average incompatible verdicts \cite{RumiMathnawi}.

The second observation concerns the boundaries of the present contribution, which are best stated plainly. ALA is a representation layer, not a decision procedure: it preserves the roles of a judgment and supports operations on them, but it does not prescribe which action to take. It does not detect dishonesty, bias, or conflict of interest, and no numerical representation can guarantee truthful elicitation; the framework only ensures that the conditions under which a value is used are exposed rather than hidden. It does not supply a universal rule for obtaining \(\theta\), \(\sigma\), and \(\gamma\) in practice. Section~\ref{subsec:elicitation} provides an admissibility discipline and one illustrative route, but domain-specific identification, validation, and sensitivity analysis remain open tasks. These are limits of scope rather than defects.

Thus, the present paper is foundational, and the program it opens is broader than the results reported here. It fixes the normalized real-valued realization, the cognitive order, the projection discipline, the non-collapse result, the arithmetic core, the aggregation envelope, and the support--permission case. Subsequent work may develop an axiomatic theory of ALA elicitation, distance and similarity measures, traceability profiles, ALA-based decision making, information measures, and connections with probability and optimization, together with dynamics in which evidence moves valuation and contextual alignment while restraint responds to stakes rather than to evidence.

The crisp--fuzzy distinction concerns the representation selected inside a role; ALA concerns the architecture that separates the roles. These are orthogonal modeling questions. The typed product \(L_\theta\times L_\sigma\times L_\gamma^{\mathrm{op}}\) therefore admits crisp, graded, structured, and mixed-carrier realizations. Even a fully crisp profile may preserve a nontrivial separation between a categorical valuation, categorical contextual admissibility, and a categorical prohibition on operational use. ALA is consequently not exhausted by the representation of vagueness or uncertainty. It is more broadly a role-typed calculus of evaluative judgment, of which fuzzy and uncertainty-oriented applications form an important but non-exclusive class.

The full carrier-independent theory remains outside the present scope. It requires, for each proposed role carrier, declared native orders and bounds, role-preserving operations, standardization maps, and proofs of compatibility with the global ALA policy. Questions of commutation with native operations, generalized projection, non-collapse, and aggregation envelopes must be resolved carrier by carrier. The unit-cube calculus developed here is the first complete realization of ALA and a standard interface for such extensions; it is not a proof that every structured uncertainty formalism already supplies an admissible ALA carrier.

\section{Conclusion}
\label{sec:conclusion}

This paper introduced the Awareness Logic of Ambiguity (ALA) as a formal framework for preserving the internal structure of membership-like judgment before scalar projection. The starting point was the support--permission separation: a judgment may be strongly supported, contextually aligned, and held with high confidence, while still being withheld from immediate operational commitment. This distinction is not a peripheral decision-theoretic concern, but a structural feature of reasoning under ambiguity. ALA formalizes it by representing an evaluative act as a triadic interpretive state whose coordinates preserve valuation, contextual alignment, and epistemic restraint as distinct roles.

The formal development showed that these roles are not merely three numerical decorations attached to a scalar membership grade. They determine the geometry and algebra of the framework. The reversed polarity of epistemic restraint induces a cognitive order in which stronger valuation and stronger contextual alignment increase permissiveness, while stronger restraint decreases it. This order yields a bounded distributive lattice, a role-preserving complement, and admissible scalar projections that are delayed, declared, and order-consistent. The same order-theoretic skeleton extends to the typed product \(L_\theta\times L_\sigma\times L_\gamma^{\mathrm{op}}\) whenever the native carriers are bounded distributive lattices, while the unit cube supplies the complete normalized realization developed here. The non-collapse theorem establishes the central structural result: equality after projection does not imply equality of ALA states, and projection-induced scalar equality may fracture under an ALA lattice operation.

The paper also justified the conditional semantic minimality of the triad under role-transparency and independent-variability assumptions. Removing valuation, contextual alignment, or epistemic restraint collapses a distinct role of judgment. A scalar grade cannot, by itself, preserve whether weakness comes from low predicate support, weak contextual alignment, or high restraint toward operational commitment. The arithmetic and aggregation layers developed here therefore act directly on ALA states rather than on already-compressed scalar outputs. ALA thereby relocates scalar membership to its proper place: an interface imposed after the semantic source of the judgment has been preserved.

The clinical stress test illustrates the operational meaning of this architecture. The same evidential basis, contextual alignment, and auxiliary confidence can be represented with lower restraint toward a comparatively reversible follow-up interface and higher restraint toward an invasive intervention interface. The difference is not a change in the diagnostic predicate or a reduction in confidence; it is a change in the operational commitment interface, captured by epistemic restraint. Any final permission or prohibition remains a downstream policy conclusion. The principle ``knowing is not permission'' thereby acquires formal content: supported judgment and permission to act are related, but they are not identical.

The present paper is foundational rather than exhaustive. Supplementary diagnostics on role-preserving admissible elicitation, projection behavior, and coordinate deletion are collected in Appendix~\hyperref[app:diagnostic-details]{A}. Several extensions remain open. Future work may develop distances, similarities, projection-ambiguity measures, traceability profiles, and order-violation diagnostics for comparing ALA states before and after projection, together with residual operations, information measures, and application-specific fusion rules. Future work may also develop ALA-based decision making, probability, statistics, optimization, and AI-assisted interpretive systems.

A dedicated line of research is opened by the carrier-independent architecture. Future work may construct role-local calculi for specific bounded ordered carriers, characterize admissible rolewise standardizations \(\mathcal R\), distinguish standardization loss from cross-role projection loss, determine when native processing commutes with standardization, and establish carrier-specific preservation, aggregation-envelope, and non-collapse results. Each extension must remain role preserving: internal qualifications may enrich a role value without being identified with a different ALA role, and the global permissive--conservative policy must be proved rather than assumed for every native calculus and declared interface.

The central contribution of ALA is therefore not the addition of more numerical components to fuzzy membership. Nor is ALA confined to uncertain or fuzzy judgments: crispness, fuzziness, interval structure, reliability qualification, and other representational forms belong inside the selected role carriers, whereas ALA governs the architecture that keeps the roles distinct. It preserves what makes a judgment interpretable before that judgment is standardized and then compressed into a scalar report. In environments where context, responsibility, and operational consequence matter, scalar agreement is not the end of meaning. The unit-cube calculus is the first complete realization of this architecture, not its final boundary.

\section*{Acknowledgments}

The author thanks the anonymous reviewers for their constructive comments, which improved the clarity and presentation of this work, and the colleagues and friends who discussed its early conceptual ideas and encouraged its further development.

Above all, this work is lovingly dedicated to the author's daughter, whose name inspired the acronym of this framework, and to her mother, whose presence quietly inhabits its motivating example. ALA is presented here as a formal scientific framework, a theory of why a judgment should not be reduced to a single number before its meaning is understood. Its origin, however, is simpler and more personal: the framework was shaped not only by mathematical curiosity, but by a deeply human appreciation for awareness, interpretation, and meaning, the very qualities that, for its author, were never scalars to begin with.

\section*{Author Contributions}
\textbf{S. A. Edalatpanah:}
Conceptualization; Methodology; Formal analysis;
Investigation; Writing -- original draft;
Writing -- review \& editing.

\section*{Funding}
This research did not receive any specific grant from funding agencies in the public, commercial, or not-for-profit sectors.

\section*{Data Availability}
No datasets were generated or analyzed during the current study.

\section*{Conflicts of Interest}
The author declares a potential non-financial competing interest due to his editorial and/or managerial affiliation with the journal and its publishing ecosystem. Given this role, the author had no involvement in the peer-review of this article and had no access to information regarding its peer review. Full responsibility for the editorial process, including reviewer selection, editorial evaluation, revision assessment, and the final publication decision, was delegated to an independent journal editor. The author declares no relevant financial competing interests.

\section*{Use of AI Tools}
During the multi-year preparation of this work, the author used several generative artificial-intelligence tools, in the form of large-language-model assistants, to improve the readability, fluency, and linguistic accuracy of the manuscript and to assist in checking the internal consistency of definitions, statements, and cross-references. All conceptual contributions, formal results, and proofs are the author's own. After using these tools, the author carefully reviewed and edited the content and takes full responsibility for the accuracy, integrity, originality, and final content of the publication.

\begingroup
\fontfamily{ppl}\fontsize{9}{10.35}\selectfont
\raggedright

\endgroup

\phantomsection
\section*{Appendix A. Additional diagnostic details}
\label{app:diagnostic-details}

\setcounter{equation}{0}
\setcounter{table}{0}
\renewcommand{\theequation}{A\arabic{equation}}
\renewcommand{\thetable}{A\arabic{table}}
\renewcommand{\theHequation}{appendix.A.\arabic{equation}}
\renewcommand{\theHtable}{appendix.A.\arabic{table}}

This appendix records three supplementary diagnostics that support the main article but are not needed for the core flow of the paper. They are included to clarify how role separation may be checked, how projection can be diagnosed, and what is lost when one coordinate is deleted.

\subsection*{A.1 Role-separated elicitation tests}

ALA treats elicitation as admissible only when it respects the semantic roles of the three coordinates. The following diagnostic principle is not an additional coordinate definition but a practical test for whether contextual weakness has been silently absorbed into valuation.

Holding the content-directed evidence profile, concept, and evaluator perspective fixed, a degradation attributable solely to the adequacy of the interpretive conditions should not be silently recoded as a change in primitive valuation. It should primarily lower contextual alignment:
\begin{equation}
\label{eq:app-context-perturbation}
\begin{gathered}
 z_A^\theta,A,p \ \hbox{fixed},\qquad c' \ \hbox{a degraded interpretive context relative to } c,\\
 \theta_A \ \hbox{unchanged under the stipulated perturbation},\qquad
 \sigma_A(x,c',p) \leq \sigma_A(x,c,p).
\end{gathered}
\end{equation}
This is a conditional diagnostic test, not a claim that every observed degradation leaves every empirically estimated valuation unchanged. If the degraded channel also changes the content-directed evidence available to the evaluator, then both \(\theta\) and \(\sigma\) may legitimately change. Conversely, if the interpretive channel is held fixed while the object--concept match is weakened, the valuation coordinate should change. The test blocks only the circular move in which a purely contextual inadequacy is hidden inside \(\theta\) and then used to argue that \(\sigma\) is redundant.

Thus, a clean and a degraded imaging channel may differ primarily in contextual alignment when the content-directed profile is stipulated as fixed; when the degradation alters that profile as well, the general dependence \(\theta_A(x,c,p)\) should be used rather than forcing invariance.

\subsection*{A.2 Projection diagnostics}

A single scalar projection is intentionally lossy, but a declared family of projections can test whether two ALA states are distinguishable through all admissible scalar interfaces of a given class. For the positive linear order-aligned family, define
\begin{equation}
\label{eq:app-positive-linear-family}
\Phi_{\mathrm{lin}}
=
\left\{
\phi_w(A)=w_\theta\theta_A+w_\sigma\sigma_A+w_\rho(1-\gamma_A)
:\
(w_\theta,w_\sigma,w_\rho)\in\Delta^2_{+}
\right\},
\end{equation}
where \(\Delta^2_{+}\) is the relative interior of the probability simplex. If two states have the same value under every projection in \(\Phi_{\mathrm{lin}}\), then they are identical as ALA states. Indeed, equality for all positive weights gives
\begin{equation}
\label{eq:app-linear-separation}
 w_\theta(\theta_A-\theta_B)
+w_\sigma(\sigma_A-\sigma_B)
+w_\rho\bigl((1-\gamma_A)-(1-\gamma_B)\bigr)=0
\end{equation}
for every \(w\in\Delta^2_{+}\). Perturbing the balanced weight vector inside the simplex forces all three coordinate differences in Eq.~\eqref{eq:app-linear-separation} to vanish. Thus, universal indiscernibility under this family is equivalent to exact equality of ALA states.

For a fixed pair of states, the positive linear projections that equalize them are given by
\begin{equation}
\label{eq:app-equalizing-set}
\Omega(A,B)=
\left\{
 w\in\Delta^2_{+}:
 w_\theta\Delta_\theta+w_\sigma\Delta_\sigma+w_\rho\Delta_\rho=0
\right\},
\end{equation}
where \(\Delta_\theta=\theta_A-\theta_B\), \(\Delta_\sigma=\sigma_A-\sigma_B\), and \(\Delta_\rho=(1-\gamma_A)-(1-\gamma_B)\). This set is a diagnostic object, not a ranking rule. For distinct states, \(\Omega(A,B)\) is nonempty exactly when the pair is incomparable in the cognitive order. If one state strictly dominates the other, every positive order-aligned projection preserves the strict inequality and no equalizing weight is possible.

\subsection*{A.3 Coordinate deletion and role collapse}

Deleting a coordinate does not merely reduce numerical resolution. It deletes a typed role from the evaluative act. Table~\ref{tab:app-coordinate-deletion} summarizes the corresponding collapses.

\begin{center}
\captionsetup{type=table,hypcap=false}
\caption{Coordinate deletion and role collapse.}
\label{tab:app-coordinate-deletion}
\fontsize{9}{10.8}\selectfont
\renewcommand{\arraystretch}{1.0}
\setlength{\tabcolsep}{4pt}
\begin{tabular}{>{\raggedright\arraybackslash}p{2.3cm}>{\raggedright\arraybackslash}p{2.7cm}>{\raggedright\arraybackslash}p{3.0cm}>{\raggedright\arraybackslash}p{6.0cm}}
\hline
\textbf{Deleted coordinate} &
\textbf{Remaining representation} &
\textbf{Collapsed role} &
\textbf{Consequence} \\
\hline
\(\theta\) & \((\sigma,\gamma)\) & Interpretive valuation & The framework can record contextual fit and restraint, but not the degree of predicate-directed valuation being assessed. \\
\(\sigma\) & \((\theta,\gamma)\) & Contextual alignment & Judgments made under well-aligned and poorly aligned interpretive conditions may become indistinguishable. \\
\(\gamma\) & \((\theta,\sigma)\) & Epistemic restraint & Lower-commitment and higher-commitment interfaces with the same support profile may collapse into the same state. \\
Scalar projection \(\phi(A)\) & One reported scalar & At least two role directions & Distinct triadic states may become projection-equivalent even though later operations can separate them. \\
\hline
\end{tabular}
\end{center}

The table explains why a pure valuation readout is not a strictly admissible projection of the whole ALA state. It may be reported descriptively, but it deletes contextual alignment and epistemic restraint at the scalar interface. ALA instead keeps the three roles visible until the projection or decision interface has been declared.

\end{document}